\documentclass[12pt]{article}
\usepackage{amsfonts}
\usepackage{amsthm}
\usepackage{CJK}
\usepackage{amssymb}
\usepackage{amsmath}    
\usepackage{graphicx,floatrow}   
\usepackage{verbatim}   
\usepackage{color}      
\usepackage{subfigure}  
\usepackage{epsfig}
\usepackage{subfigure}
\usepackage{float}
\usepackage{epstopdf}
\usepackage{cases}
\usepackage{appendix}
\usepackage{bm}
\usepackage{lineno}

\allowdisplaybreaks

\newtheorem{theorem}{Theorem}[section]
\newtheorem{lemma}[theorem]{Lemma}
\newtheorem{definition}{Definition}[section]
\newtheorem{prop}[theorem]{Proposition}

\theoremstyle{definition}
\newtheorem{remark}{Remark}[section]
\newtheorem{example}[remark]{Example}
\numberwithin{equation}{section}

\newcommand{\norm}[1]{\left\Vert#1\right\Vert}
\newcommand{\abs}[1]{\left\vert#1\right\vert}

\newcommand{\average}[1]{\ensuremath{\{\!\!\{#1\}\!\!\}} }
\newcommand{\jump}[1]{\ensuremath{[\![#1]\!]} }

\begin{document}


\title{Dissipative and conservative local discontinuous Galerkin methods for the Fornberg-Whitham type equations}
\date{}
\author{Qian Zhang\thanks{School of Mathematical Sciences, University of Science and Technology of China, Hefei, Anhui 230026, P.R. China.  E-mail: gelee@mail.ustc.edu.cn. Current address: Department of Applied Mathematics, The Hong Kong Polytechnic University, Hong Kong, E-mail: qian77.zhang@polyu.edu.hk. Research supported by the 2019 Hong Kong Scholar Program G-YZ2Y.}
\and Yan Xu \thanks{Corresponding author. School of Mathematical Sciences, University of Science and Technology of China, Hefei, Anhui 230026, P.R. China.  E-mail: yxu@ustc.edu.cn. Research supported by National Numerical Windtunnel grants NNW2019ZT4-B08,  Science Challenge Project TZZT2019-A2.3, NSFC grants 11722112.}
\and Chi-Wang Shu \thanks{Division of Applied Mathematics, Brown University, Providence, RI 02912, USA. Email:
chi-wang\_shu@brown.edu. Research supported by NSF grant DMS-1719410.}
}

\maketitle

\begin{abstract}
In this paper, we construct high order energy dissipative and conservative local discontinuous Galerkin methods for the Fornberg-Whitham type equations. We give the proofs for the dissipation and conservation for related conservative quantities. The corresponding error estimates are proved for the proposed schemes. The capability of our schemes for different types of solutions is shown via several numerical experiments. The dissipative schemes have good behavior for shock solutions, while for
a long time approximation, the conservative schemes can reduce the shape error and the decay of amplitude significantly.
\end{abstract}
\vskip 1 cm


\section{Introduction}
The Fornberg-Whitham type equation we study in this paper is given by
\begin{align}\label{eqn:FW_intro}
u_t + f(u)_x + (1-\partial_x^2)^{-1}u_x = 0, \ I \in [a,b], \ t>0
\end{align}
or its equivalent form
\begin{equation}\label{eqn:FW_intro2}
u_t - u_{xxt} + f(u)_x + u_x  = f(u)_{xxx},
\end{equation}
by operating $(1-\partial_x^2)$ on \eqref{eqn:FW_intro}. We consider the nonlinear term $f(u) = \frac{1}{p}u^p$,
where $p \geq 2$ is an integer. When the parameter $p = 2$, the equation \eqref{eqn:FW_intro} becomes the Fornberg-Whitham equation derived in \cite{Whitham_1967_MPS} as a nonlinear dispersive wave equation. There are three conservative quantities for the Fornberg-Whitham type equation
\begin{equation}\label{conservative quantity}
E_0 = \int_I u dx, \quad E_1 = \int_I (u - u_{xx}) dx,\quad E_2 = \int_I u^2 dx,
\end{equation}
where the quantity $E_0$ is called mass, and $E_2$ is energy.

Many mathematical properties of the Fornberg-Whitham equation have been discussed, this equation was first proposed for studying the qualitative behavior of wave breaking in \cite{Whitham_1967_MPS}. Some investigation of wave breaking conditions can be found in \cite{Itasaka_2018_arxiv, Hoermann_2017_JDE}. Note that the Fornberg-Whitham equation is also called Burgers-Poisson equation in \cite{Fellner_2004_SIAM}. There has been lots of work focusing on finding the traveling wave solutions in \cite{Zhou2010_NA, Zhou2008_JMAA}. It admits a wave of greatest height, as a peaked limiting form of the traveling wave solution \cite{Fornberg_1978_MPS}. Recently, some well-posedness results are proposed in \cite{Holmes_2016_JDE, Holmes_2017_JDE}.  There are not many numerical schemes for the Fornberg-Whitham type equation. In \cite{Hormann_2018_arxiv}, the finite difference method is adopted to solve the shock solution. The authors did some valuable numerical analysis by the discontinuous Galerkin method in \cite{Liu2015_NM}, in which the comparisons has been made between the conservative scheme and dissipative scheme, as well as theoretical analysis.

 The discontinuous Galerkin (DG) method was first introduced by Reed and Hill in 1973 \cite{Reed1973} for solving steady-state linear hyperbolic equations. The key point of this method is the design of suitable inter-element boundary treatments (so-called numerical fluxes) to obtain highly accurate and stable schemes in several situations. Within the DG framework, the local discontinuous Galerkin (LDG) method can be obtained by extending to handle derivatives of order higher than one. The first LDG method was introduced by Cockburn and Shu in \cite{Shu1998_Siam} for solving the convection-diffusion equation. Their work was motivated by the successful numerical experiments of Bassi and Rebay \cite{Bassi1997_JCP} for compressible Navier-Stokes equations. The LDG methods can be applied in many equations, such as KdV type equations \cite{Yan2002_Siam, Xu2005_PDNP, Xu2006_CMAME, Xu2007_CMAME,Xing2016_CiCP, Zhang2019CiCP},  Camassa-Holm equations \cite{Xu2004_JCM, Zhang2019_Jsc},  Degasperis-Procesi equation \cite{Xu2011_Cicp}, Schr\"{o}dinger equations \cite{Xu2005_JCP, Xia2014CiCP}, and more nonlinear equations or system \cite{Yan2002_JSC, Xu2004_JCM, Xu2010_CiCP,Liu2018_cicp, Liu2019_cicp}.

 There are also many conservative DG schemes that are proposed to ``preserve structure", such as KdV equation \cite{Bona2013_MC, Xing2016_CiCP, Zhang2019CiCP}, Zakharov system \cite{Xia2010JCP}, Schr\"odinger-KdV system \cite{Xia2014CiCP}, short pulse equation \cite{Zhang2019_JCP}, etc. Usually, the structure preserving  schemes can help reduce the shape error of waves along with long time evolution. For example in \cite{Bona2013_MC, Liu2015_NM, Xing2016_CiCP, Zhang2019CiCP}, compared with dissipative schemes, the energy conservative or Hamiltonian conservative numerical schemes for the KdV equation have less shape error or amplitude damping for long time approximations, especially in the low-resolution cases.

In this paper, we adopt the LDG method as a spatial discretization to construct high order accurate numerical schemes for the Fornberg-Whitham type equations. Through the two equivalent forms of this type of equation, we develop dissipative and conservative schemes, respectively.
The corresponding conservative quantities can be proved as dissipative or conservative for the semi-discrete schemes.
For the dissipative schemes, the theoretical results are confirmed as the $(k+1)$-$th$ order of accuracy if $p$ is odd, and $(k+\frac{1}{2})$-$th$ order of accuracy if $p$ is even.  Here and in what follows, $k$ is the polynomial degree of the finite element space. {The proof for the dissipative schemes is motivated by the work on conservation laws \cite{Zhang2004_SIAM, Zhang2010_SIAM}. Since the numerical fluxes are different from the work \cite{Zhang2004_SIAM, Zhang2010_SIAM}, there exist technical obstacles to derive the a priori error estimates due to the lack of control on some jump terms at interfaces. The error estimate results
for the Fornberg-Whitham type equations can not achieve the same optimal order as in
\cite{Zhang2004_SIAM, Zhang2010_SIAM} for the upwind flux. For the conservative schemes, we prove the $k$-$th$ order of accuracy in a different way from \cite{Bona2013_MC}.} Numerically, the dissipative schemes have good behavior for shock solutions, while for a long time approximation, the conservative schemes can reduce the shape error and the decay of amplitude significantly.
The two proposed LDG schemes vary slightly on efficiency which will be explained in the numerical experiments.

The paper is organized as follows. Through the two equivalent forms of the Fornberg-Whitham type equations, we construct two dissipative, conservative DG schemes in Section \ref{DG1} and \ref{DG2}, respectively.  We demonstrate the dissipation and conservation correspondingly. Additionally, some results of the error estimate are stated.    Subsequently, several numerical experiments are presented in Section \ref{experiments} to show the capability of the methods. This paper is concluded in Section \ref{conclusion}. Some more technical proofs of relevant lemmas are listed in the appendix.

\section {The LDG scheme for  equation (\ref{eqn:FW_intro}) }\label{DG1}

\subsection{Notations}\label{notation}
  We denote the mesh $\mathcal{T}_h $ by $I_j = [x_{j-\frac{1}{2}}, x_{j+\frac{1}{2}}] $ for $j = 1,\ldots, N$, where $x_{\frac{1}{2}} = a, x_{N+\frac{1}{2}} = b $ with the cell center denoted by $ x_j
  =   \frac{1}{2}(x_{j-\frac{1}{2}}+x_{j+\frac{1}{2}})$. The cell size is $\Delta x_j = x_{j+\frac{1}{2}}- x_{j-\frac{1}{2}} $ and $ h =\ \max\limits_{1\leq   j\leq N} \ \Delta x_j$.  The finite element space as the solution and test function space consists of piecewise polynomials
 $$V_h^k = \{v:v|_{I_j} \in P^k(I_j); 1\leq j\leq N\},$$                                                                                               where $P^k(I_j)$ denotes the set of polynomials of degree up to $k$ defined on the cell $I_j$. Notably, the functions in $V_h^k$ are allowed to be discontinuous across cell interfaces. The values of $u$ at $x_{j+\frac{1}{2}}$ are denoted by $u_{j+\frac{1}{2}}^-$ and $u_{j+\frac{1}{2}}^+$,
   from    the left cell $I_j$ and the right cell $I_{j+1}$, respectively. Additionally, the jump of $u$ is defined as
   $$\jump{u} = u^+ - u^- ,$$
   the average of    $u$ as
   $$\average{u} = \frac{1}{2}(u^+ + u^-).$$

To simplify expressions, we adopt the round bracket and angle bracket for  the $L^2$ inner product and boundary term on cell $I_j$
  \begin{equation}\label{bracket_def}
    \begin{split}
  (u, v)_{I_j}& = \int_{I_j} uv dy,\\
  <\hat{u}, v>_{I_j}& = \hat{u}_{j+\frac{1}{2}}v_{j+\frac{1}{2}}^- - \hat{u}_{j-\frac{1}{2}}v_{j-\frac{1}{2}}^+
    \end{split}
     \end{equation}
  for the one dimensional case.

\subsection{The LDG scheme}\label{se:ldg1}
We consider the Fornberg-Whitham type equations
\begin{align}\label{eqn:GFW}
u_t + f(u)_x + (1-\partial_x^2)^{-1}u_x = 0,
\end{align}
where $f(u) = \displaystyle\frac{1}{p}u^p$, and $p \geq 2$ is an integer. The periodic boundary condition is adopted here, which is not essential. The LDG methods we propose here or later can be designed for non-periodic boundary condition easily.

First, we split the equation \eqref{eqn:GFW} into a first-order system
\begin{align*}
&u_t + f(u)_x  + v = 0,\\
&v- q_x = u_x,\\
&q = v_x.
\end{align*}

Then the semi-discrete LDG scheme is formulated as: Find numerical solutions $u_h, v_h, q_h \in V_h^k$, such that
\begin{subnumcases}{\label{scheme:DG}}
((u_h)_t, \phi)_{I_j} + <\widehat{f(u_h)}, \phi>_{I_j} - (f(u_h), \phi_x)_{I_j}  +(v_h , \phi)_{I_j} = 0,\label{scheme:DG1} \\
(v_h , \varphi)_{I_j} - <\widehat{q_h}, \varphi>_{I_j} + (q_h, \varphi_x)_{I_j} = <\widehat{u_h}, \varphi>_{I_j} - (u_h, \varphi_x)_{I_j},\label{scheme:DG2}\\
(q_h , \psi)_{I_j} = <\widehat{v_h}, \psi>_{I_j} - (v_h, \psi_x)_{I_j},\label{scheme:DG3}
\end{subnumcases}
for any test function $\phi,\varphi,\psi \in V^k_h$. Here, the ``hat" terms in \eqref{scheme:DG} are the so-called ``numerical fluxes'', which are functions defined on the cell boundary from integration by parts and should be designed based on different guiding principles for different PDEs to ensure the stability and local solvability of the intermediate variables. The main distinctions between dissipative and conservative schemes are the choices of numerical fluxes. We have two LDG methods as follows:

$\mathbf{Scheme \ \mathcal{D}1 :}$
For the dissipative numerical flux of the nonlinear term $f(u)$ here,  we take the Godunov flux to prepare the optimal convergence rate subsequently,
\begin{equation}\label{eqn: diss_flux_non}
\widehat{f(u_h)}\equiv \widehat{f}(u_h^-,u_h^+) =  \begin{cases} \min_{u_h^-\leq u_h \leq u_h^+}\; f(u_h), & u_h^- < u_h^+ \\
    \max_{u_h^+\leq u_h \leq u_h^-}\; f(u_h), & u_h^+ \leq u_h^-\end{cases},
\end{equation}
which is called an upwind numerical flux satisfying the following property,
\begin{equation}\label{eqn:upwind}
\widehat{f(u_h)} =  \begin{cases} f(u_h^-), \ & \text{if} \ f'(u)\geq 0,\  u \in[ \min{(u_h^-, u_h^+)}, \max{(u_h^-, u_h^+)}]  \\
                                  f(u_h^+), \ & \text{if} \ f'(u)< 0, \ u \in[ \min{(u_h^-, u_h^+)}, \max{(u_h^-, u_h^+)}] \end{cases}.
\end{equation}
For the numerical fluxes $\widehat{q_h},\widehat{v_h},\widehat{u_h}$, we select
\begin{equation}\label{eqn: diss_flux}
\widehat{q_h} = q_h^-,\ \widehat{v_h} = v_h^+,\ \widehat{u_h} = u_h^-
\end{equation}
to maintain the energy stability. Numerically, the dissipative scheme \eqref{scheme:DG} with fluxes \eqref{eqn: diss_flux_non}, \eqref{eqn: diss_flux} can achieve $(k+1)$-$th$ order of accuracy.

$\mathbf{Scheme \ \mathcal{C}1 :}$ For the conservative scheme $\mathcal{C}1$, the numerical flux for the nonlinear term $f(u)$ is given by
\begin{equation}\label{fluxC1:non}
\widehat{f(u_h)} \equiv \widehat{f}(u_h^-,u_h^+) =  \begin{cases}\displaystyle \frac{\jump{F(u_h)}}{\jump{u_h}}, \; &\jump{u_h} \neq 0 \\ f(\average{u_h}),\; &\jump{u_h} = 0\end{cases},
\end{equation}
where $ F(u) = \int^u f(\tau)d\tau$, especially for $f(u) = \frac{1}{p}u^p, p \geq 2 $,
$$ \quad \widehat{f(u_h)} = \frac{1}{p(p+1)}\sum\limits_{m=0}^p(u_h^+)^{p-m}(u_h^-)^m $$
 as in \cite{Bona2013_MC}.
Then we choose the central fluxes for $\widehat{q_h},\widehat{v_h},\widehat{u_h}$,
\begin{equation}\label{eqn: con_flux}
\widehat{q_h} = \average{q_h},\ \widehat{v_h} = \average{v_h},\ \widehat{u_h} = \average{u_h}.
\end{equation}
For the conservative scheme \eqref{scheme:DG} with fluxes \eqref{fluxC1:non}, \eqref{eqn: con_flux} on uniform meshes,
we obtain the $k$-$th$ order for odd $k$, and $(k+1)$-$th$ order for even $k$ numerically.

\subsection{Dissipation and conservation}\label{sec:stab}
In this section, we provide proof of energy dissipation or energy conservation for the proposed LDG schemes in Section \ref{se:ldg1}.  The proposition demonstrates the conservative quantities based on which we construct the proposed numerical schemes $\mathcal{D}1$ and $\mathcal{C}1$, including mass $E_0$, and energy $E_2$ in \eqref{conservative quantity}.

Before the proposition, we define some bilinear operators to simplify our expressions.
\begin{definition} We define bilinear operators $\mathcal{L}^{\pm,c}_j$ as
\begin{align}\label{operatorD}
\mathcal{L}^{\pm,c}_j(\omega,\phi) = -(\omega, \phi_x)_{I_j} + <\widehat{\omega},\phi>_{I_j},
\end{align}
where the direction of $\widehat{w}$ determines the operator $\mathcal{L}^+$, $\mathcal{L}^-$ or $\mathcal{L}^c$. Wherein $\mathcal{L}^{\pm}$ denotes the cases with $\omega^{\pm}$ correspondingly, and $\mathcal{L}^{c}$ is for $\widehat{\omega} = \omega^{c}$.
\end{definition}
\begin{definition}
The operators $\mathcal{N}^{c,d}_j$ for the nonlinear term $f(u)$ are defined as
\begin{align}\label{operatorN}
\mathcal{N}^{c,d}_j(\omega,\phi) = -(f(\omega), \phi_x)_{I_j} + <\widehat{f(\omega)},\phi>_{I_j}
\end{align}
The distinction between the dissipative form $\mathcal{N}^{d}_j$ and the conservative form $\mathcal{N}^{c}_j$ lies in the
numerical flux $\widehat{f}$, which are taken in \eqref{eqn: diss_flux_non} or \eqref{fluxC1:non}, respectively.
\end{definition}

\begin{lemma}\label{lemma:properties_DN}
Let $\mathcal{L}^{\pm,c} = \sum\limits_{j=1}^{N}\mathcal{L}^{\pm,c}_j, \mathcal{N}^{c,d} = \sum\limits_{j=1}^{N} \mathcal{N}^{c,d}_j$, there hold the following properties,
\begin{equation}\label{properties_DN}
\begin{split}
&\mathcal{L}^c(\omega,\omega) = 0, \quad \mathcal{L}^+(\omega,\omega) = - \frac{1}{2}\sum\limits_{j=1}^{N}\jump{\omega}^2_{j+\frac{1}{2}};\\
&\mathcal{L}^+(\omega,\phi) + \mathcal{L}^-(\phi,\omega) =0, \quad \mathcal{L}^c(\omega,\phi) + \mathcal{L}^c(\phi,\omega) = 0; \\
& \mathcal{L}^-(\omega,\phi)+\mathcal{L}^-(\phi,\omega) = \sum\limits_{j=1}^{N}\jump{\omega}\jump{\phi}_{j+\frac{1}{2}} ; \\
&\mathcal{N}^d(\omega,\omega) \geq 0,  \quad \mathcal{N}^c(\omega,\omega) = 0
\end{split}
\end{equation}
for $\forall \omega, \phi \in V_h^k$.
\end{lemma}
The properties in Lemma \ref{lemma:properties_DN} can be easily derived by algebraic manipulation which had already been proved in \cite{Bona2013_MC}, \cite{Xu2006_CMAME}, so we do not give the details here.

\begin{prop}\label{prop3}
For periodic problems, we have
\begin{itemize}
\item {Scheme $\mathcal{D}1 $}
\begin{align}
\label{prop3_d1}  \frac{d}{dt} E_0(u_h) =  \frac{d}{dt}\int_I u_h dx = 0,\quad\quad \frac{d}{dt} E_2(u_h) = \frac{d}{dt}\int_I u_h^2 dx \leq 0.
\end{align}
\item {Scheme $\mathcal{C}1 $}
\begin{align}
\label{prop3_c1}  \frac{d}{dt} E_0(u_h) = \frac{d}{dt} \int_I u_h dx = 0,\quad\quad \frac{d}{dt} E_2(u_h) = \frac{d}{dt}\int_I u_h^2 dx = 0.
\end{align}

\end{itemize}

\end{prop}

\begin{proof}
First, we can obtain the mass $E_0$ conservation after summing up equations \eqref{scheme:DG1},
\eqref{scheme:DG2} with test functions $\phi = 1, \varphi =1$.

Next, we prove the $L^2$ stability by taking the test functions as
\begin{equation*}
\phi = u_h, \ \varphi = -u_h, \ \varphi = -q_h, \ \psi = v_h
\end{equation*}
in scheme \eqref{scheme:DG}. After summing up corresponding equations over all intervals,
\begin{itemize}
\item For the dissipative scheme $\mathcal{D}1$
\end{itemize}
\begin{align*}
&((u_h)_t, u_h)_{I} + \mathcal{N}^{d}(u_h,u_h) + \mathcal{L}^-(q_h,u_h) + \mathcal{L}^-(u_h,q_h)+ \mathcal{L}^-(u_h,u_h) + \mathcal{L}^-(q_h,q_h) - \mathcal{L}^+(v_h,v_h)  \\
& =((u_h)_t, u_h)_{I} + \mathcal{N}^{d}(u_h,u_h) + \frac{1}{2}\sum\limits_{j=1}^{N}\big((\jump{u_h}+\jump{q_h})^2_{j+\frac{1}{2}} + \jump{v_h}^2_{j+\frac{1}{2}}\big) = 0.
\end{align*}
\begin{itemize}
\item For the conservative scheme $\mathcal{C}1$
\end{itemize}
\begin{align*}
&((u_h)_t, u_h)_{I} + \mathcal{N}^{c}(u_h,u_h) + \mathcal{L}^c(q_h,u_h) + \mathcal{L}^c(u_h,q_h)
+ \mathcal{L}^c(u_h,u_h) + \mathcal{L}^c(q_h,q_h) - \mathcal{L}^c(v_h,v_h)\\& =((u_h)_t, u_h)_{I}  = 0.
\end{align*}
 Here we have used the results in Lemma \ref{lemma:properties_DN}. Then we obtain the final dissipation or conservation in \eqref{prop3_d1} and \eqref{prop3_c1}.
%
%
%
%
\end{proof}


\subsection{Error estimates}\label{sec:error_estimate}

In this section, we provide error estimates of the LDG schemes in Section \ref{se:ldg1} for the sufficiently smooth periodic solution of the Fornberg-Whitham type equations.
\subsubsection{Notations, projections, and auxiliary results} \label{thm:prepare}
First, we make some conventions for different constants. For any time $t$ in $[0,T]$, we assume that the exact solution and its spatial derivatives are all bounded. We use the notation $C$ to denote a positive constant which is independent of $h$, but depends on $\abs{f'}$ and the exact solution of the problem considered in this paper. Additionally, the notation $C_*$  is used to denote the constants which are relevant to the maximum of $\abs{f''}$. 
 Under different circumstances, these constants will have different values.

Next, we will introduce some projection properties to be used later. The standard $L^2$ projection of a function $\zeta$ with $k+1$ continuous derivatives into space $V_h^k$, is denoted by $\mathcal{P}$, i.e., for each $I_j$
\begin{equation*}
\begin{split}
&(\mathcal{P}\zeta - \zeta,\phi)_{I_j} =0, \ \forall \phi \in P^{k}(I_j),
\end{split}
\end{equation*}
and the Gauss Radau projections $\mathcal{P}^{\pm}$ into $V_h^k$  satisfy, for each $I_j$,
\begin{align*}
(\mathcal{P}^{+}\zeta - \zeta,\phi)_{I_j} =0, \ \forall \phi \in P^{k-1}(I_j), \ \text{and} \ \mathcal{P}^{+}\zeta(x_{j-\frac{1}{2}}^+) = \zeta({x_{j-\frac{1}{2}}}),\\
(\mathcal{P}^{-}\zeta - \zeta,\phi)_{I_j} =0, \ \forall \phi \in P^{k-1}(I_j),\ \text{and} \
\mathcal{P}^{-}\zeta(x_{j+\frac{1}{2}}^-) = \zeta({x_{j+\frac{1}{2}}}).
\end{align*}
For the projections mentioned above, it is easy to show \cite{1975_Ciarlet_NH} that
\begin{equation}\label{projection error}
\norm{ \zeta^e}_{L^2(I)} +  h^{\frac{1}{2}} \norm{ \zeta^e}_{\infty} + h^{\frac{1}{2}}\norm{ \zeta^e}_{L^2({\partial I})}  \leq Ch^{k+1}
\end{equation}
where $\zeta^e =\zeta -\mathcal{P}\zeta$ or $ \zeta -\mathcal{P}^{\pm}\zeta$, and the positive constant $C$ only depends on $\zeta$.

Then some inverse inequalities of the finite element space $V_h^k$ will be applied in the subsequent proofs.
\begin{lemma}\cite{1975_Ciarlet_NH}
 For $\forall \omega \in V_h^k$, there exists a positive constant $C$ which is independent on $\omega, h$, such that
\begin{equation}\label{eqn:inverse inequality}
(i) \norm{\omega_x}_{L^2({I})} \leq C h^{-1}\norm{\omega}_{L^2(I)},\
 (ii)\norm{\omega}_{L^2(\partial{I})} \leq C h^{-\frac{1}{2}}\norm{\omega}_{L^2(I)},\
  (iii)\norm{\omega}_{\infty} \leq C h^{-\frac{1}{2}}\norm{\omega}_{L^2(I)},
\end{equation}
where
$$\norm{\omega}_{L^2(\partial{I})}  = \sqrt{\sum\limits_{j=1}^{N}(\omega_{j+\frac{1}{2}}^-)^2 +(\omega_{j-\frac{1}{2}}^+)^2}.$$
\end{lemma}

\subsubsection{The main error estimate results}\label{sec:Error_1}
\begin{theorem}\label{thm1}
It is assumed that the Fornberg-Whitham type equations \eqref{eqn:GFW}  with periodic boundary condition has a sufficiently smooth exact solution $u$. The numerical solution $u_h$ satisfies the semi-discrete LDG scheme \eqref{scheme:DG}. For regular partitions of $I = (a, b)$, and the finite element space $V^k_h$, there hold the following error estimates for small enough $h$
\begin{itemize}
\item {Scheme $\mathcal{D}1 $:}
\begin{align}\label{eqn:thm1d}
   \begin{cases}&\norm{ u - u_h}_{L^2(I)} \leq Ch^{k+1},  \quad\text{if $p$ is odd},\\
                                                      &\norm{ u - u_h}_{L^2(I)} \leq Ch^{k+\frac{1}{2}}, \quad \text{if $p$ is even}.
                                                 \end{cases}, \ k \geq 1
\end{align}
\item {Scheme $\mathcal{C}1 $:}
\begin{align}
   \norm{ u - u_h}_{L^2(I)} \leq Ch^{k},   \ k \geq 2  \label{eqn:thm1c}
\end{align}

\end{itemize}
where the integer $p$ is in the nonlinear term $f(u) = \frac{1}{p}u^p$. The constant $C$ depends on the final time $T$, $k$, $\norm{u}_{k+2}$ and the bounds of derivatives up to second order of the
nonlinear term $f(u)$. Here, $\norm{u}_{k+2}$ is the maximum of the standard Sobolev $k+2$ norm over $[0, T]$.

\end{theorem}


\subsubsection{The error equation}\label{error_eqution1}
Since the exact solution also satisfies the numerical scheme \eqref{scheme:DG}, doing subtraction can bring us the error equations.
Due to the different choices of test functions in the error equations, we define the bilinear form $\mathcal{B}_j$ as
\begin{align*}
&\mathcal{B}_j(u-u_h, v-v_h, q-q_h ; \phi, \bm{\varphi},\bm{\psi}) = ((u - u_h)_t,\phi)_{I_j}+ (v - v_h,\phi)_{I_j} \label{eqn:B_bilinear}\\
&+(q - q_h,\bm{\psi})_{I_j}- \mathcal{L}^+_j(v-v_h,\bm{\psi})+
  (v - v_h,\bm{\varphi})_{I_j} - \mathcal{L}^-_j(q-q_h,\bm{\varphi}) - \mathcal{L}^-_j(u-u_h,\bm{\varphi})\notag
\end{align*}
and the form  for the nonlinear term $f(u)$, which is linear with respect to its second argument, is
\begin{align*}
\mathcal{H}_j(f;u,u_h,\phi) = (f(u) - f(u_h),\phi_x)_{I_j} - <f(u) - \widehat{f(u_h)},\phi>_{I_j}.
\end{align*}
Here, the notations $\bm{\varphi}=(\varphi_1, \varphi_2, \varphi_3)$, $\bm{\psi}=(\psi_1, \psi_2, \psi_3)$ are vectors
consisting of test functions in the finite element space $V_h^k$.  We also define
\begin{align*}
(u,\bm{\varphi})_{I_j} = \sum\limits_{i=1}^3 (u,\varphi_i)_{I_j},\quad \mathcal{L}^{\pm,c}_j(u,\bm{\varphi}) = \sum\limits_{i=1}^3 \mathcal{L}^{\pm,c}_j(u,\varphi_i)
\end{align*}
for writing convenience. After applying summation over all cells $I_j$, the error equation is expressed by
{
\begin{align}
\sum\limits_{j=1}^{N}  \mathcal{B}_j(u-u_h, v-v_h, q-q_h; \phi,\bm{\varphi},\bm{\psi}) = \sum\limits_{j=1}^{N}\mathcal{H}_j(f;u,u_h,\phi).
\end{align}
}

Here, for the dissipative scheme $\mathcal{D}1$, we define
\begin{equation}\label{def:test_function_1}
\begin{split}
&\xi^u = \mathcal{P}^-u-u_h, \quad \eta^u = \mathcal{P}^-u - u, \\
&\xi^v = \mathcal{P}^+v-v_h,\quad \eta^v = \mathcal{P}^+v - v,\\
&\xi^q = \mathcal{P}^-q-q_h,\quad \eta^q = \mathcal{P}^-q - q.
\end{split}
\end{equation}
Without causing misunderstanding, the notations in \eqref{def:test_function_1} are still used to represent the analogs for the conservative scheme. For conservative scheme $\mathcal{C}1$,  we replace the Gauss-Radau projections by the
standard $L^2$ projections
\begin{equation*}
\begin{split}
&\xi^u = \mathcal{P}u-u_h, \quad \eta^u = \mathcal{P}u - u, \\
&\xi^v = \mathcal{P}v-v_h,\quad \eta^v = \mathcal{P}v - v,\\
&\xi^q = \mathcal{P}q-q_h,\quad \eta^q = \mathcal{P}q - q.
\end{split}
\end{equation*}

 Taking test functions
\begin{equation*}
\begin{split}
&\phi = \xi^u, \ \bm{\varphi} = \bm{\xi^1}\triangleq (-\xi^u, -\xi^q, \xi^v),\ \bm{\psi} = \bm{\xi^2}\triangleq (\xi^u, \xi^q, \xi^v),
\end{split}
\end{equation*}
we have the energy equality as
\begin{align*}
\sum\limits_{j=1}^{N} \mathcal{B}_j(\xi^u- \eta^u, \xi^v- \eta^v, \xi^q- \eta^q ; \xi^u, \bm{\xi^1},\bm{\xi^2}) = \sum\limits_{j=1}^{N}\mathcal{H}_j(f;u,u_h,\xi^u).
\end{align*}

\subsubsection{The proof of the main results in Theorem \ref{thm1}}\label{main_results1}

Next, we analyze $\mathcal{B}_j$ and $\mathcal{H}_j$, respectively. Some primary estimate results will be provided in the following lemmas.
 \begin{itemize}
\item {\bf Estimates for the linear terms}
\end{itemize}
\begin{lemma}\label{lemma:another_energy}
The following energy equality holds,
\begin{equation}\label{eqn:another_energy}
\norm{v_h}^2_{L^2(I)} + \norm{q_h}^2_{L^2(I)} + (q_h,u_h)_I = 0.
\end{equation}
\end{lemma}

\begin{proof}The proof is provided in Appendix \ref{proofanother_energy}.
\end{proof}

\begin{lemma}\label{lemma:b}
For the bilinear form $\mathcal{B}_j$, the following equations hold by projection properties
\begin{itemize}
\item  Scheme $\mathcal{D}1 $
\begin{align}
&\sum\limits_{j=1}^{N} \mathcal{B}_j(\xi^u- \eta^u, \xi^v- \eta^v, \xi^q- \eta^q ; \xi^u, \bm{\xi^1}, \bm{\xi^2}) \notag\\
&=  (\xi_t^u,\xi^u)_I  + \norm{\xi^v}^2_{L^2(I)}+ \norm{\xi^q}^2_{L^2(I)}   + \sum\limits_{j=1}^{N} \frac{1}{2}((\jump{\xi^u}+ \jump{\xi^q})^2_{j+\frac{1}{2}} + \jump{\xi^v}^2_{j+\frac{1}{2}})+(\xi^u,\xi^q)_I \notag\\
& \quad \ - ( \eta^q + \eta^u_t,\xi^u)_I  -(\eta^q-\eta^v,\xi^q)_I -(\eta^q+\eta^v, \xi^v)_I ; \label{eqn:B_estimate}
\end{align}
\item  Scheme $\mathcal{C}1 $
\begin{align}
&\sum\limits_{j=1}^{N} \mathcal{B}_j(\xi^u- \eta^u, \xi^v- \eta^v, \xi^q- \eta^q ; \xi^u, \bm{\xi^1}, \bm{\xi^2})= ({\xi}_t^u,{\xi}^u)_I  + \norm{{\xi}^v}^2_{L^2(I)} + \norm{{\xi}^q}^2_{L^2(I)}  \notag \\
&  + ({\xi}^u,{\xi}^q)_{I} -\sum\limits_{j=1}^{N}( <\widehat{{\eta}^q}+\widehat{{\eta}^u} -\widehat{{\eta}^v}, {\xi}^u + {\xi}^q >_{I_j} -  <\widehat{{\eta}^q}+\widehat{{\eta}^u} +\widehat{{\eta}^v}, {\xi}^v>_{I_j}) . \label{eqn:B_estimate2}
\end{align}
\end{itemize}
\end{lemma}
\begin{proof}The proof is provided in Appendix \ref{proofb}.
\end{proof}

 \begin{itemize}
\item {\bf Estimates for the nonlinear terms}
\end{itemize}


Subsequently, our attention is turned to $\mathcal{H}_j$ involving the nonlinear term $f(u)$,
\begin{equation}\label{H_divided}
\begin{split}
\sum\limits_{j=1}^{N}\mathcal{H}_j(f;u,u_h,\xi^u) &= \sum\limits_{j=1}^{N} (f(u)- f(u_h),\xi^u_x)_{I_j} + (f(u) - f(u_h^{ref}))\jump{\xi^u}_{j+\frac{1}{2}}\\
& + \sum\limits_{j=1}^{N} (f(u_h^{ref}) - \widehat{f(u_h)})\jump{\xi^u}_{j+\frac{1}{2}}\triangleq \mathcal{T}_1 + \mathcal{T}_2.
\end{split}
\end{equation}
The index ``$ref$" denotes the direction of the value on each element interface depending
on the flow direction of the exact solution $u$ in the adjacent elements.  For the dissipative scheme $\mathcal{D}1$,
\begin{align}\label{uref}
\chi^{ref} \triangleq \begin{cases} \chi^+, \ &\text{if} \ f'(u) < 0 \ \text{on}\  I_j \cup x_{j+\frac{1}{2}}\cup I_{j+1} \\
                      \chi^-,\ & \text{if} \ f'(u) > 0 \ \text{on}\   I_j\cup x_{j+\frac{1}{2}} \cup I_{j+1} \\
                      \average{\chi}, \ & \text{otherwise}
        \end{cases}.
\end{align}
For the conservative scheme $\mathcal{C}1$, the direction of the boundary value is considered as $\average{\chi}$ for
simplicity of the proof.
{
\begin{lemma}\label{lemma:H_estimate2}
If the dissipative flux is taken as $f(u) =\displaystyle \frac{1}{p}u^p, p \geq 2$, we have
\begin{align}\label{eqn:H_estimate2_podd}
 \mathcal{T}_1 + \mathcal{T}_2
 & \leq  C_*\norm{\xi^u}^2_{L^2(I)} + Ch^{2k+\mu}  + Ch^{-\frac{3}{2}}\norm{\xi^u}^3_{L^2(I)}
\end{align}
for sufficient small $h$ and $k \geq 1$, where $\mu = 1 $ for even $p$, and $\mu = 2$ for odd $p$.
\end{lemma}
}
\begin{proof}The proof is provided in Appendix \ref{proofH_estimate2}.
\end{proof}
{
\begin{lemma}\label{lemma:hc1}
If the conservative flux is taken as $f(u) = \displaystyle\frac{1}{p}u^p, p \geq 2$, then
\begin{align}\label{eqn:hc1}
\mathcal{T}_1 + \mathcal{T}_2 \leq C_*\norm{{\xi}^u}^2_{L^2(I)} + Ch^{2k} + Ch^{-\frac{3}{2}}\norm{{\xi}^u}^3_{L^2(I)}
\end{align}
for sufficient small $h$ and $k \geq 2$.
\end{lemma}
}
\begin{proof}The proof is provided in Appendix \ref{proofhc1}.
\end{proof}

{
\begin{lemma}\label{lemma}
Let $v\in V_h^k$, if it satisfies
\begin{align}\label{lemma:condition}
\frac{d}{dt}\norm{v}^2_{L^2(I)}  \leq C_*\norm{v}^2_{L^2{(I)}} + Ch^{2k+\tilde{\mu}}+ Ch^{-\frac{3}{2}}\norm{v}^3_{L^2(I)},
\end{align}
then there holds
\begin{align}\label{lemma:result}
\norm{v}^2_{L^2(I)}  \leq  Ch^{2k+\tilde{\mu}},
\end{align}
where $\tilde{\mu}$ is a constant and $k \geq \frac{3-\tilde{\mu}}{2}$.

\end{lemma}}
\begin{proof}The proof is provided in Appendix \ref{prooflemma}.
\end{proof}

\begin{itemize}
\item  {\bf Final error estimates in Theorem \ref{thm1}}
\end{itemize}

Finally, we are ready to get the error estimates in Theorem \ref{thm1}. We divide the final error $\norm{u-u_h}_{L^2(I)}$ into two parts: the projection error $\eta^u$  and the approximation error  $ {\xi^u}$. Once we prove the order of the
approximation errors, the results of Theorem \ref{thm1} is derived by triangle inequality and the interpolation inequality \eqref{projection error} straightforwardly. Therefore, we mainly focus on $\norm{\xi^u}_{L^2(I)}$.

\begin{itemize}
\item [] {\bf Estimates in  \eqref{eqn:thm1d}:}
\end{itemize}
 Gathering together the estimates \eqref{eqn:B_estimate}, \eqref{eqn:H_estimate2_podd}, the final error estimate for Scheme $\mathcal{D}1$ is listed as follows,
\begin{align*}
&(\xi^u_t, \xi^u)_{I} + \norm{\xi^v}^2_{L^2(I)} + \norm{\xi^q}^2_{L^2(I)} \leq -(\xi^u,\xi^q)_I + ( \eta^q + \eta^u_t,\xi^u)_I +(\eta^q-\eta^v,\xi^q)_I +(\eta^q+\eta^v, \xi^v)_I \\
&+  C_*\norm{\xi^u}^2_{L^2(I)} + Ch^{2k+\mu}  + Ch^{-\frac{3}{2}}\norm{\xi^u}^3_{L^2(I)}
\end{align*}
 where $\mu = 1$ for even $p$,  and $\mu = 2$ for odd $p$.
\begin{itemize}
\item [] {\bf Estimates in  \eqref{eqn:thm1c}:}
\end{itemize}
 Lemma \ref{lemma:b} and Lemma \ref{lemma:hc1} lead us to the error estimate for Scheme $\mathcal{C}1$
\begin{align*}
&({\xi}_t^u,{\xi}^u)_I  + \norm{{\xi}^v}^2_{L^2(I)} + \norm{{\xi}^q}^2_{L^2(I)}
 \leq   C_*\norm{\xi^u}^2_{L^2(I)} + Ch^{2k} + Ch^{-\frac{3}{2}}\norm{\xi^u}^3_{L^2(I)} \\
& - ({\xi}^u,{\xi}^q)_{I}+ \sum\limits_{j=1}^{N} <\widehat{{\eta}^q}+\widehat{{\eta}^u} - \widehat{{\eta}^v}, {\xi}^u + {\xi}^q >_{I_j} - \sum\limits_{j=1}^{N}  <\widehat{{\eta}^q}+\widehat{{\eta}^u} +\widehat{{\eta}^v}, {\xi}^v>_{I_j}.\notag
\end{align*}
Using the Young's inequality with weights as
 $$\abs{a_1}\abs{a_2} = \abs{\sqrt{2}a_1}\abs{\frac{1}{\sqrt{2}}a_2} \leq a_1^2 +\frac{1}{4}a_2^2,$$
 and inverse inequality (ii) in (2.15), we get
 {
\begin{align}\label{new_estiamte}
 \frac{1}{2}\frac{d}{dt}\norm{\xi^u}^2_{L^2(I)}  \leq C_*\norm{\xi^u}^2_{L^2{(I)}} + Ch^{2k+\tilde{\mu}}+
Ch^{-\frac{3}{2}}\norm{\xi^u}^3_{L^2(I)} ,
\end{align}}
where $\tilde{\mu} = \mu$ for Scheme $\mathcal{D}1$, i.e. the parity of $p$ determines the different convergence rates, and $\tilde{\mu} = 0$ for Scheme $\mathcal{C}1$. Subsequently, we apply Lemma \ref{lemma} and finally get the results of Theorem \ref{thm1}.

{
\begin{remark}
There is a further result for the conservative scheme with even $k$ and odd $N$, i.e. the optimal order of accuracy can be proved in \cite{Liu2015_NM}.
\end{remark}
}

\section{The LDG scheme for equation (\ref{eqn:FW_intro2})}\label{DG2}
Based on the form
\begin{equation}\label{eqn:FW1}
u_t - u_{xxt} + f(u)_x + u_x = f(u)_{xxx},
\end{equation}
we construct another two LDG schemes, including a dissipative scheme and a conservative scheme named by
Scheme $ \mathcal{D}2$, and Scheme $ \mathcal{C}2$, respectively. The conservation or dissipation and the
corresponding error estimate are also provided in this section.

\subsection{The LDG scheme}\label{LDG2}
\label{se:ldg2}
Referring to \cite{Xu2008_siam, Xu2011_Cicp}, we split the above equation \eqref{eqn:FW1} into
\begin{align}
\label{FW_1} &w = u - u_{xx}, \\
\label{FW_2} &w_t + f(u)_x + u_x = f(u)_{xxx},
\end{align}
with periodic boundary condition.  Then we first rewrite the above equation \eqref{FW_1} into a first-order system
\begin{equation}\label{FW1Ode1}
\begin{split}
&u - r_x = w, \\
&r - u_x = 0.
\end{split}
\end{equation}
By this standard elliptic equation \eqref{FW_1}, we can solve $u$ from a known $w$. Then the LDG method for \eqref{FW1Ode1} is formulated as follows: Find numerical solutions $u_h$, $r_h \in V_h^k$  such that
\begin{subnumcases}{\label{scheme:FW1DG}}
(u_h,\phi)_{I_j} - <\widehat{r_h}, \phi>_{I_j} + (r_h, \phi_x)_{I_j} = (w_h , \phi)_{I_j},\label{scheme:FW1DG1}\\
(r_h , \psi)_{I_j} - <\widehat{u_h}, \psi>_{I_j} + (u_h, \psi_x)_{I_j} = 0 ,\label{scheme:FW1DG2}
\end{subnumcases}
for all test functions $\phi, \psi \in V_h^k$.

For \eqref{FW_2}, we can also rewrite it into a first-order system:
\begin{equation} \label{FW1Ode2}
\begin{split}
&w_t + s = p_x,\\
&p = s_x - u,\\
&s = f(u)_x.
\end{split}
\end{equation}
 Subsequently, we define the LDG scheme for \eqref{FW1Ode2} as: Find numerical solutions $w_h,u_h,$  $p_h, s_h \in V_h^k$ such that
\begin{subnumcases}{\label{scheme:FW2DG}}
((w_h)_t,\varphi)_{I_j}  + (s_h , \varphi)_{I_j} =  <\widehat{p_h}, \varphi>_{I_j} - (p_h, \varphi_x)_{I_j},\label{scheme:FW2DG1}\\
(p_h , \vartheta)_{I_j} =  <\widehat{s_h}, \vartheta>_{I_j} - (s_h, \vartheta_x)_{I_j} - (u_h,\vartheta)_{I_j},\label{scheme:FW2DG2}\\
(s_h , \sigma)_{I_j} =  <\widehat{f(u_h)}, \sigma>_{I_j} - (f(u_h), \sigma_x)_{I_j},\label{scheme:FW2DG3}
\end{subnumcases}
for all test functions $\varphi, \vartheta, \sigma \in V_h^k$.

$\mathbf{Scheme \ \mathcal{D}2 :}$
The dissipative numerical flux for the nonlinear term $f(u)$ is again taken as the Godunov flux \eqref{eqn: diss_flux_non}.
The remaining numerical fluxes are considered to guarantee $L^2$ stability as
\begin{equation}\label{fluxD1:lin}
\widehat{u_h} = u_h^+,\ \widehat{r_h} = r_h^-,\ \widehat{s_h} = s_h^+, \ \widehat{p_h} = p_h^-.
\end{equation}
Numerically,  Scheme $\mathcal{D}2$ can obtain the optimal order of accuracy for the variable $u$.

$\mathbf{Scheme \ \mathcal{C}2 :}$
The conservative numerical flux for the nonlinear term $f(u)$ is the same as \eqref{fluxC1:non}, and then we choose the central fluxes for $u_h, r_h, s_h, p_h$,
\begin{equation}\label{fluxC1:lin}
\widehat{u_h} = \average{u_h},\ \widehat{r_h} = \average{r_h},\ \widehat{s_h} = \average{s_h}, \ \widehat{p_h} = \average{p_h}.
\end{equation}
Numerically,  Scheme $\mathcal{C}2$ can obtain the $k$-$th$ order when piecewise polynomials of odd degree $k$ are used,
and $(k+1)$-$th$ order for even degree $k$.

\subsection{Dissipation and conservation}
In this section, a proposition demonstrating dissipation or conservation for the proposed LDG schemes in Section \ref{LDG2} is stated, including $E_1$ conservation and energy $E_2$ dissipation or conservation.

\begin{prop}\label{prop1}

For periodic problems, we have
\begin{itemize}
\item {Scheme $\mathcal{D}2 $}
\begin{align}
\label{prop1_d1}  \frac{d}{dt} E_1(u_h) =  \frac{d}{dt}\int_I w_h dx = 0,\quad\quad \frac{d}{dt} E_2(u_h) = \frac{d}{dt}\int_I u_h^2 dx \leq 0.
\end{align}
\item {Scheme $\mathcal{C}2 $}
\begin{align}
\label{prop1_c1}  \frac{d}{dt} E_1(u_h) = \frac{d}{dt} \int_I w_h dx = 0,\quad\quad \frac{d}{dt} E_2(u_h) = \frac{d}{dt}\int_I u_h^2 dx = 0.
\end{align}

\end{itemize}

\end{prop}

\begin{proof}
 Owing to the form of conservation law we give in equation \eqref{FW_2}, we can have the $E_1$ conservativeness with periodic boundary conditions for  Scheme $\mathcal{D}2, \mathcal{C}2$ trivially.

Subsequently,  we will begin the proof of $E_2$ dissipation or conservation. First, for the two equations in \eqref{scheme:FW1DG}, we take a first-order temporal derivative as
\begin{subnumcases}{\label{scheme:FW1DG_t}}
((u_h)_t,\phi)_{I_j} - <\widehat{(r_h)}_t, \phi>_{I_j} + ((r_h)_t, \phi_x)_{I_j} = ((w_h)_t , \phi)_{I_j},\label{scheme:FW1DG1_t}\\
((r_h)_t , \psi)_{I_j} - <\widehat{(u_h)}_t, \psi>_{I_j} + ((u_h)_t, \psi_x)_{I_j} = 0 .\label{scheme:FW1DG2_t}
\end{subnumcases}
Since the numerical schemes \eqref{scheme:FW1DG_t}, \eqref{scheme:FW2DG} hold for any test function in test space $V_h^k$, we choose
\begin{align*}
\varphi = \phi =- p_h \ \text{and} \ -(r_h)_t, \ \vartheta = \psi = (u_h)_t \ \text{and} \  s_h, \sigma = -u_h.
\end{align*}
After summation of corresponding equalities over all intervals, Lemma \ref{lemma:properties_DN} leads us to the energy stability
\begin{itemize}
\item  For the dissipative scheme $\mathcal{D}2$
\begin{align*}
&((u_h)_t, u_h)_{I} + \mathcal{N}^{d}(u_h,u_h) + \mathcal{L}^-(p_h,p_h)+ \mathcal{L}^-((r_h)_t,(r_h)_t)-\mathcal{L}^+(s_h,s_h)-\mathcal{L}^+((u_h)_t,(u_h)_t)  \\
&+\mathcal{L}^-((r_h)_t,p_h)+ \mathcal{L}^-(p_h,(r_h)_t) -\mathcal{L}^+((u_h)_t, s_h)-\mathcal{L}^+( s_h,(u_h)_t)\\
& = ((u_h)_t, u_h)_{I} + \mathcal{N}^{d}(u_h,u_h) + \frac{1}{2}\sum\limits_{j=1}^{N}\big( (\jump{(r_h)_t}+\jump{p_h})^2_{j+\frac{1}{2}} +(\jump{(u_h)_t}+\jump{s_h})^2_{j+\frac{1}{2}}\big) =0;
\end{align*}
\item  For the conservative scheme $\mathcal{C}2$
\begin{align*}
&((u_h)_t, u_h)_{I} + \mathcal{N}^{c}(u_h,u_h) + \mathcal{L}^c(p_h,p_h)+ \mathcal{L}^c((r_h)_t,(r_h)_t)-\mathcal{L}^c(s_h,s_h)-\mathcal{L}^c((u_h)_t,(u_h)_t)  \\
&+\mathcal{L}^c((r_h)_t,p_h)+ \mathcal{L}^c(p_h,(r_h)_t) -\mathcal{L}^c((u_h)_t, s_h)-\mathcal{L}^c( s_h,(u_h)_t) = ((u_h)_t, u_h)_{I} = 0,
\end{align*}
\end{itemize}

 i.e. the results in \eqref{prop1_d1} and \eqref{prop1_c1}.
\end{proof}

%
%

\subsection{Error estimates}
In this section, we will provide error estimates of the LDG schemes in Section \ref{se:ldg2} for the sufficiently smooth exact solution  of the Fornberg-Whitham type equations with periodic boundary conditions. With the preparations in Section \ref{thm:prepare}, we directly give the theorem of the error estimates.

\subsubsection{The main error estimate results}
\begin{theorem}\label{thm2}
 It is assumed that the Fornberg-Whitham type equations \eqref{eqn:FW1}  with periodic boundary condition has a sufficiently smooth exact solution $u$. The numerical solution $u_h$ satisfies the semi-discrete LDG scheme \eqref{scheme:FW1DG} and \eqref{scheme:FW2DG}. For regular partitions of $I = (a, b)$, and the finite element
space $V^k_h$, there hold the following error estimates for small enough $h$,
\begin{itemize}
\item {Scheme $\mathcal{D}2 $}
\begin{align}\label{eqn:thm2d}
 \begin{cases}&\norm{ u - u_h}_{L^2(I)} \leq Ch^{k+1}, \quad \text{if p is odd},\\
                                                      &\norm{ u - u_h}_{L^2(I)} \leq Ch^{k+\frac{1}{2}}, \quad \text{if p is even}.
                                                 \end{cases}, \ k \geq 1
                                                 \end{align}

\item {Scheme $\mathcal{C}2 $}
\begin{align}  \norm{ u - u_h}_{L^2(I)} \leq Ch^{k},  \ k \geq 2 \label{eqn:thm2c}
\end{align}

\end{itemize}
where the integer $p$ is in the nonlinear term $f(u) = \frac{1}{p}u^p$. The constant $C$ depends on the final time $T$, $k$, $\norm{u}_{k+2}$ and the bounds of derivatives up to second order of the nonlinear term $f(u)$. Here, $\norm{u}_{k+2}$ is the maximum of the standard Sobolev $k+2$ norm over $[0, T]$.
\end{theorem}

\subsubsection{The error equation}
Combining the error equations with different test functions, we define the bilinear form ${\bar{\mathcal{B}}}_j$ as
\begin{align*}
&\bar{\mathcal{B}}_j(u-u_h, v-v_h, w-w_h, p-p_h, s-s_h; \bm{\phi}, \bm{\psi}, \bm{\varphi},\bm{\vartheta},\sigma) \label{eqn:B_bilinear2}  \\
& = ((u - u_h)_t,\bm{\phi})_{I_j} - ((w-w_h)_t , \bm{\phi})_{I_j} -  \mathcal{L}_j^-(r_t- \widehat{(r_h)}_t, \bm{\phi})
+((r - r_h)_t , \bm{\psi})_{I_j} - \mathcal{L}_j^+(u_t -\widehat{(u_h)}_t, \bm{\psi})\notag \\
&+  ((w -w_h)_t,\bm{\varphi})_{I_j}  + (s-s_h ,\bm{\varphi})_{I_j} - \mathcal{L}_j^-( p-\widehat{p_h}, \bm{\varphi})  + (p-p_h , \bm{\vartheta})_{I_j} - \mathcal{L}_j^+(s-\widehat{s_h}, \bm{\vartheta})  , \notag\\
& + (u-u_h,\bm{\vartheta})_{I_j} + (s-s_h , \sigma)_{I_j} -  <f(u) - \widehat{f(u_h)}, \sigma>_{I_j} + (f(u)-f(u_h), \sigma_x)_{I_j}.\notag
\end{align*}
 After applying summation over all cells $I_j$, the error equation is expressed by
\begin{align*}
\sum\limits_{j=1}^{N} \bar{\mathcal{B}}_j(u-u_h, r-r_h, p-p_h, s-s_h; \bm{\phi}, \bm{\psi}, \bm{\varphi},\bm{\vartheta},\sigma) = \sum\limits_{j=1}^{N}\mathcal{H}_j(f;u,u_h,\sigma).
\end{align*}
In the same way as Section \ref{error_eqution1}, we define
\begin{equation*} \label{errd2}
\begin{split}
&\xi^u = \mathcal{P}^+u-u_h, \ \eta^u = \mathcal{P}^+u - u, \ \xi^r = \mathcal{P}^-r-r_h,\ \eta^r = \mathcal{P}^-r - r,\\
&\xi^s = \mathcal{P}^+s-s_h, \ \eta^s = \mathcal{P}^+s - s, \ \xi^p = \mathcal{P}^-p-p_h,\ \eta^p = \mathcal{P}^-p - p,
\end{split}
\end{equation*}
for the dissipative scheme $\mathcal{D}2$. The Gauss-Radau projections are changed into the standard $L^2$ projections for
the conservative scheme $\mathcal{C}2$
\begin{equation*}
\begin{split}
&\xi^u = \mathcal{P}u-u_h, \ \eta^u = \mathcal{P}u - u, \ \xi^r = \mathcal{P}r-r_h,\ \eta^r = \mathcal{P}r - r,\\
&\xi^s = \mathcal{P}s-s_h, \ \eta^s = \mathcal{P}s - s, \ \xi^p = \mathcal{P}p-p_h,\ \eta^p = \mathcal{P}p - p,
\end{split}
\end{equation*}

Taking test functions
\begin{equation*}
\begin{split}
 \sigma = \xi^u, \ &\bm{\phi} = \bm{\varphi} = \bm{\xi^3}\triangleq (-\xi^p, -\xi^r_t,\xi^u_t,\xi^s), \ \bm{\vartheta} = \bm{\psi} =\bm{\xi^4}\triangleq (\xi^p, \xi^r_t,\xi^u_t,\xi^s),
 \end{split}
\end{equation*}
  we have the energy equality as
\begin{align*}
&\sum\limits_{j=1}^{N} \bar{\mathcal{B}}_j(\xi^u- \eta^u, \xi^r- \eta^r, \xi^p- \eta^p,\xi^s- \eta^s ;  \bm{\xi^3} ,\bm{\xi^4},\bm{\xi^3},\bm{\xi^4},\xi^u)=\sum\limits_{j=1}^{N}\mathcal{H}_j(f;u,u_h,\xi^u).
\end{align*}

\subsubsection{The proof of the main results in Theorem \ref{thm2}}

The relevant estimates for $\mathcal{H}_j$ have already been
stated in Section \ref{main_results1}. We just give the estimate results for $\bar{\mathcal{B}}_j$.
{
\begin{itemize}
\item {\bf Estimates for the linear terms}
\end{itemize}}
\begin{lemma}\label{lemma:another_energy2}
The following energy equality holds,
\begin{equation}\label{eqn:another_energy2}
 \norm{s_h + (u_h)_t }^2_{L^2(I)} + \norm{p_h + (r_h)_t}^2_{L^2(I)} + (u_h, p_h + (r_h)_t)_I = 0 .
\end{equation}
\end{lemma}
\begin{proof}The proof is provided in Appendix \ref{proofanother_energy2}.
\end{proof}

\begin{lemma}\label{lemma:b2}
For the bilinear forms $\bar{\mathcal{B}}_j$, the following equations hold by projection properties
\begin{itemize}
\item   Scheme  $\mathcal{D}2$
\begin{align}
&\sum\limits_{j=1}^{N} \bar{\mathcal{B}}_j(\xi^u- \eta^u, \xi^r- \eta^r, \xi^p- \eta^p,\xi^s- \eta^s ; \bm{\xi^3} ,\bm{\xi^4},\bm{\xi^3},\bm{\xi^4},\xi^u) \notag \\
& = (\xi_t^u,\xi^u)_I  + \norm{\xi^s + \xi^u_t}^2_{L^2(I)} + \norm{\xi^p + \xi^r_t}^2_{L^2(I)} + \sum\limits_{j=1}^{N} \frac{1}{2}(\jump{\xi^p} + \jump{\xi^r_t})^2_{j+\frac{1}{2}}  \notag  \\
 &\quad \ + \sum\limits_{j=1}^{N}\frac{1}{2}(\jump{\xi^s} + \jump{\xi^u_t})^2_{j+\frac{1}{2}} +(\xi^u, \xi^p + \xi^r_t)_I - (\eta^u + \eta^p + \eta^s + \eta_t^u + \eta^r_t,\xi^u_t + \xi^s)_I  \notag \\
 &\quad \ -(\eta^u +\eta^p -\eta^s - \eta^u_t + \eta_t^r  ,\xi^r_t +\xi^p)_I - (\eta^s,\xi^u)_I ; \label{eqn:B_estimate22d}
\end{align}
\item   Scheme  $\mathcal{C}2$
\begin{align}
&\sum\limits_{j=1}^{N} \bar{\mathcal{B}}_j(\xi^u- \eta^u, \xi^r- \eta^r, \xi^p- \eta^p,\xi^s- \eta^s ; \bm{\xi^3} ,\bm{\xi^4},\bm{\xi^3},\bm{\xi^4},\xi^u) \notag \\
& = ({\xi}_t^u,{\xi}^u)_I  + \norm{{\xi}^s + {\xi}^u_t}^2_{L^2(I)} + \norm{{\xi}^p + {\xi}^r_t}^2_{L^2(I)}
- \sum\limits_{j=1}^{N} <\widehat{{\eta}^r_t} + \widehat{{\eta}^p} - \widehat{{\eta}^s} - \widehat{{\eta}^u_t},{\xi}^r_t + {\xi}^p>_{I_j}  \notag \\
&\quad \ +({\xi}^u, {\xi}^p + {\xi}^r_t)_I + \sum\limits_{j=1}^{N} <\widehat{{\eta}^r_t} +\widehat{{\eta}^p} +\widehat{{\eta}^s} + \widehat{{\eta}^u_t} ,{\xi}^u_t +{\xi}^s>_{I_j}.\label{eqn:B_estimate22c}
\end{align}
\end{itemize}
\end{lemma}
\begin{proof}The proof is provided in Appendix \ref{proofb2}.
\end{proof}
{
\begin{itemize}
\item  {\bf Final error estimates in Theorem \ref{thm2}}
\end{itemize}}

After the above preparations, the proof of Theorem \ref{thm2} can be provided. We still pay our attention to the approximation error $\xi^u$. After the triangle inequality  and  interpolation property \eqref{projection error}, we complete the proof for Theorem \ref{thm2}.
\begin{itemize}
\item [] \textbf{Estimates in \eqref{eqn:thm2d}}:
\end{itemize}
Combining \eqref{eqn:B_estimate22d} and \eqref{eqn:H_estimate2_podd}, the error estimate for the
dissipative scheme $\mathcal{D}2$ is given by,
\begin{align*}
&(\xi^u_t, \xi^u)_{I} + \norm{\xi^s + \xi^u_t}^2_{L^2(I)} + \norm{\xi^p + \xi^r_t}^2_{L^2(I)} \leq
-(\xi^u, \xi^p + \xi^r_t)_I + Ch^{2k+\mu} + C_*\norm{\xi^u}^2_{L^2(I)}\\
 &+ (\eta^u + \eta^p + \eta^s + \eta_t^u + \eta^r_t,\xi^u_t + \xi^s)_I +(\eta^u +\eta^p -\eta^s  - \eta^u_t + \eta_t^r ,\xi^r_t +\xi^p)_I + (\eta^s,\xi^u)_I \\
& + C_*\norm{\xi^u}^2_{L^2(I)} + Ch^{2k+\mu}  + Ch^{-\frac{3}{2}}\norm{\xi^u}^3_{L^2(I)}
\end{align*}
where $\mu = 1$ for even $p$,  and $\mu = 2$ for odd $p$.
\begin{itemize}
\item [] \textbf{Estimates in \eqref{eqn:thm2c}}:
\end{itemize}
Together with the estimates \eqref{eqn:B_estimate22c} and \eqref{eqn:hc1},  we have the following inequality for the
conservative scheme $\mathcal{C}2$,
\begin{align*}
& ({\xi}_t^u,{\xi}^u)_I  + \norm{{\xi}^s + {\xi}^u_t}^2_{L^2(I)} + \norm{{\xi}^p + {\xi}^r_t}^2_{L^2(I)} \notag \\
&\leq -({\xi}^u, {\xi}^p + {\xi}^r_t)_I \\
&+ \sum\limits_{j=1}^{N}\big( <\widehat{{\eta}^r_t} + \widehat{{\eta}^p} - \widehat{{\eta}^s} - \widehat{{\eta}^u_t},{\xi}^r_t + {\xi}^p>_{I_j} - <\widehat{{\eta}^r_t} +\widehat{{\eta}^p} +\widehat{{\eta}^s} + \widehat{{\eta}^u_t} ,{\xi}^u_t +{\xi}^s>_{I_j}\big) \notag\\
&+ C_*\norm{\xi^u}^2_{L^2(I)} + Ch^{2k} + Ch^{-\frac{3}{2}}\norm{\xi^u}^3_{L^2(I)}.
\end{align*}

The Young's inequality and inverse inequality (ii) in \eqref{eqn:inverse inequality}  imply
{
\begin{align}
 \frac{1}{2}\frac{d}{dt}\norm{\xi^u}^2_{L^2(I)}   \leq C_*\norm{\xi^u}^2_{L^2{(I)}} + Ch^{2k+\tilde{\mu}}+
Ch^{-\frac{3}{2}}\norm{\xi^u}^3_{L^2(I)} ,
\end{align}
}
where $\tilde{\mu} = \mu$ for Scheme $\mathcal{D}2$, and $\tilde{\mu} = 0$ for Scheme $\mathcal{C}2$.  Through Lemma 2.9,  we can
straightforwardly obtain the results of Theorem \ref{thm2}.


\section{Numerical experiments}\label{experiments}
In this section, we present several numerical experiments to illustrate the capability of our numerical schemes, including the errors, convergence rate tables, and some plots of numerical solutions. For the time discretization, we adopt the explicit third-order methods in \cite{1988_Shu_JCP, Gottlieb_2001_SIAM}. {Under the situation of degree $k$ $(k\geq 2)$ polynomial approximation space, we take the time step as $\Delta t = \alpha{\Delta x}^{(k+1)/3}, \alpha = 0.1$ to unify orders of temporal and spatial discretization.} Without a specific explanation, the periodic boundary condition and uniform meshes are used. {All examples were performed on a Windows desktop system using an Intel Core i5 processor, and programmed in Intel Visual Fortran.}

\begin{example} \label{ex:smooth}$\bf{ Smooth\ solution}$

For the Fornberg-Whitham type equation with $p=3$, we test the accuracy of our numerical schemes by a tectonic smooth traveling solution
\begin{align}\label{smooth}
u(x,t) = \sin(x-t), \ x\in [0,2\pi].
\end{align}
Here a source term is needed to add in equation \eqref{eqn:GFW} to make sure the equation holds. The $L^2$ and $L^{\infty}$
errors and orders of accuracy for the four LDG schemes are contained in Tables \ref{tab: acc_test1} and \ref{tab: acc_test2},
respectively. The $L^2$ and $L^{\infty}$ errors for the dissipative schemes $\mathcal{D}1$ and $\mathcal{D}2$ are similar, so
do the conservative schemes. The convergence rates of both norms for the dissipative schemes are $(k+1)$-$th$ order which are
both optimal. Notably, these results are identical to the theoretical proof. Even for even $p$, the numerical tests still
show the optimal error order of accuracy.
Owing to the choices of the central fluxes, the conservative schemes $\mathcal{C}1$ and $\mathcal{C}2$
have $(k+1)$-$th$ order of accuracy for even $k$, and only $k$-$th$ order of accuracy for odd $k$. All above conclusion is
on the basis of uniform meshes. {A slight perturbation of uniform meshes does
not affect the error order of accuracy.}

\begin{table}[!htp]
\centering
\begin{tabular}{|c|c|cccc|cccc|}
  \hline
           & N       &$ \norm{u-u_h}_{L^2}$    & order   &$ \norm{u-u_h}_\infty$   & order
                     &$ \norm{u-u_h}_{L^2}$    & order   &$ \norm{u-u_h}_\infty$   & order\\\hline
          &           &  \multicolumn{4}{c|}{Scheme $\mathcal{D}1$} & \multicolumn{4}{c|}{Scheme $\mathcal{C}1$}\\\hline

$P^2$   &  20  &4.23E-05 &-- &1.46E-04 &--
&3.49E-05 &-- &1.24E-04 &--\\

&  40  &6.24E-06 &2.76 &2.60E-05 &2.49
&4.26E-06 &3.03 &1.57E-05 &2.98\\

&  80  &8.54E-07 &2.87 &4.93E-06 &2.40
&5.30E-07 &3.01 &1.98E-06 &2.99\\

& 160  &1.12E-07 &2.93 &7.09E-07 &2.80
&6.61E-08 &3.00 &2.49E-07 &2.99\\\hline

$P^3$ &  20  &9.68E-07 &-- &4.08E-06 &--
&2.68E-06 &-- &1.44E-05 &--\\

&  40  &7.29E-08 &3.73 &3.28E-07 &3.64
&3.41E-07 &2.97 &1.57E-06 &3.19\\

&  80  &4.18E-09 &4.13 &1.91E-08 &4.11
&4.28E-08 &3.00 &2.16E-07 &2.86\\

& 160  &2.71E-10 &3.95 &1.21E-09 &3.98
&5.47E-09 &2.97 &2.61E-08 &3.05 \\\hline

\end{tabular}
\caption{\label{tab: acc_test1} Example \ref{ex:smooth}, accuracy test for smooth solution \eqref{smooth} at $T = 0.1$. }
\end{table}

\begin{table}[!htp]
\centering
\begin{tabular}{|c|c|cccc|cccc|}
  \hline
           & N       &$ \norm{u-u_h}_{L^2}$    & order   &$ \norm{u-u_h}_\infty$   & order
                     &$ \norm{u-u_h}_{L^2}$    & order   &$ \norm{u-u_h}_\infty$   & order\\\hline
          &           &  \multicolumn{4}{c|}{Scheme $\mathcal{D}2$ } & \multicolumn{4}{c|}{Scheme $\mathcal{C}2$}\\\hline

$P^2$   &  20  &4.83E-05 &-- &2.07E-04 &--
&3.66E-05 &-- &1.57E-04 &--\\

&  40  &6.78E-06 &2.83 &3.45E-05 &2.58
       &4.16E-06 &3.14 &1.45E-05 &3.44\\

&  80  &8.80E-07 &2.95 &5.47E-06 &2.66
&5.22E-07 &2.99 &1.79E-06 &3.02\\

& 160  &1.13E-07 &2.96 &7.38E-07 &2.89
&6.57E-08 &2.99 &2.36E-07 &2.92\\\hline

$P^3$&  20  &1.40E-06 &-- &5.74E-06 &--
        &3.43E-06 &-- &1.43E-05 &-- \\

&  40  &9.32E-08 &3.91 &4.00E-07 &3.84
&4.20E-07 &3.03 &1.99E-06 &2.85\\

&  80  &5.21E-09 &4.16 &2.29E-08 &4.13
&5.30E-08 &2.99 &2.33E-07 &3.09\\

& 160  &3.37E-10 &3.95 &1.43E-09 &4.00
&6.68E-09 &2.99 &2.89E-08 &3.01\\\hline

\end{tabular}
\caption{\label{tab: acc_test2} Example \ref{ex:smooth}, accuracy test for smooth solution \eqref{smooth} at $T = 0.1$. }
\end{table}

\end{example}

\begin{example}\label{ex:shock} $\bf{ Shock\ solutions}$

This example is devoted to test two shock solutions. The initial datums are given as follows
\begin{align}
&\label{data1}\text{data 1:} \ u(x,0) =  \cos(2\pi x + 0.5) + 1 \\
&\label{data2}\text{data 2:}  \ u(x,0) =  0.2\cos(2\pi x) + 0.1\cos(4\pi x) - 0.3\sin(6\pi x) + 0.5
\end{align}
where the computational domain is $[0,1]$. However, there is no exact solution for these two initial datums. {Compared with the results in \cite{Hormann_2018_arxiv}, our dissipative scheme $\mathcal{D}1$ and $\mathcal{D}2$ with a TVB limiter \cite{1989_Cockburn_JCP} can capture the shock without oscillation, see Figure \ref{fig:shock2}.} We further extend to the situation with parameter $p=4$ in Figure \ref{fig:shock1p4}. Because of the lack of dissipation for the nonlinear term $f(u)$, the conservative schemes fail to model these shock solutions.


\begin{figure}[!htp] \centering
		\subfigure[$t=0.1$] {
			\includegraphics[width=0.45\columnwidth]{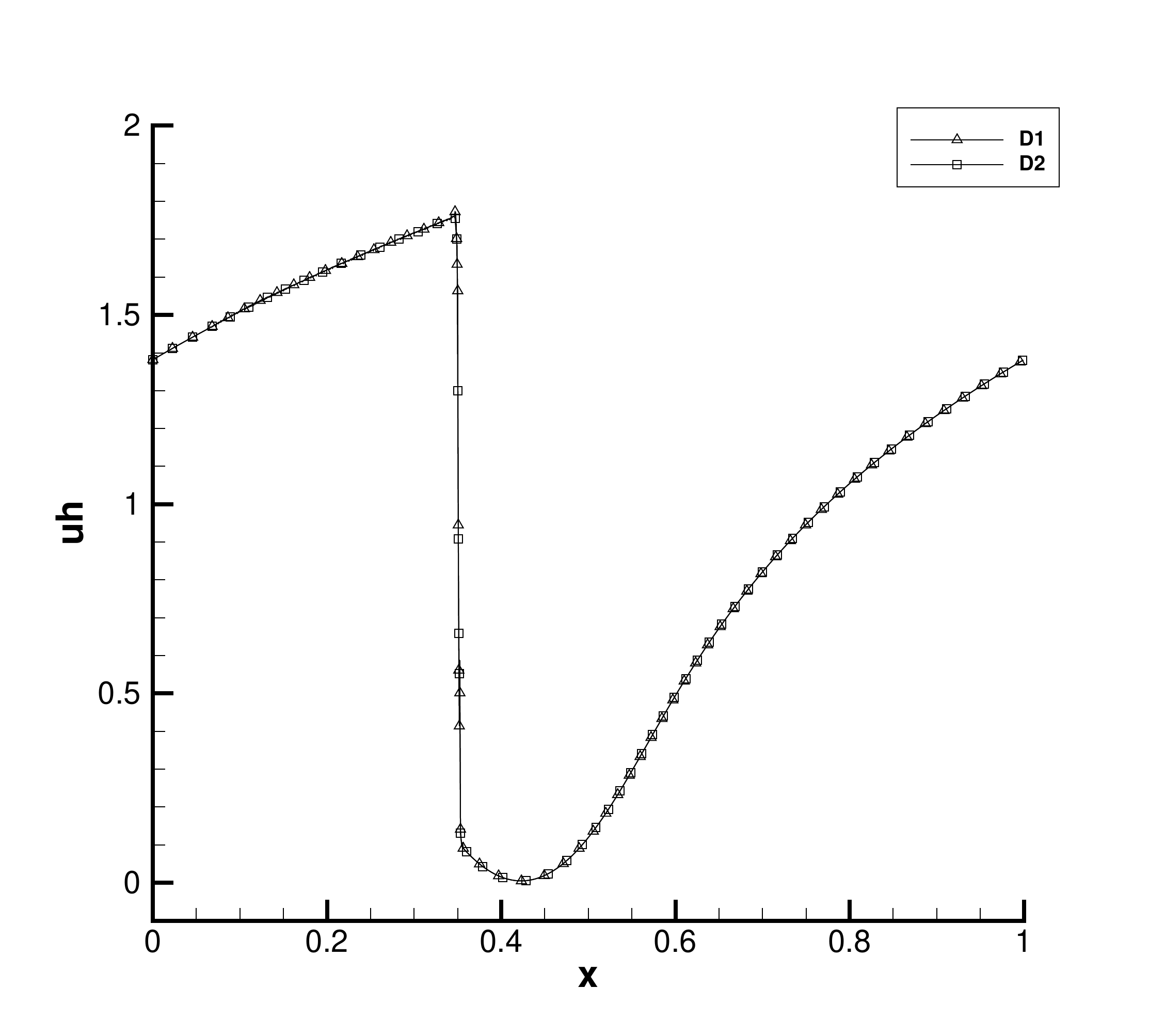}
		}
		\subfigure[$t=0.2$] {
			\includegraphics[width=0.45\columnwidth]{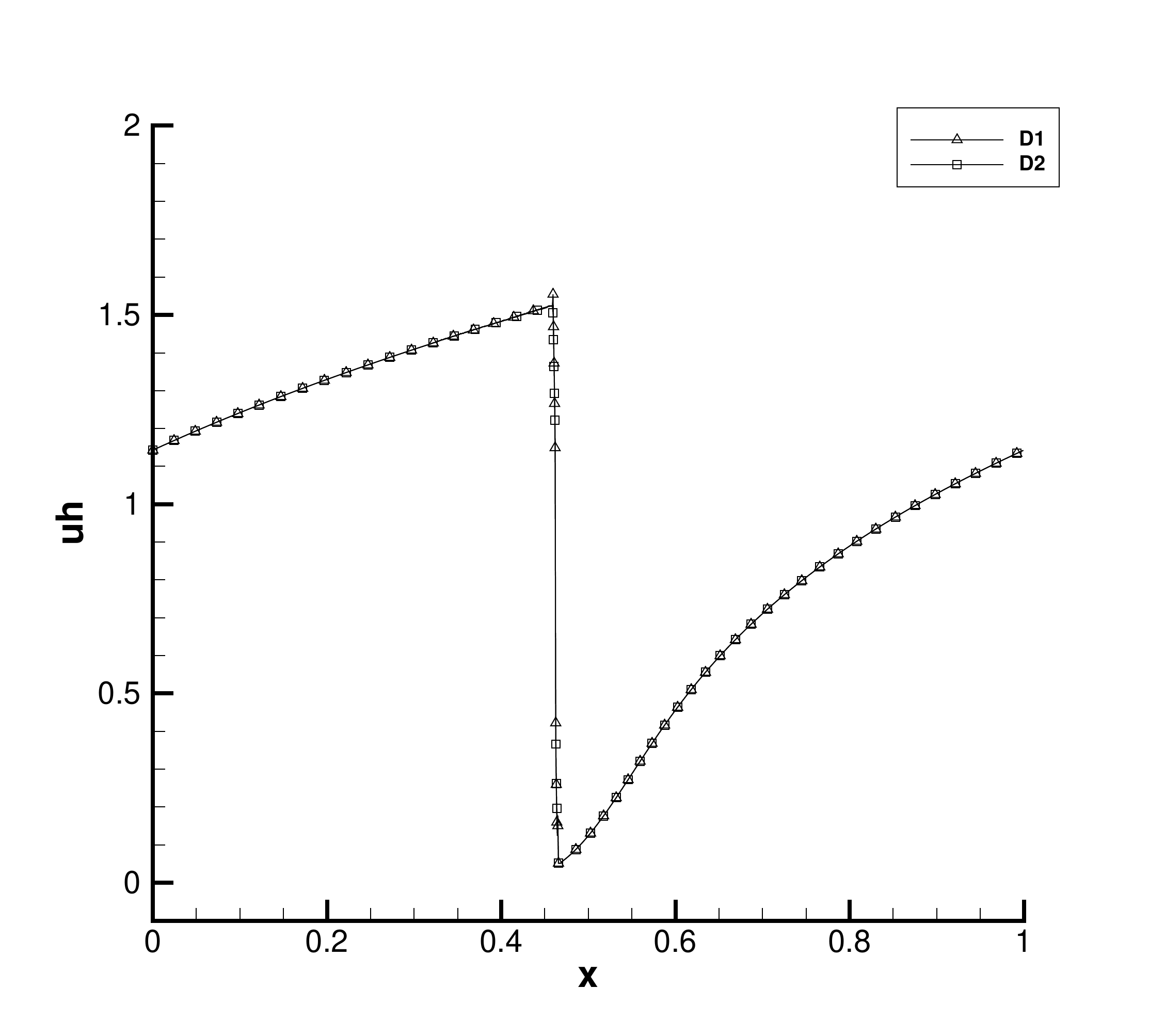}
		}
		\subfigure[$t=0.3$] {
			\includegraphics[width=0.45\columnwidth]{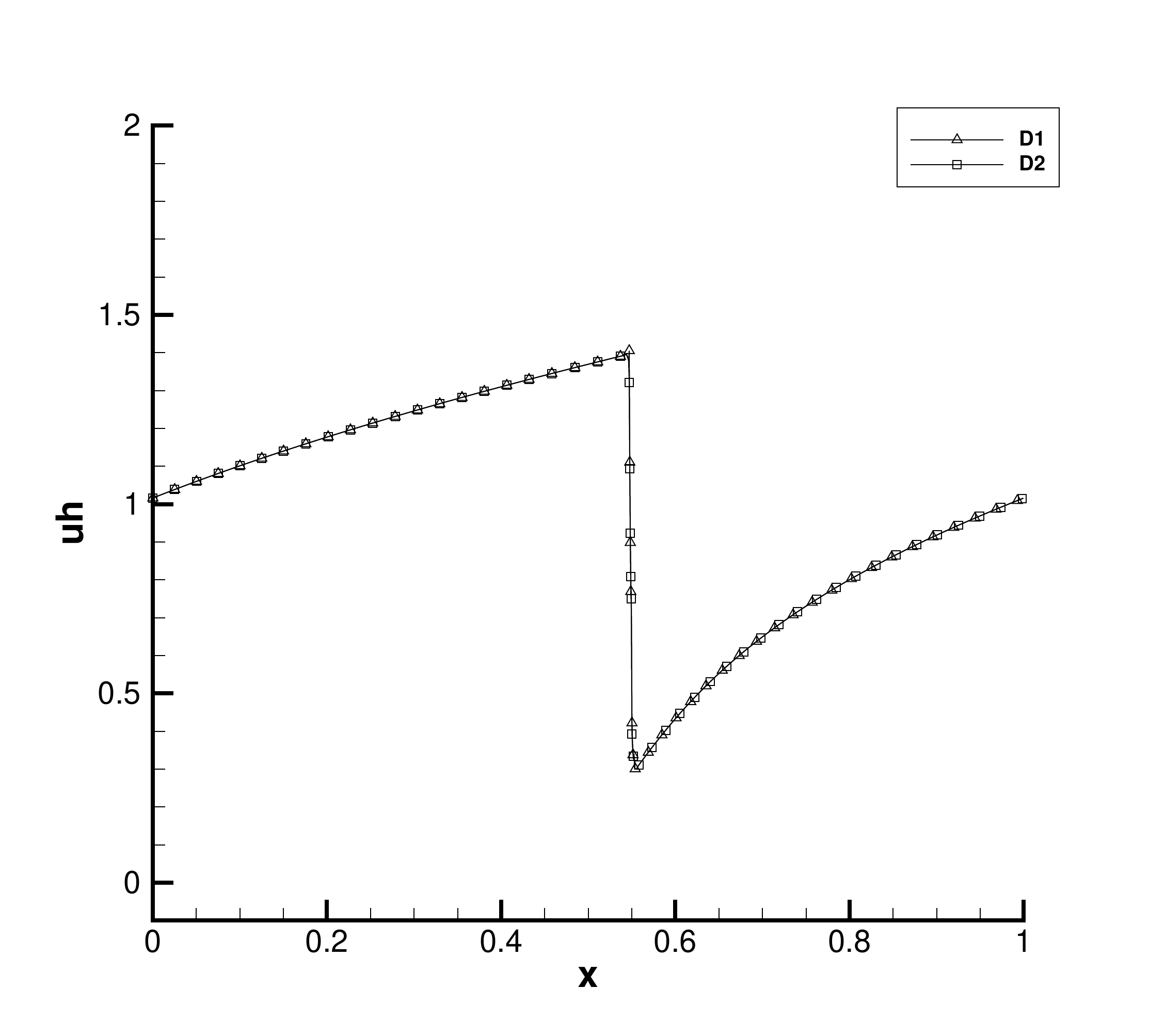}
		}
		\subfigure[$t=0.4$] {
			\includegraphics[width=0.45\columnwidth]{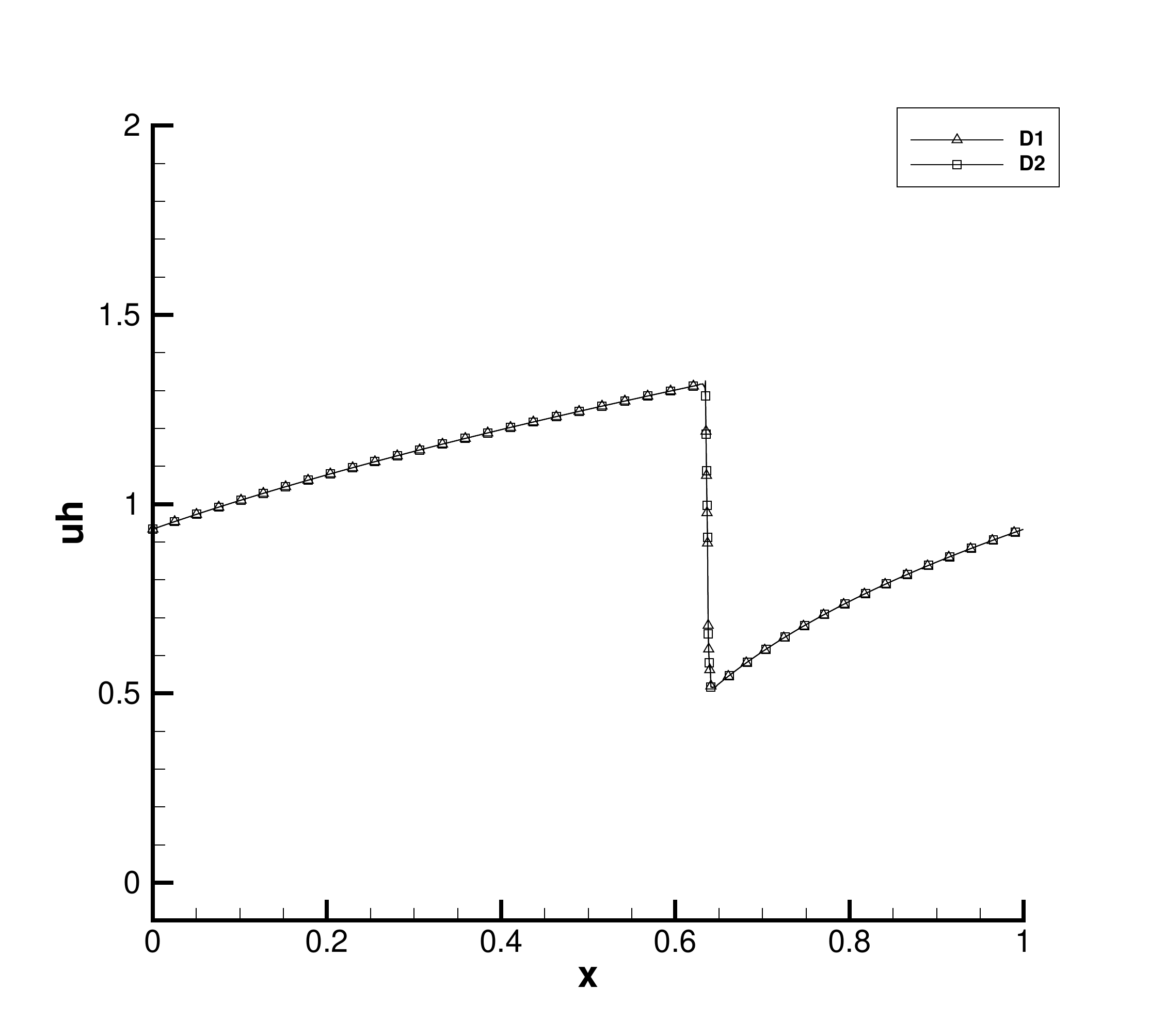}
		}
		\caption{  Example \ref{ex:shock}, different terminal time for initial data 1 \eqref{data1} in computational domain $[0,1]$, $p =4, N = 160, P^2$ elements. }
		\label{fig:shock1p4}
	\end{figure}

\begin{figure}[!htp] \centering
		\subfigure[$t=0.0$] {
			\includegraphics[width=0.45\columnwidth]{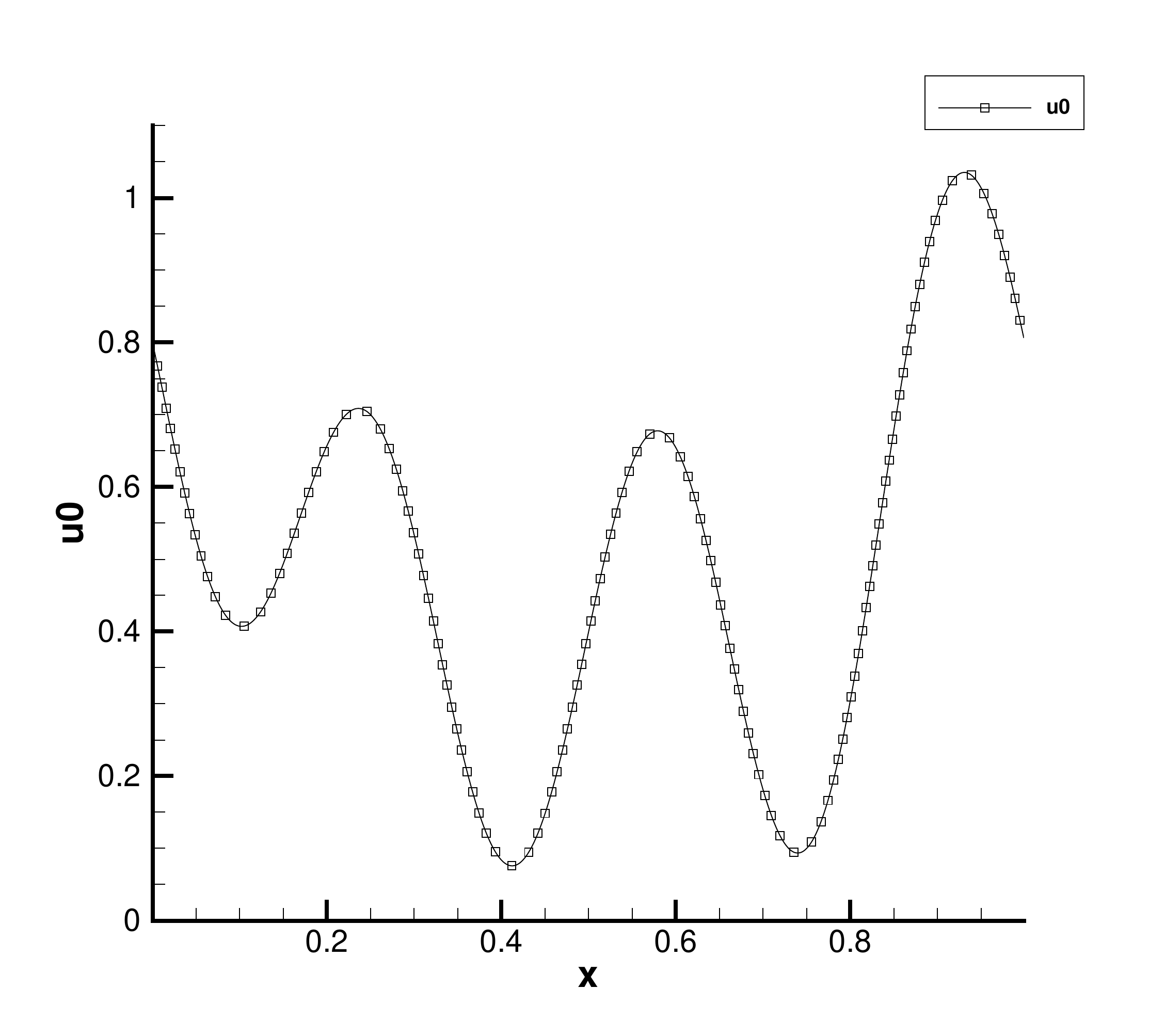}
		}
		\subfigure[$t=0.2$] {
			\includegraphics[width=0.45\columnwidth]{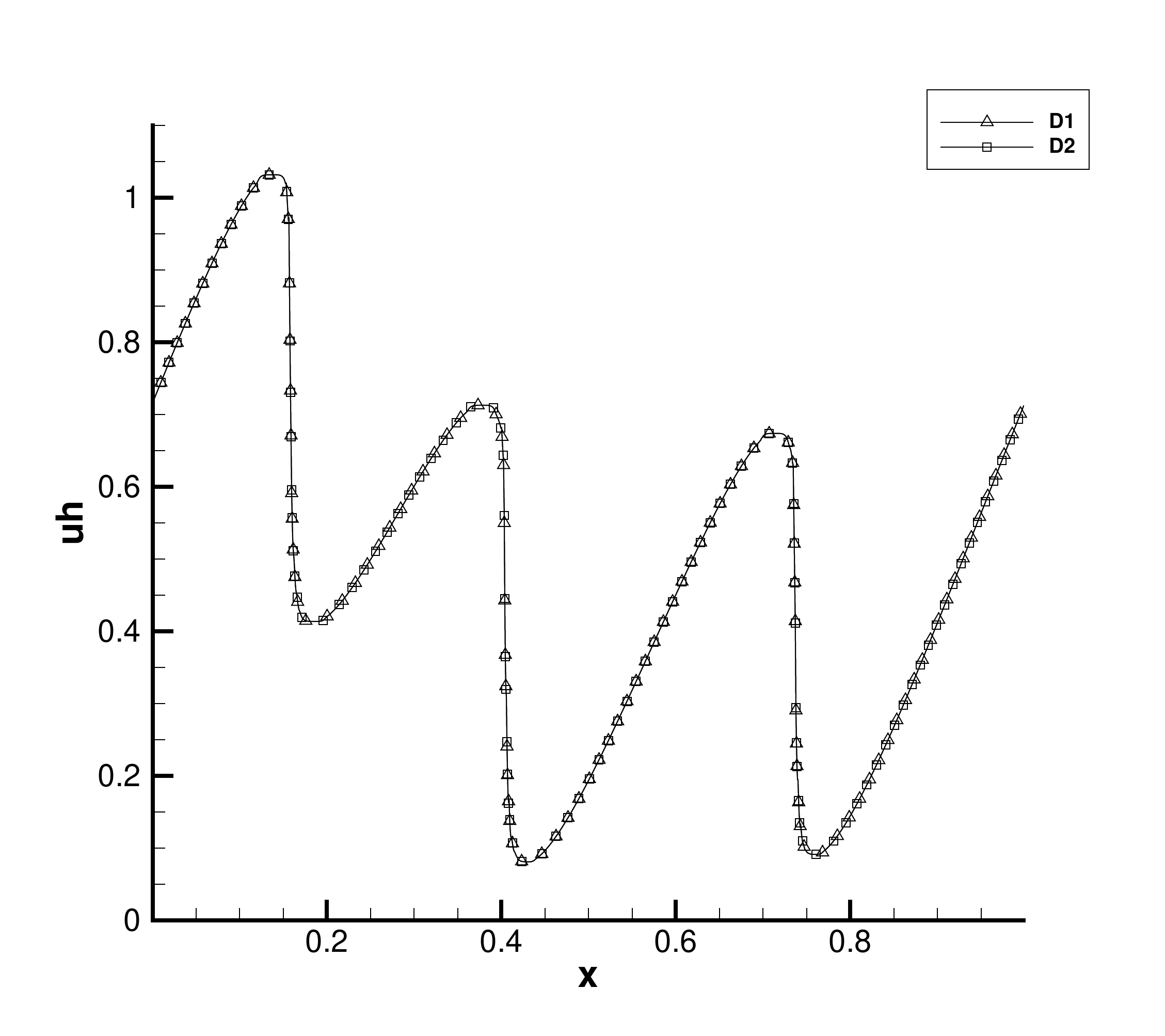}
		}
		\subfigure[$t=0.4$] {
			\includegraphics[width=0.45\columnwidth]{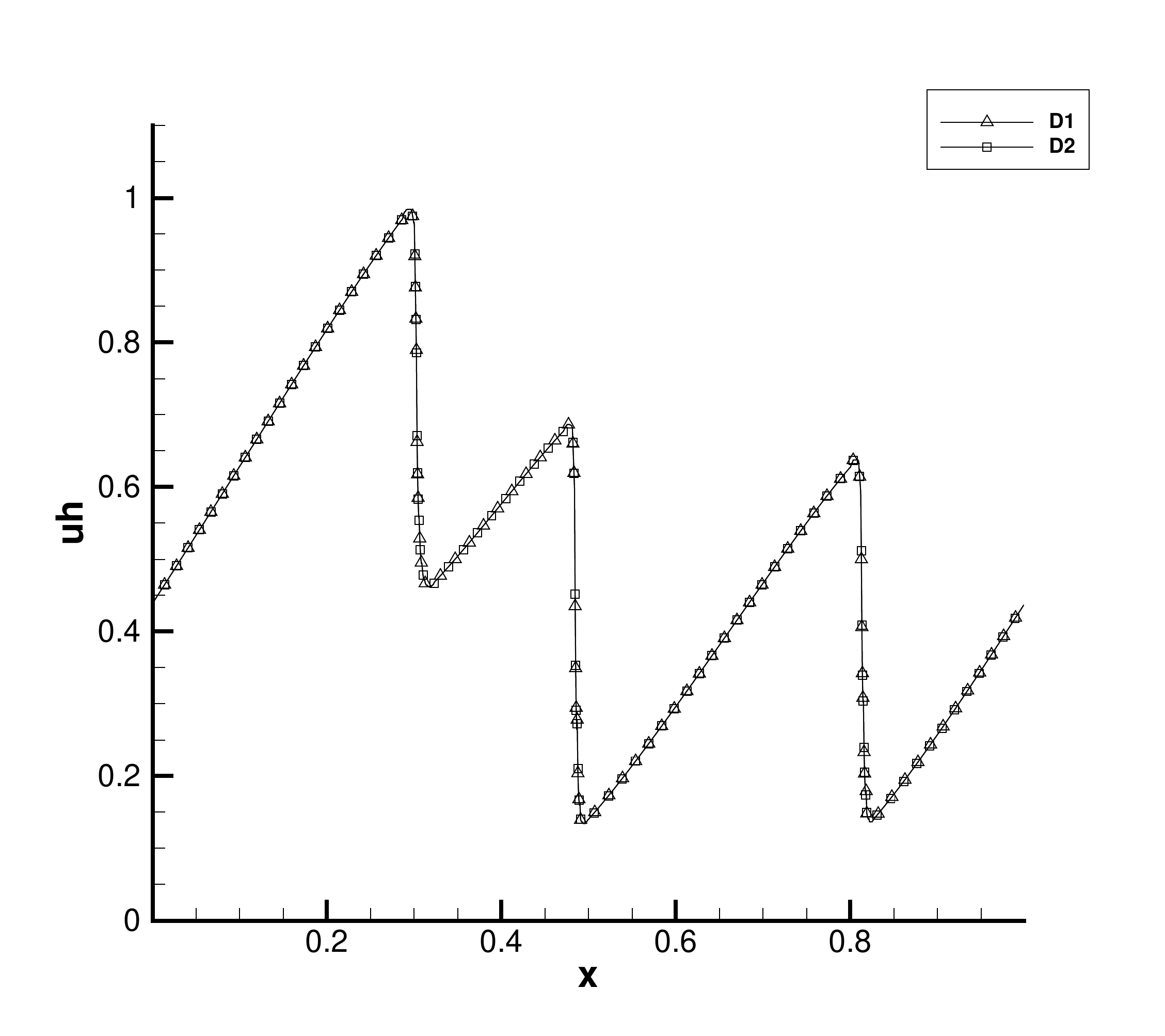}
		}
		\subfigure[$t=0.6$] {
			\includegraphics[width=0.45\columnwidth]{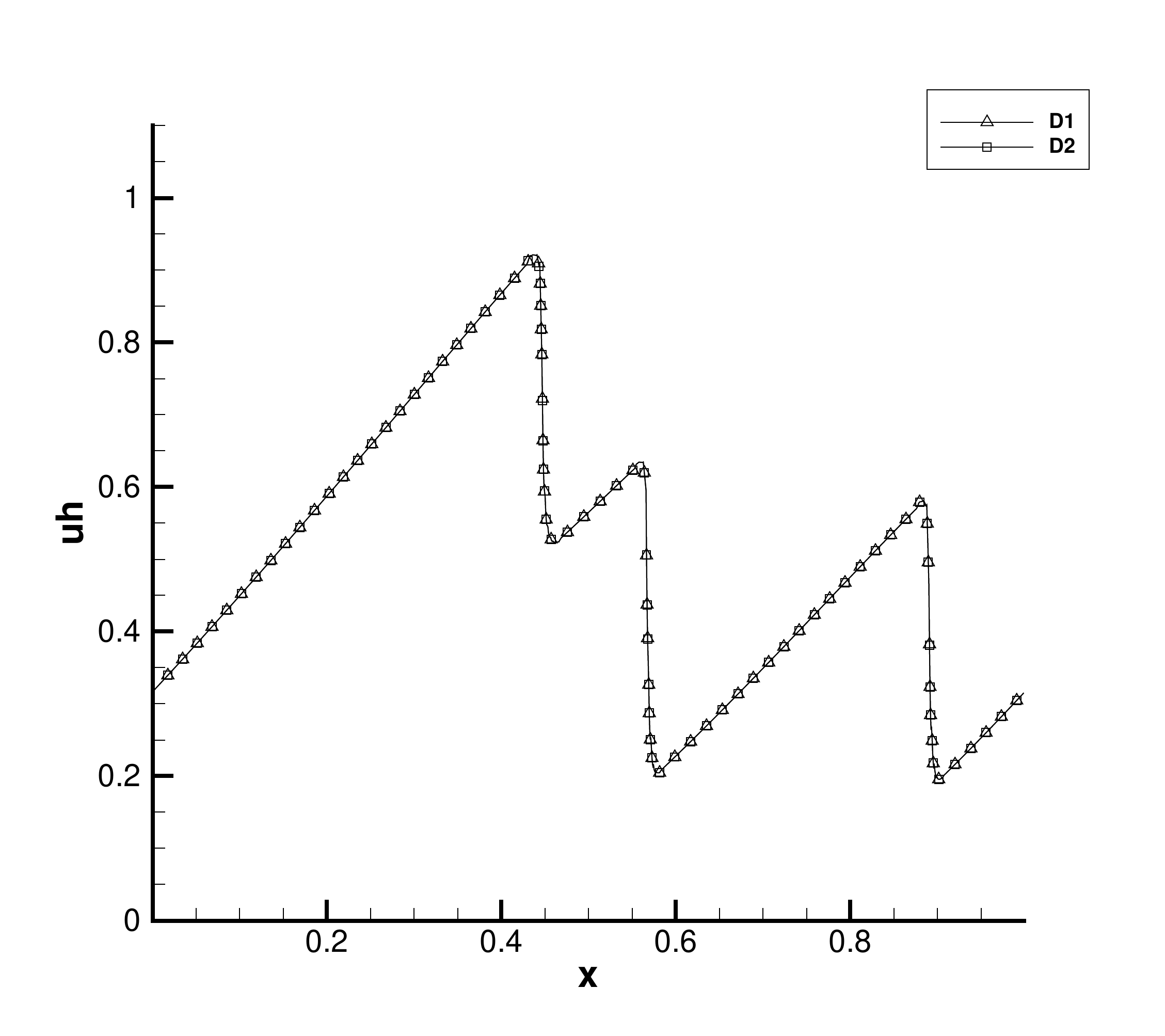}
		}
		\subfigure[$t=0.8$] {
			\includegraphics[width=0.45\columnwidth]{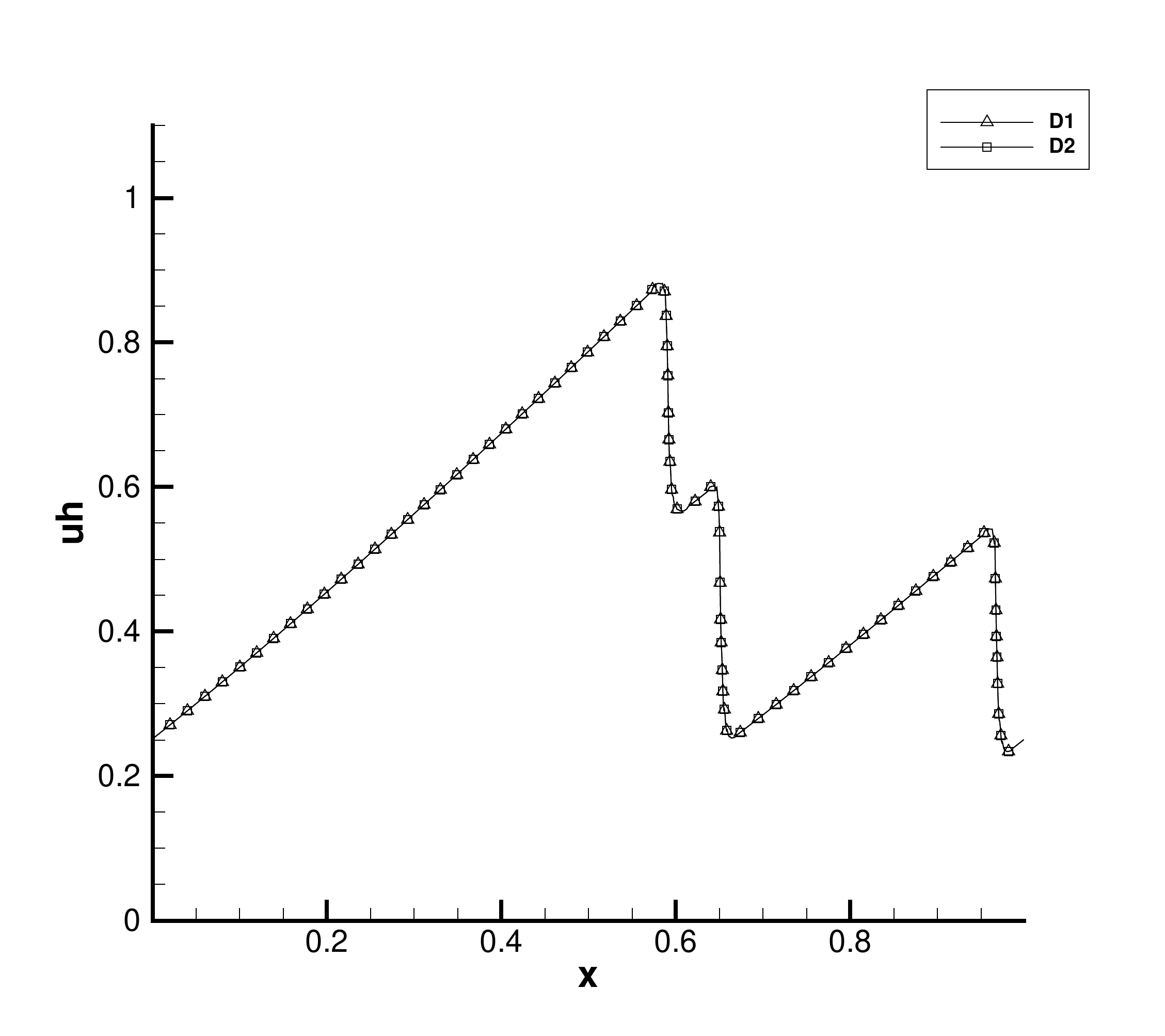}
		}
		\subfigure[$t=1.0$] {
			\includegraphics[width=0.45\columnwidth]{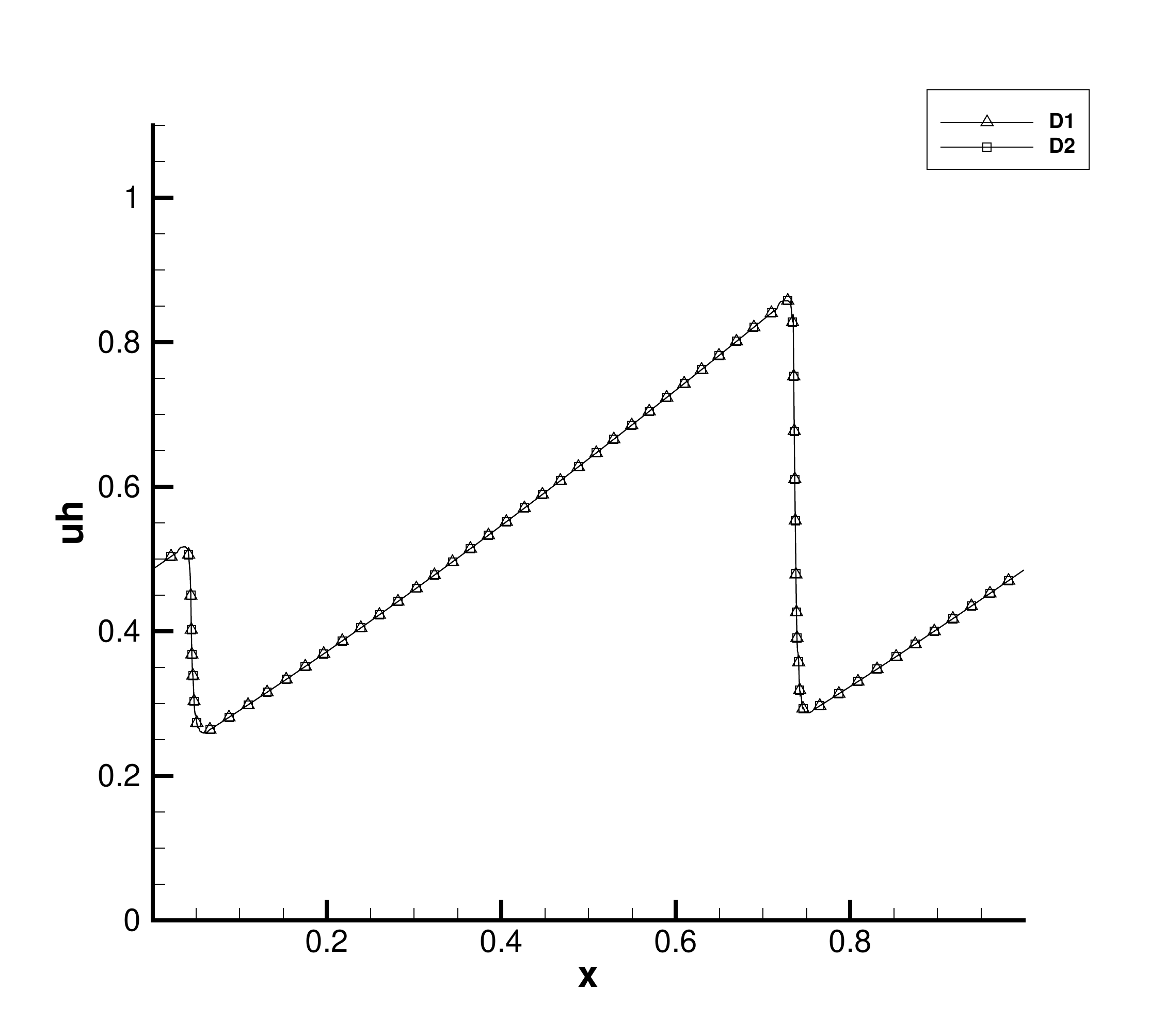}
		}
		\caption{  Example \ref{ex:shock}, different terminal time for initial data 2 \eqref{data2} in computational domain $[0,1]$, $p=2, N = 320, P^2$ elements. }
		\label{fig:shock2}
	\end{figure}

%

\end{example}

\begin{example}\label{ex:2soliton} $\bf{ Interaction \ of \ two \ solitons }$

The similarity between the KdV equation and the Fornberg-Whitham equation makes the interaction of solitons evolved
by the Fornberg-Whitham equation worth exploring. Using the initial condition of the
KdV equation in \cite{Cui_2007JCP} as
\begin{equation}\label{2soliton}
u(x,t) = 12\frac{\kappa_1^2e^{\theta_1} + \kappa_2^2e^{\theta_2}+2(\kappa_2-\kappa_1)^2e^{\theta_1 + \theta_2 }+ a^2(\kappa_2^2e^{\theta_1}+\kappa_1^2e^{\theta_2})e^{\theta_1 + \theta_2}}{(1+e^{\theta_1} + e^{\theta_2} + a^2e^{\theta_1+\theta_2})^2}
\end{equation}
where
\begin{equation*}
\begin{split}
&\kappa_1 = 0.4,\ \kappa_2 = 0.6, \ a^2 =\Big( \frac{\kappa_1-\kappa_2}{\kappa_1 + \kappa_2}\Big)^2,\\
&\theta_1 = \kappa_1x- \kappa_1^3t + 4,  \ \theta_2 = \kappa_2x- \kappa_2^3t + 15,
\end{split}
\end{equation*}
we model the interaction of two solitons in Figure \ref{fig:2soliton}. Here, the computational domain is set to $[-50,200]$. The process is similar to the KdV case: the two peakons travel from left to right, the speed of the taller one is larger than the shorter one. Finally, the taller one passes the shorter one and then they go further along with the opposite directions. The numerical solutions in Figure \ref{fig:2soliton} is consistent with the results in \cite{Liu2015_NM}.

\begin{figure}[!htp] \centering
		\subfigure[$t = 0.0$] {
			\includegraphics[width=0.45\columnwidth]{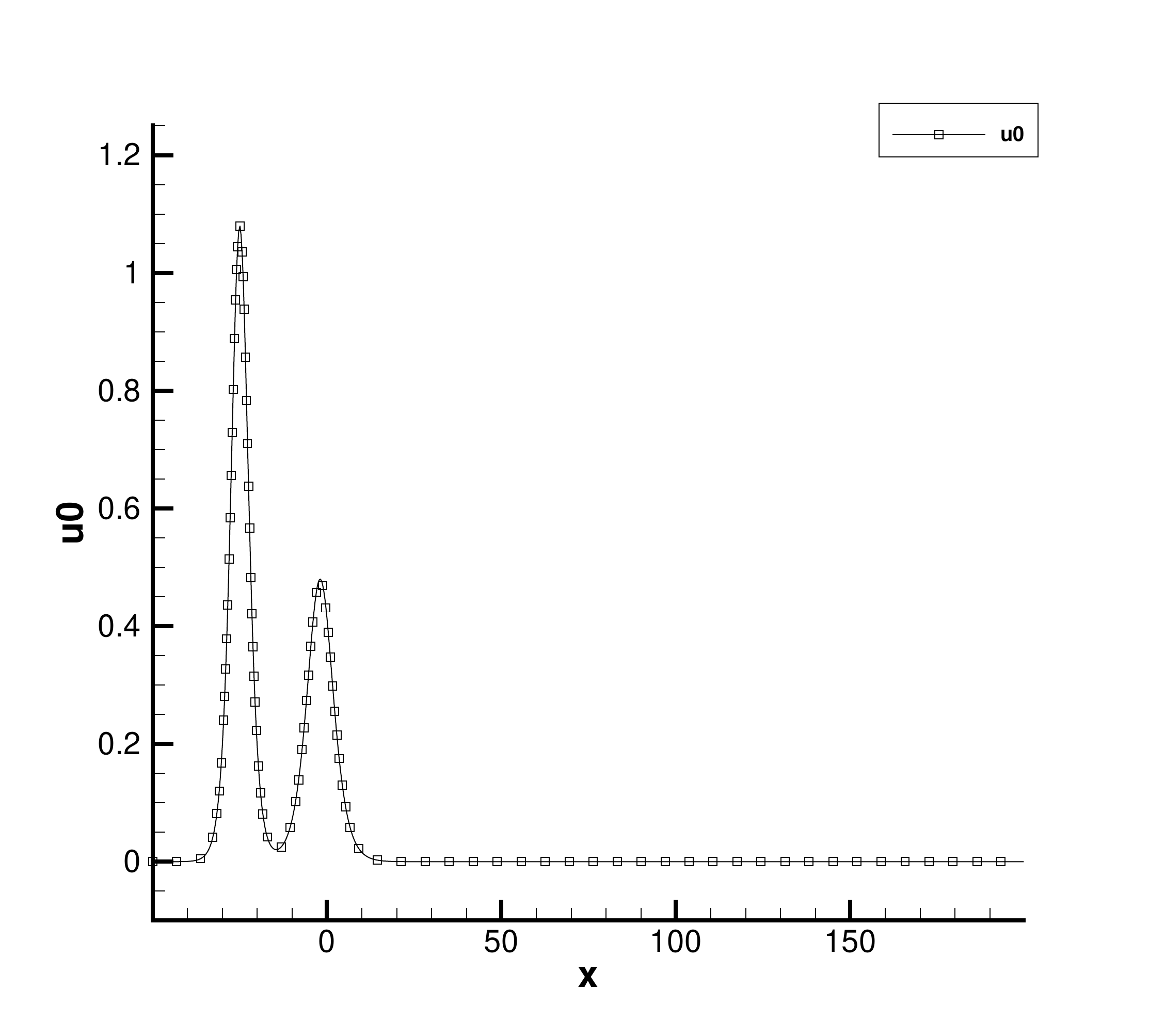}
		}
		\subfigure[$t = 40.0$] {
			\includegraphics[width=0.45\columnwidth]{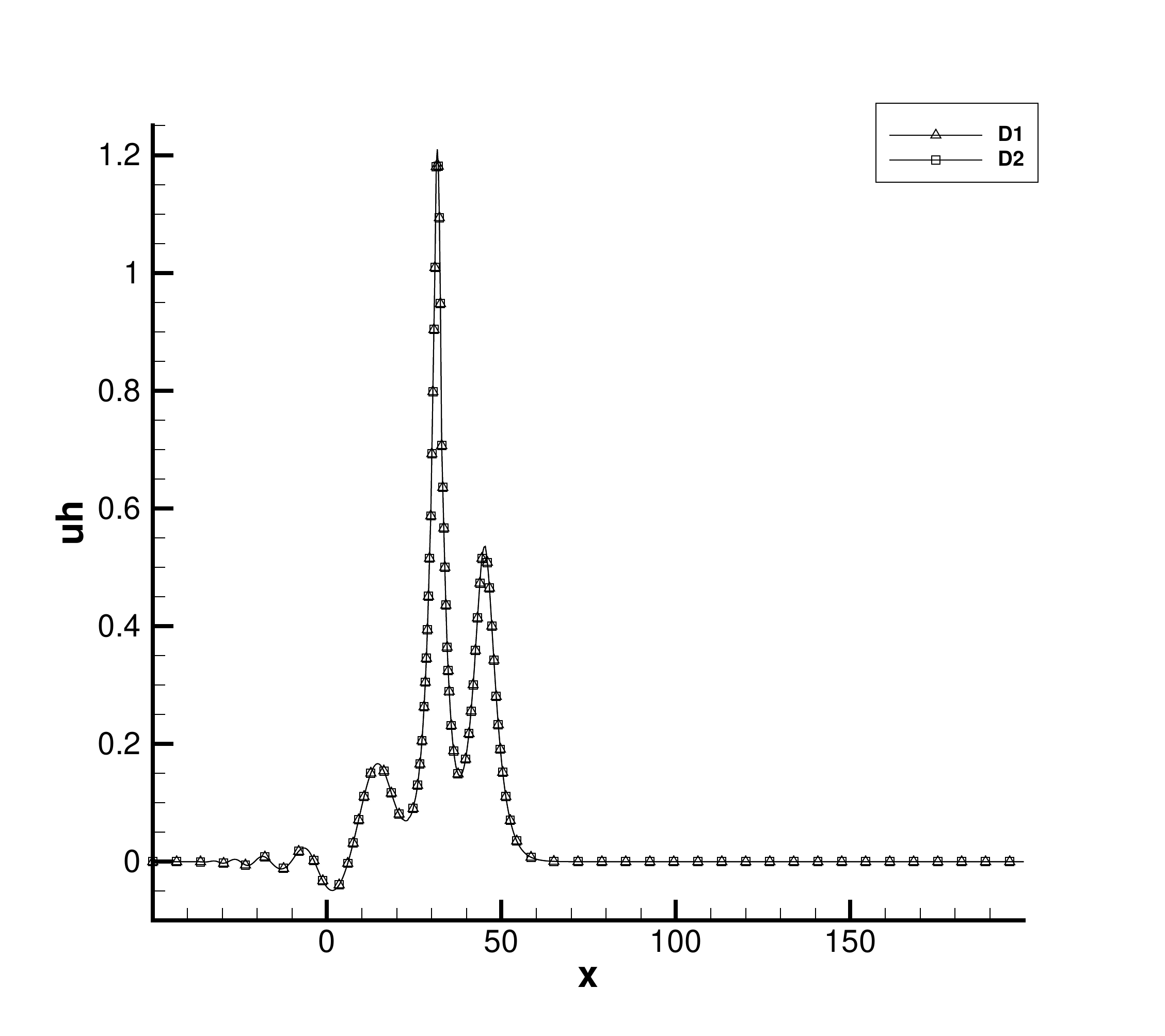}
		}
		\subfigure[$t = 80.0$] {
			\includegraphics[width=0.45\columnwidth]{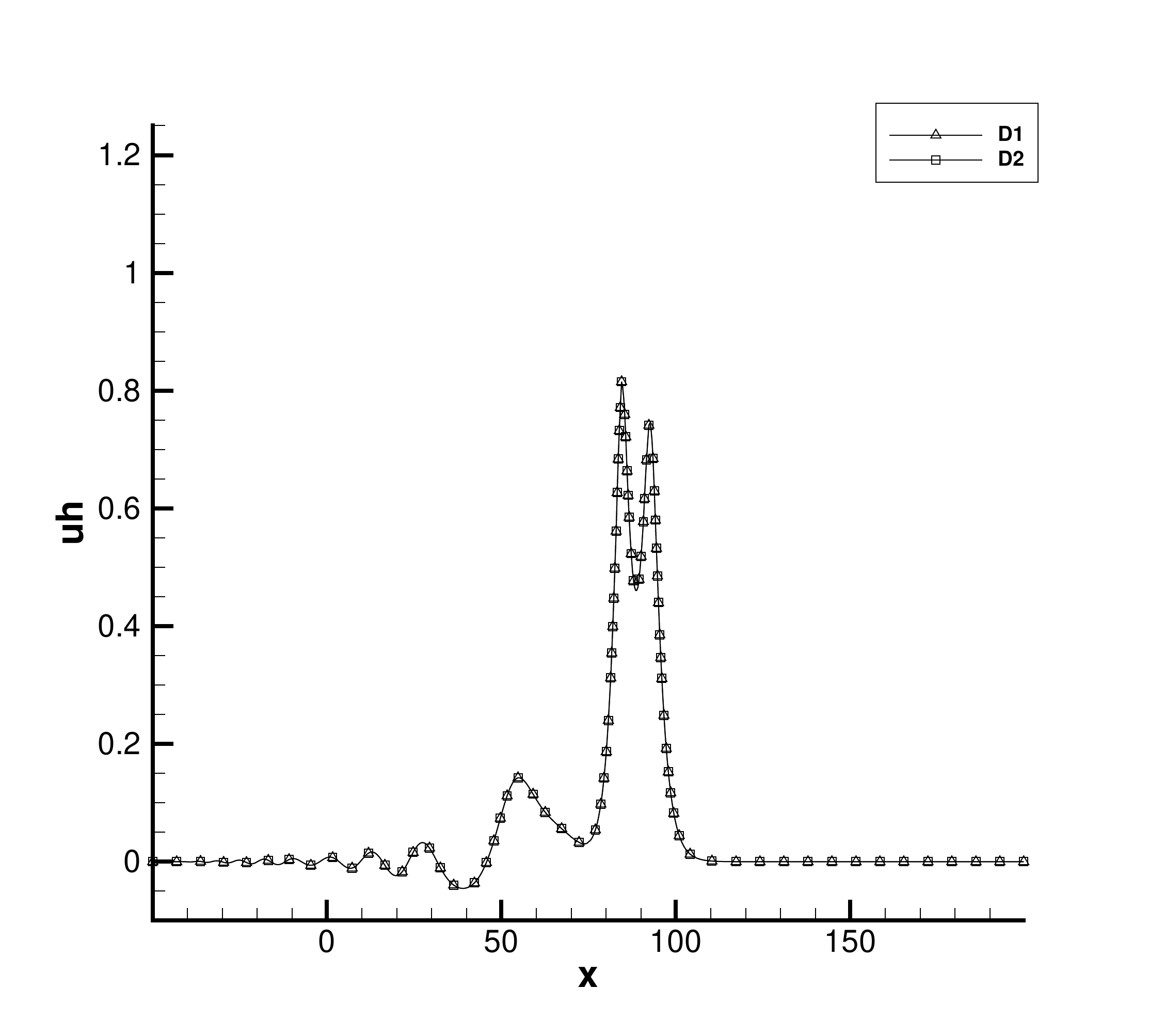}
		}
		\subfigure[$t = 120.0$] {
			\includegraphics[width=0.45\columnwidth]{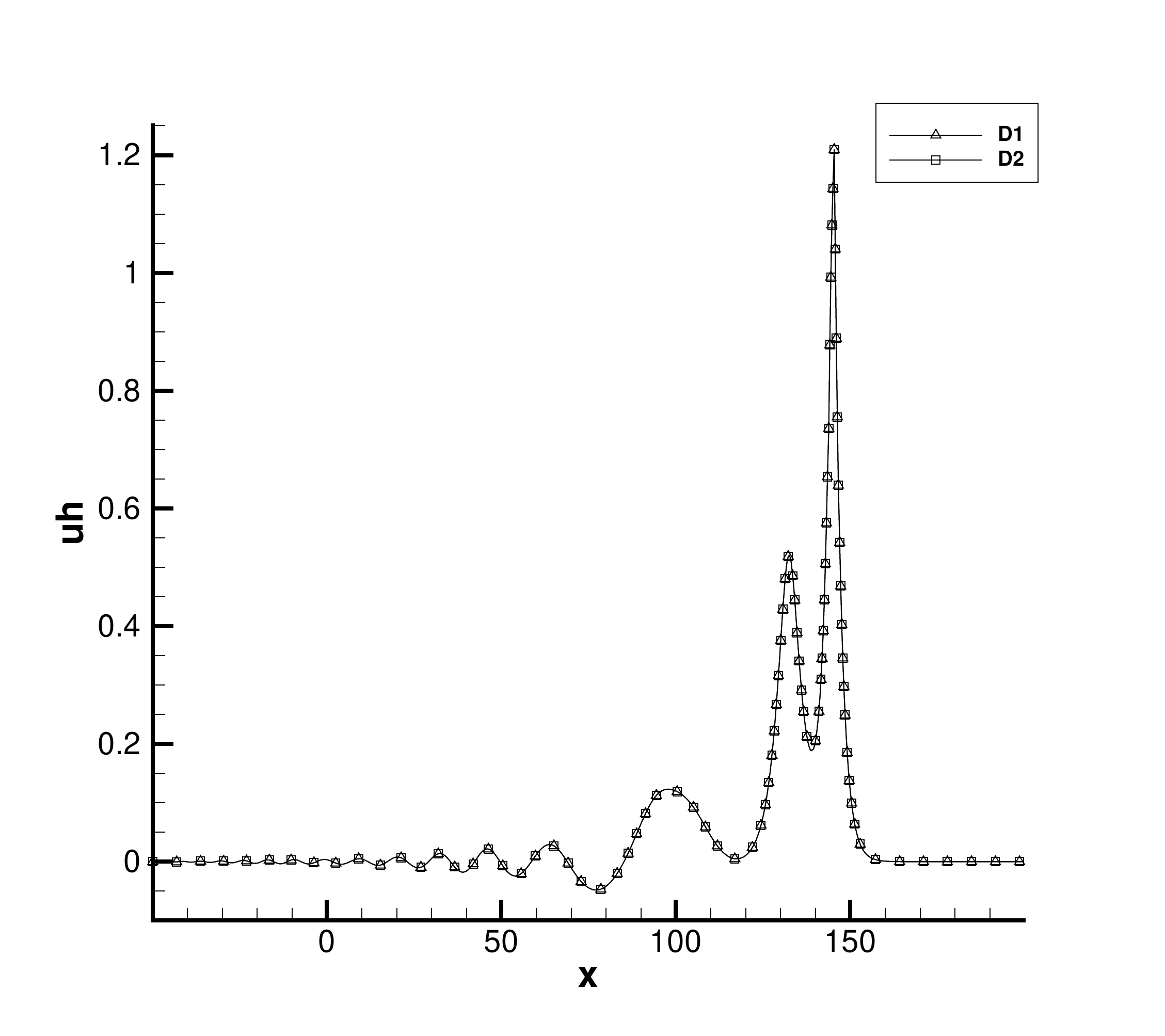}
		}
		
		\caption{\label{fig:2soliton} Example \ref{ex:2soliton}, interaction of two solitons \eqref{2soliton} with cells $N = 160$, $P^2$ elements. }
	\end{figure}

\end{example}

\begin{example}\label{ex:single_peakon} $\bf{ Single \ peakon\ solution}$

Our numerical schemes also work for the peakon solution whose first derivative is finite discontinuous. The exact solution for the case $p = 2$ is
\begin{align}\label{single_peakon}
u(x,t) = \frac{4}{3}\exp({-\frac{1}{2}\abs{x-st}}) + s - \frac{4}{3}
\end{align}
where $s$ is a constant denoting the speed of the wave. Because of the exponential decay of the solution, we can treat it as a periodic problem in the domain $[-25,25]$. We provide the sketches of this peakon solution at terminal time $T = 6$  with the speed $s = 2$. The approximations of the conservative schemes $\mathcal{C}1$ and $\mathcal{C}2$ have some oscillation at cells $ N = 160 $.
Refining the spatial meshes or using  higher order schemes can fix the oscillation, see Figure \ref{fig:single_peakon2}. From the limit value of the amplitude, it is noticed that the corner of the peakon solution is resolved better in the plots
$(c)$ and $(d)$ of Figure \ref{fig:single_peakon2} by higher order schemes.

\begin{figure}[!htp] \centering
		\subfigure[$P^2, N = 320$] {
			\includegraphics[width=0.45\columnwidth]{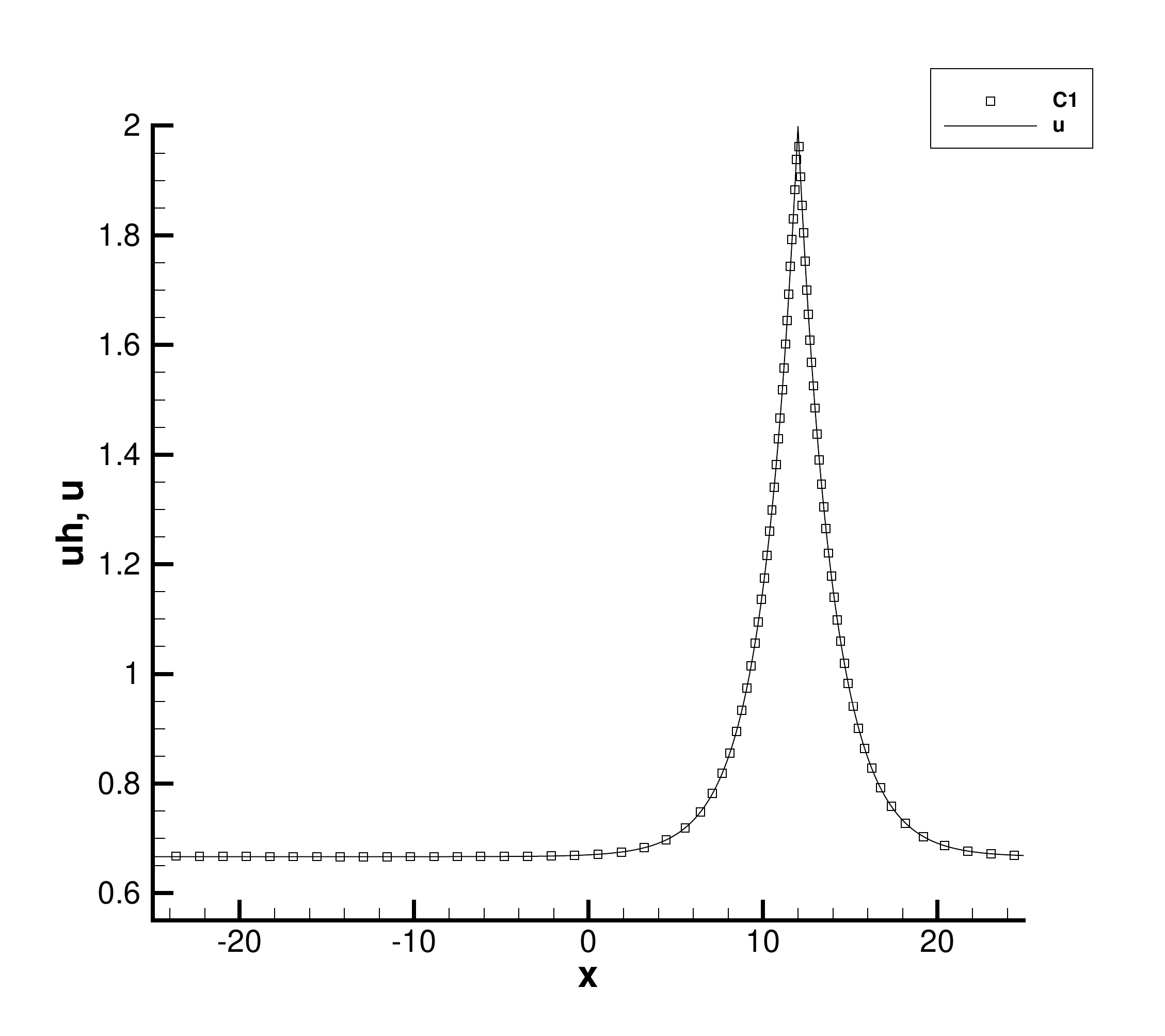}
		}
		\subfigure[$P^2, N = 320$] {
			\includegraphics[width=0.45\columnwidth]{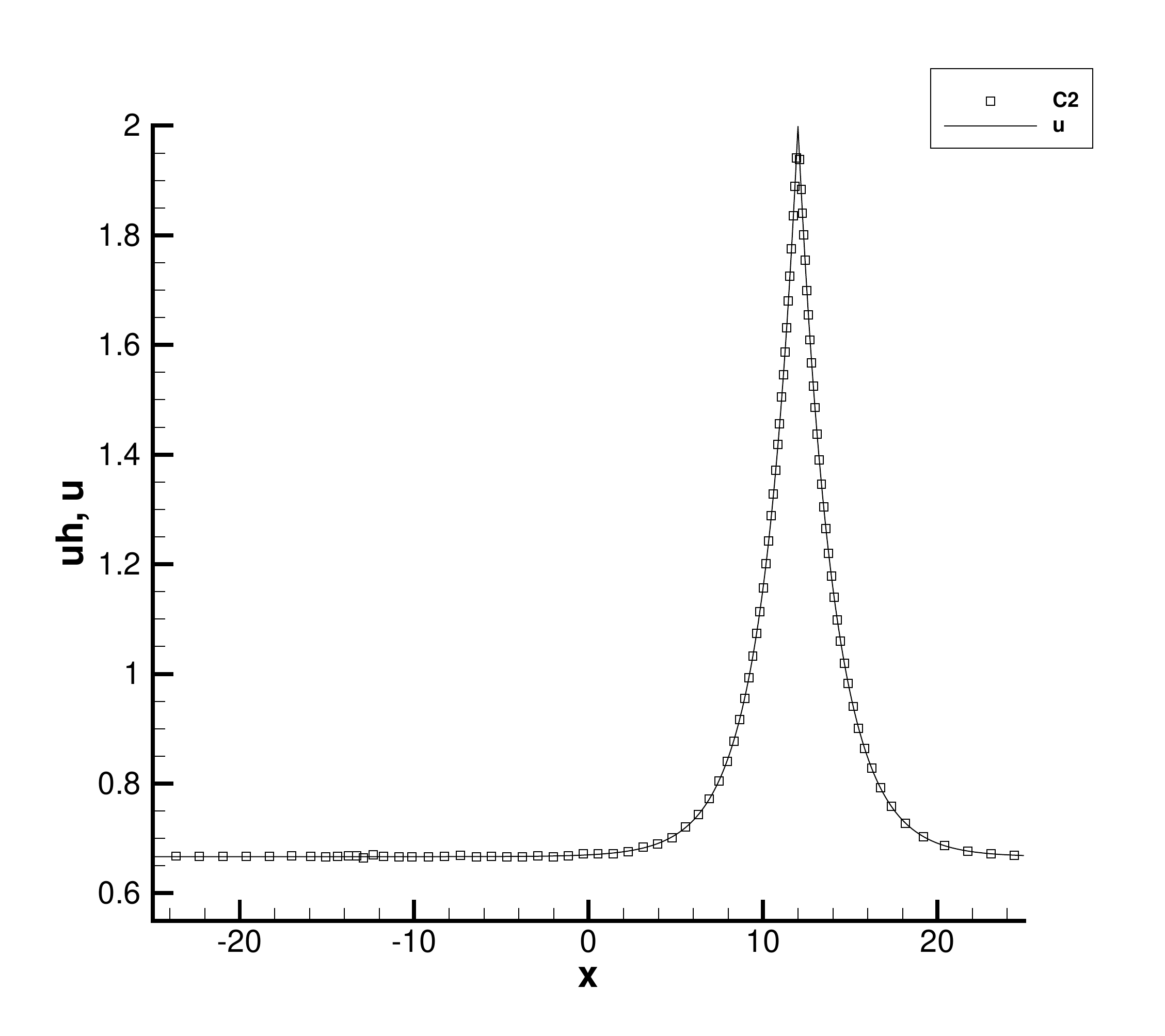}
		}
		\subfigure[$P^4, N = 160$] {
			\includegraphics[width=0.45\columnwidth]{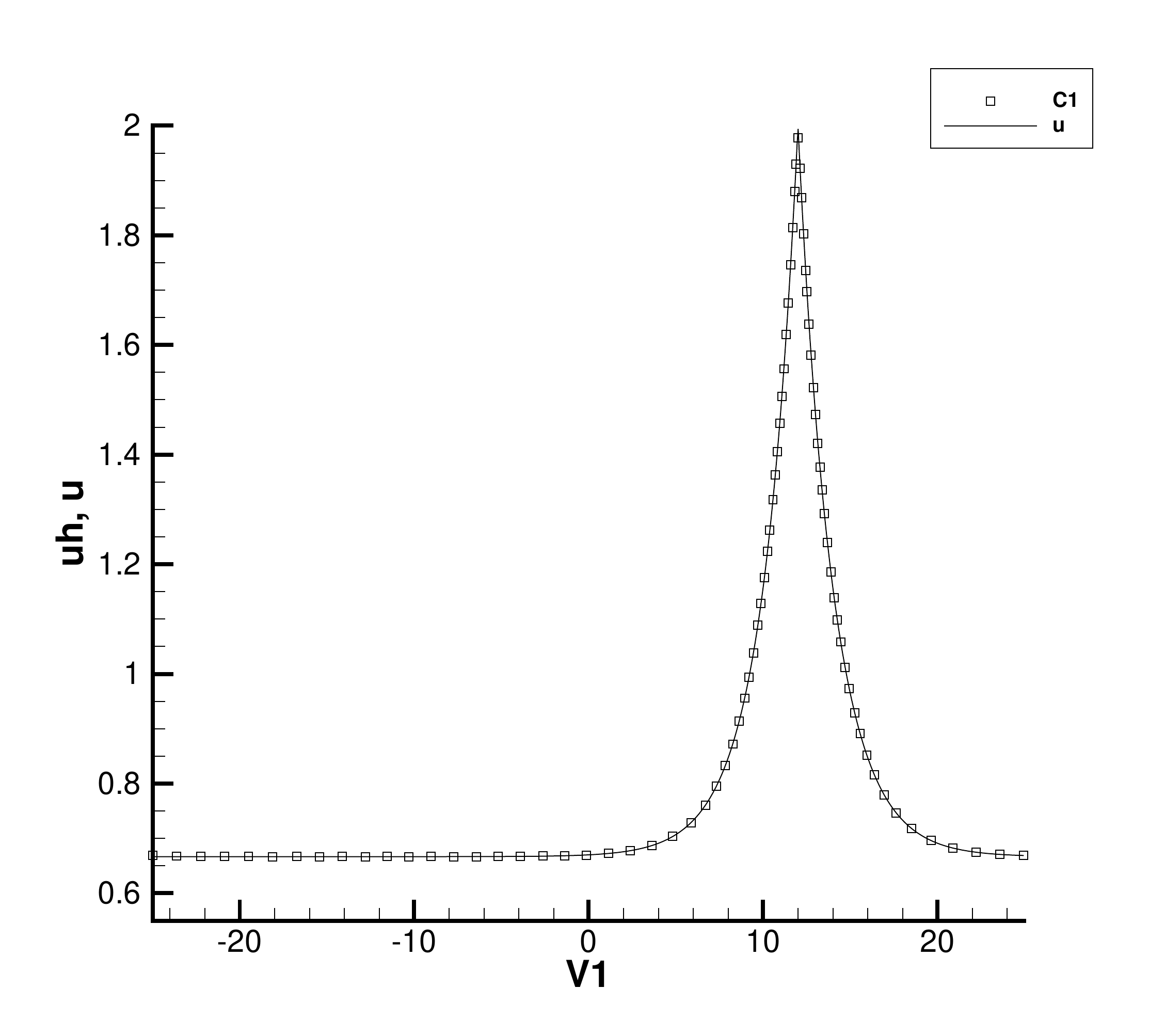}
		}
		\subfigure[$P^4, N = 160$] {
			\includegraphics[width=0.45\columnwidth]{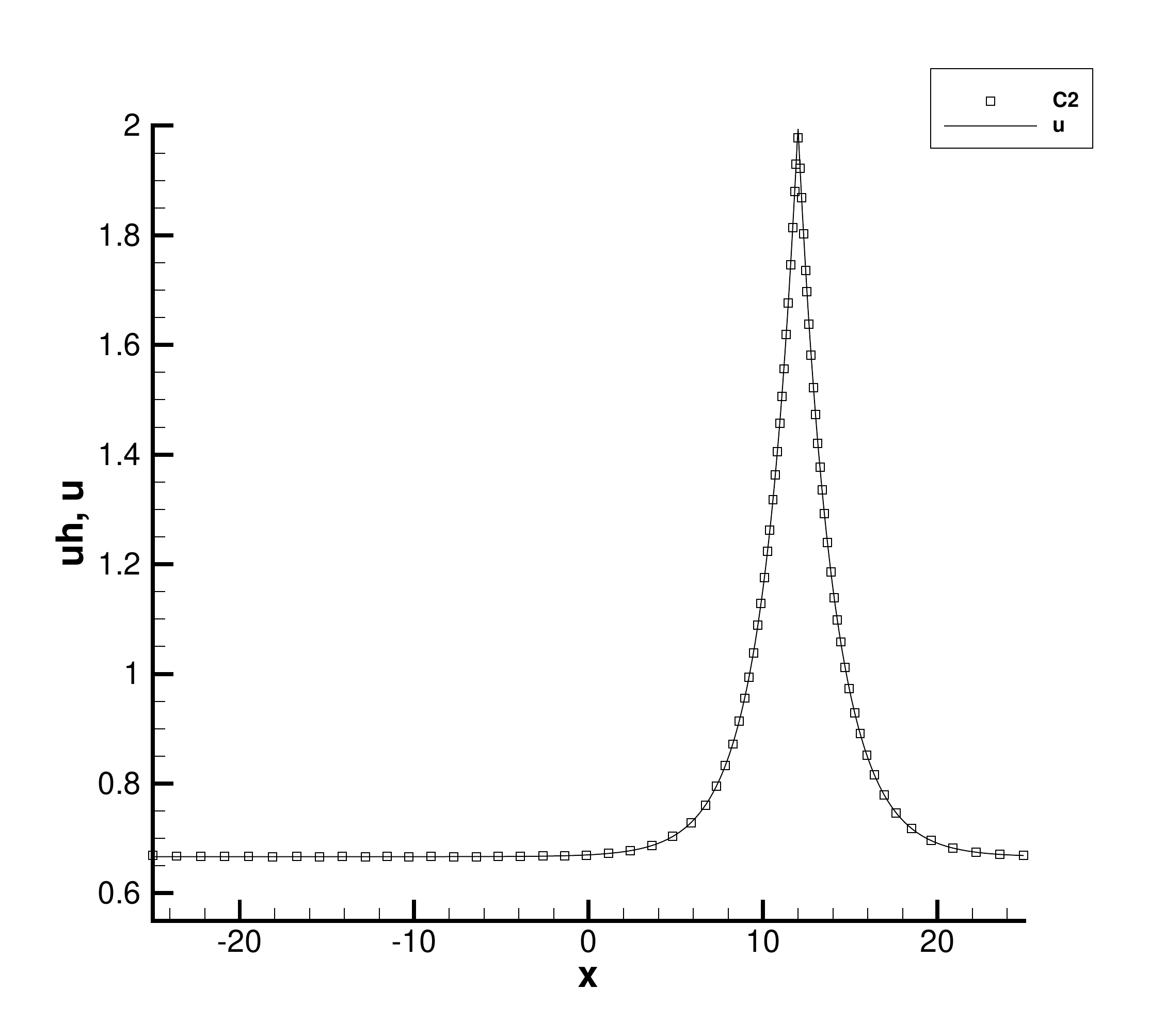}
		}
		
		\caption{\label{fig:single_peakon2} Example \ref{ex:single_peakon}, single peakon solution
\eqref{single_peakon} at $T = 6.0$ in the computational domain $[-25,25]$.    }
	\end{figure}

\end{example}

\begin{example}\label{ex:periodic_peakon} $\bf{ Periodic\ peakon\ solutions}$

For the Fornberg-Whitham equation i.e. $p=2$, we try to get the approximations for periodic peakon solutions \cite{Hormann_2018_arxiv} with period $2T_p$,
\begin{align}
&\label{periodic_peakon} u(x,t) = \varphi(x-st-2nT_p), \ \text{for} \ (2n-1)T_p < x-st < (2n+1)T_p,\\
&\varphi(\zeta) = d_+\exp(-\frac{1}{2}\abs{\zeta}) + d_-\exp(\frac{1}{2}\abs{\zeta}) + s - \frac{4}{3}, \ \zeta = x-st, \notag
\end{align}
where
\begin{align*}
&d_{\pm} = \frac{1}{6}(4\pm3\sqrt{4g + 4s-2s^2}),\\
&T_p = 2\abs{\ln(\phi -s + \frac{4}{3})- \ln(2d_-)} \ \text{with} \  \phi = \frac{1}{3}(-4+3s + \sqrt{2(9s^2-18s+8-18g)}).
\end{align*}

The parameters $s, g$ are constants where $s$ denotes the speed of the wave, and $g$ concerns the shape of the wave.
We give three different cases in Figure \ref{fig:peakon2}. When we take the speed $s = 2$, the peakon solution will tend to the cuspon solution, as the parameter $g\rightarrow\frac{4}{9}$. Our proposed schemes have accurate numerical solutions for the different values of $g$.
\begin{figure}[!htp]
\begin{center}
\begin{tabular}{ccc}
\includegraphics[width= 0.33\textwidth]{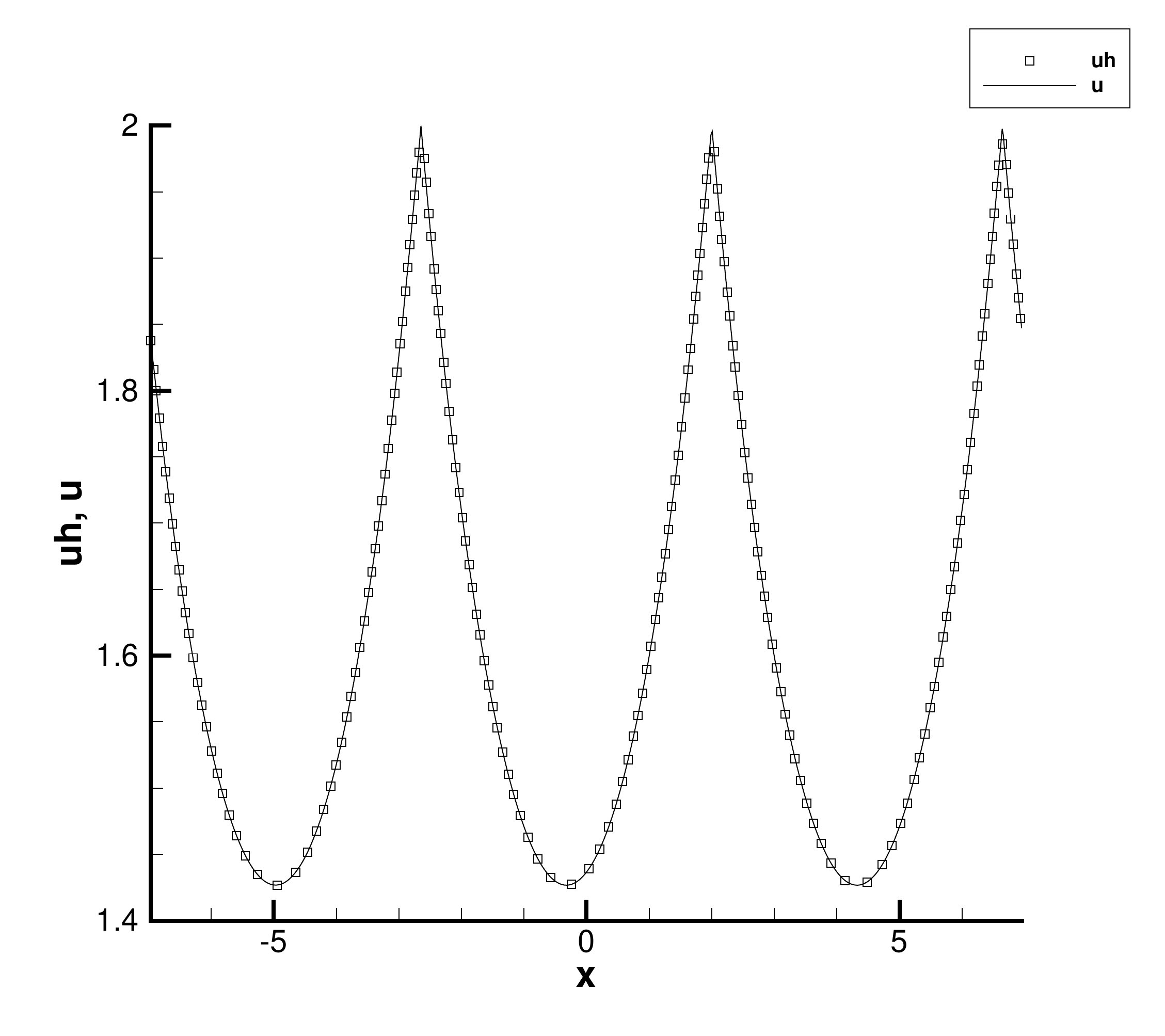} & \includegraphics[width=0.33\textwidth]{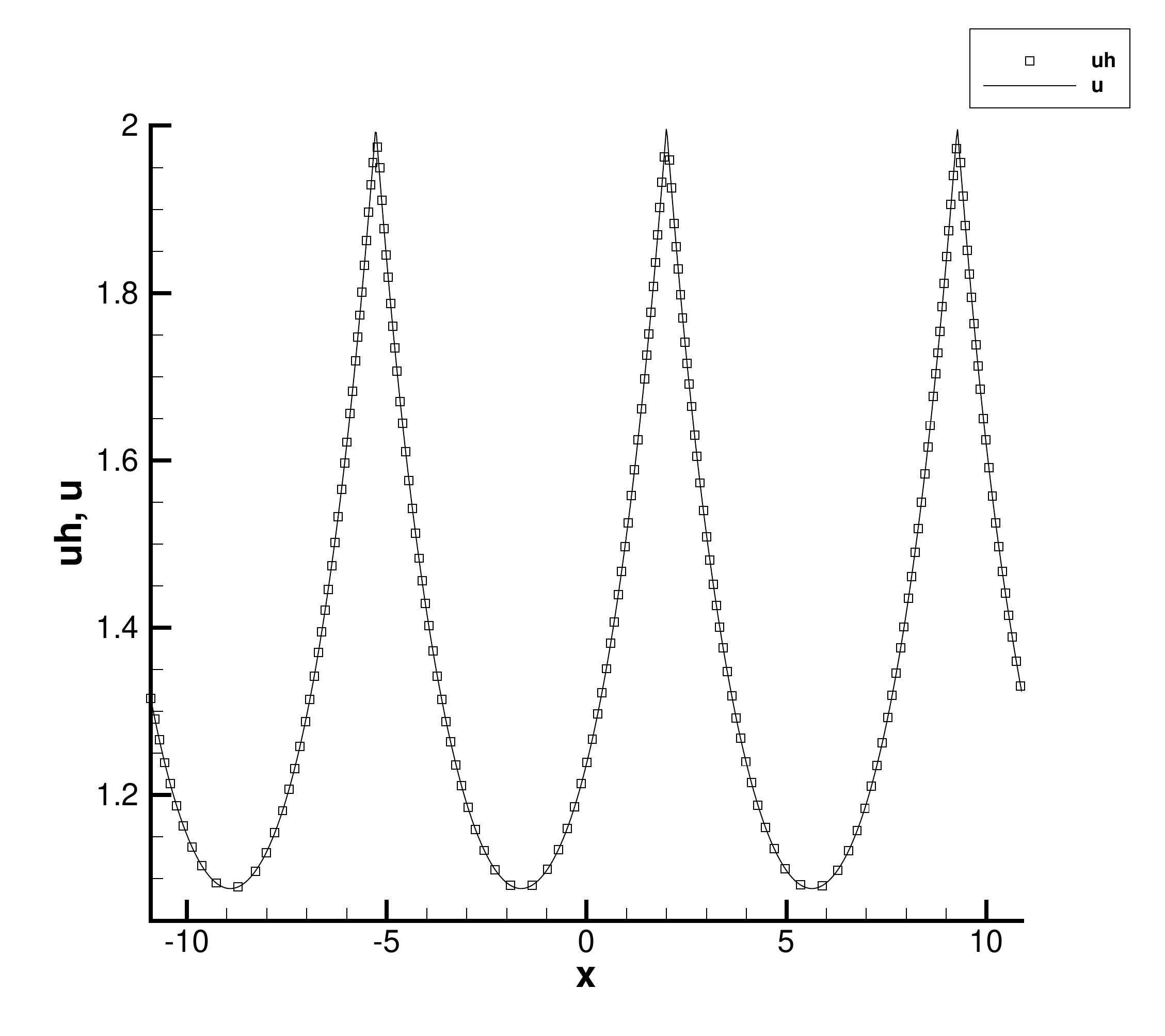}& \includegraphics[width=0.33\textwidth]{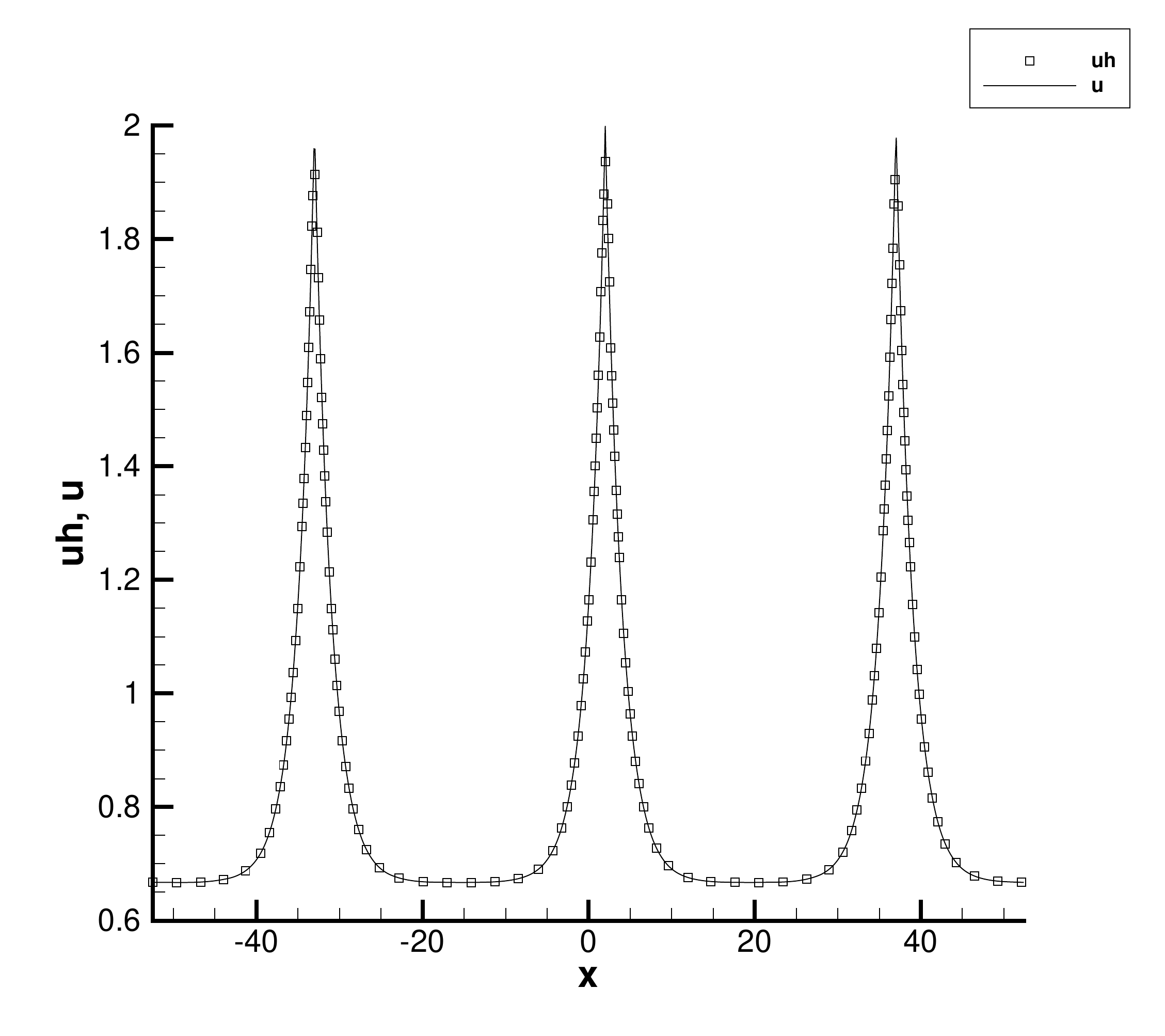}\\
(a) $ g = 0.3 $ & (b)  $ g = 0.4 $  & (c) $g = 0.4444444 $
\end{tabular}
\end{center}
\caption{\label{fig:peakon2} Example \ref{ex:periodic_peakon}, periodic peakon solutions \eqref{periodic_peakon} with different $g$ at $T=1$ in the computational domain $[-3T_p,3T_p]$, $N = 160, P^2$ elements.}
\end{figure}

For a long time approximation, we use $s = 2, g = 0.3$ in solution \eqref{periodic_peakon} as an example to illustrate the differences among the four numerical schemes. In Figure \ref{fig:peakonp2}, we take the degree  of piecewise polynomial space $k = 2$.
It can be seen that the dissipative schemes $\mathcal{D}1$ and $\mathcal{D}2$ become inaccurate due to the error of shape and decay of amplitude over a long temporal interval. The conservative schemes have more accurate approximation results than the dissipative ones. In Figure \ref{fig:peakonp4}, we show the results when the degree $k=4$. It tells that high order discretization methods can reduce shape error of waves effectively.  The conservative property we prove is a semi-discrete property for our schemes, which implies the fully discretization has energy fluctuation, see Table \ref{tab:Energy_fluctuation}. However, the conservativeness can reduce the dissipation of energy so that the conservative schemes have better approximation over a long temporal interval. {The shape error caused by dispersion error and phase speed error can be reduced by conservativeness or higher order accuracy,
the detailed analysis can be accomplished by Fourier expansion and error dynamics \cite{zhang2003_MMMAS,Sengupta2007_JCP}.}

In Table \ref{tab:CPU_time}, we make a comparison of the CPU time for the four proposed LDG schemes.
It is indicated that the dissipative scheme costs less time than the conservative scheme owing to the minimal stencils we choose, similar to \cite{Zhang2019CiCP}. On the other hand, Schemes $\mathcal{D}1$, $\mathcal{C}1$ in Section \ref{DG1} are more
effective than Schemes $\mathcal{D}2$, $\mathcal{C}2$ in Section \ref{DG2}, correspondingly.

%
\end{example}

\begin{figure}[!htp] \centering
		\subfigure[$t=0.0$] {
			\includegraphics[width=0.45\columnwidth]{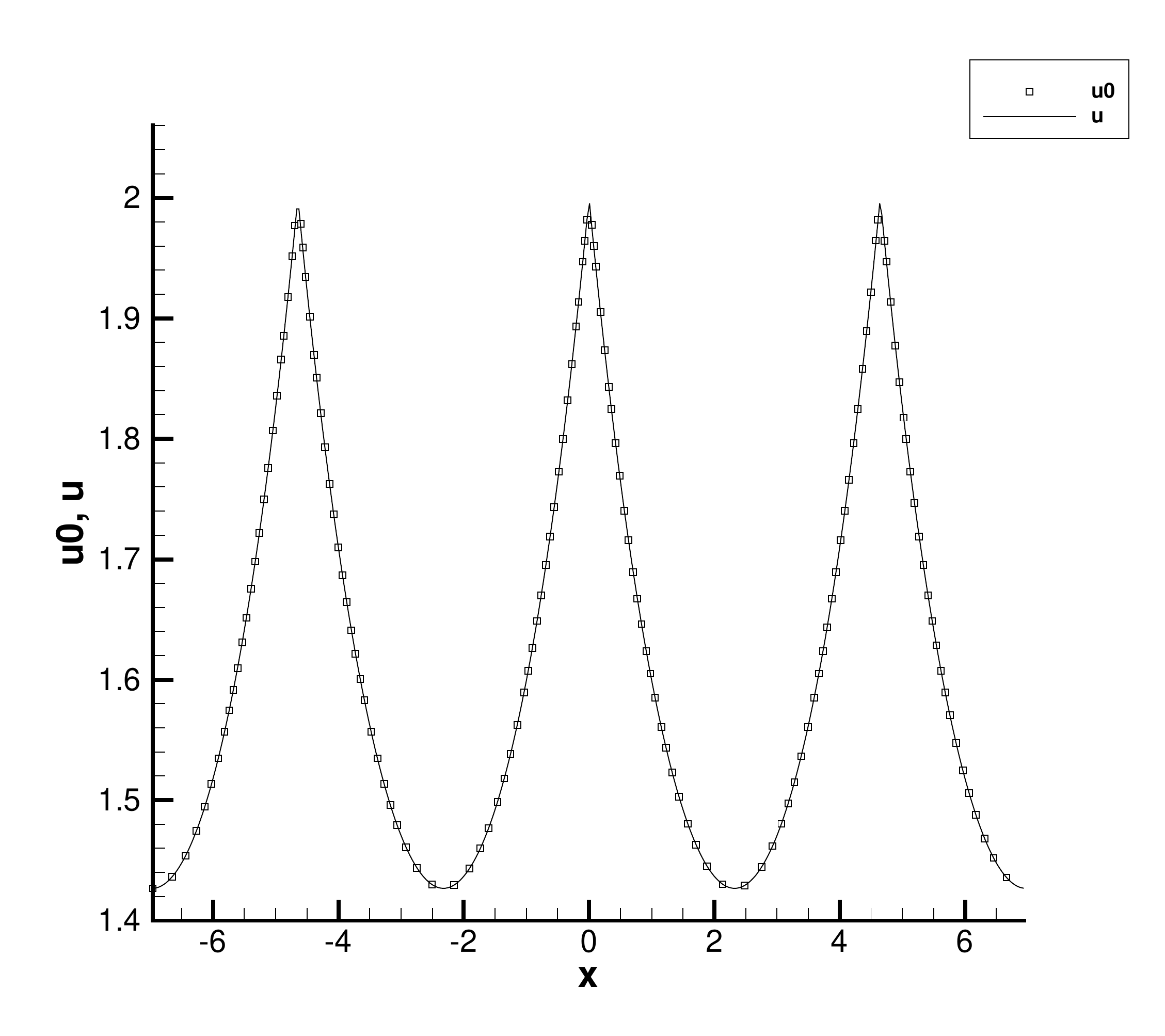}
		}
		\subfigure[$t=100.0$] {
			\includegraphics[width=0.45\columnwidth]{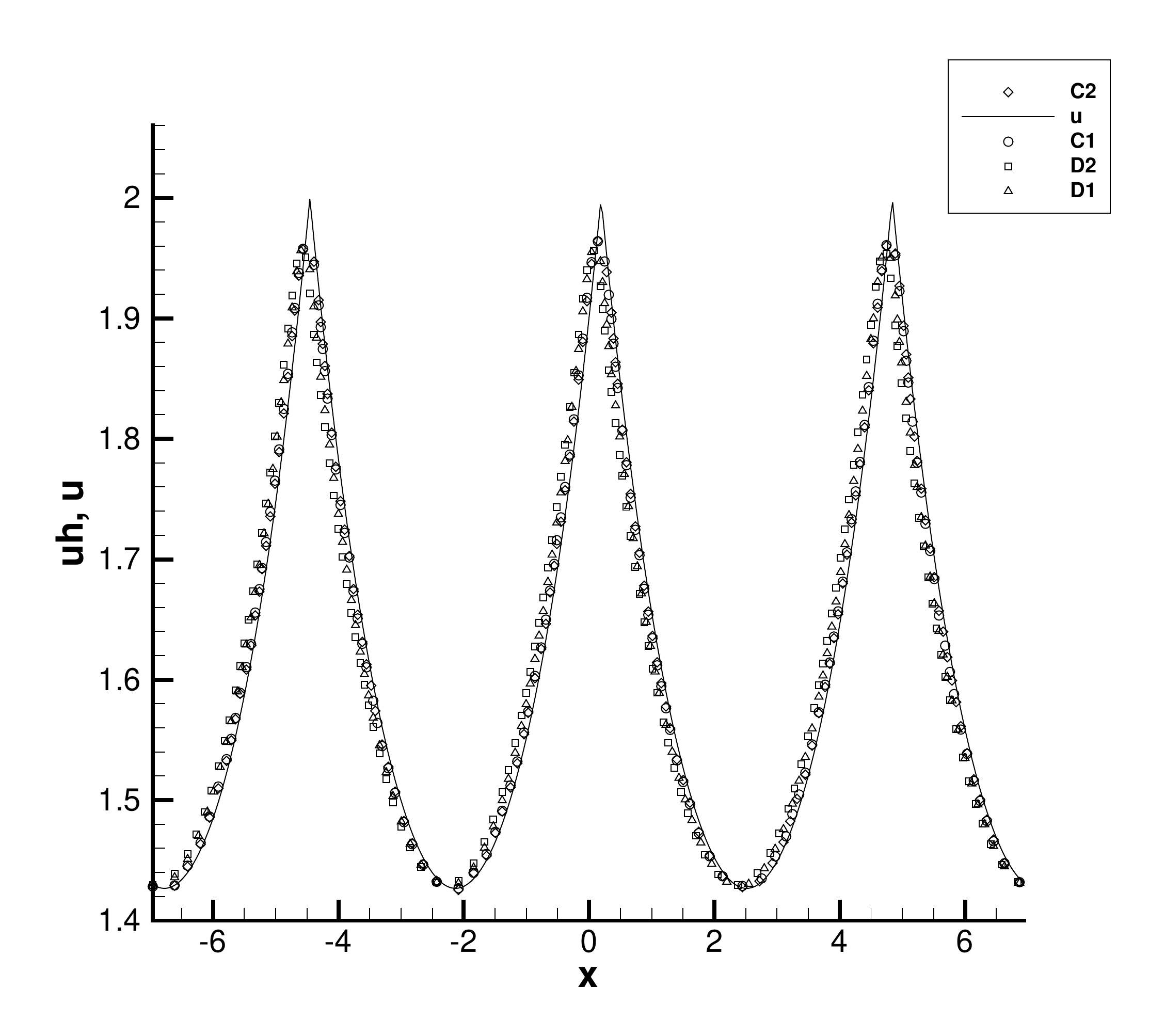}
		}
		\subfigure[$t=200.0$] {
			\includegraphics[width=0.45\columnwidth]{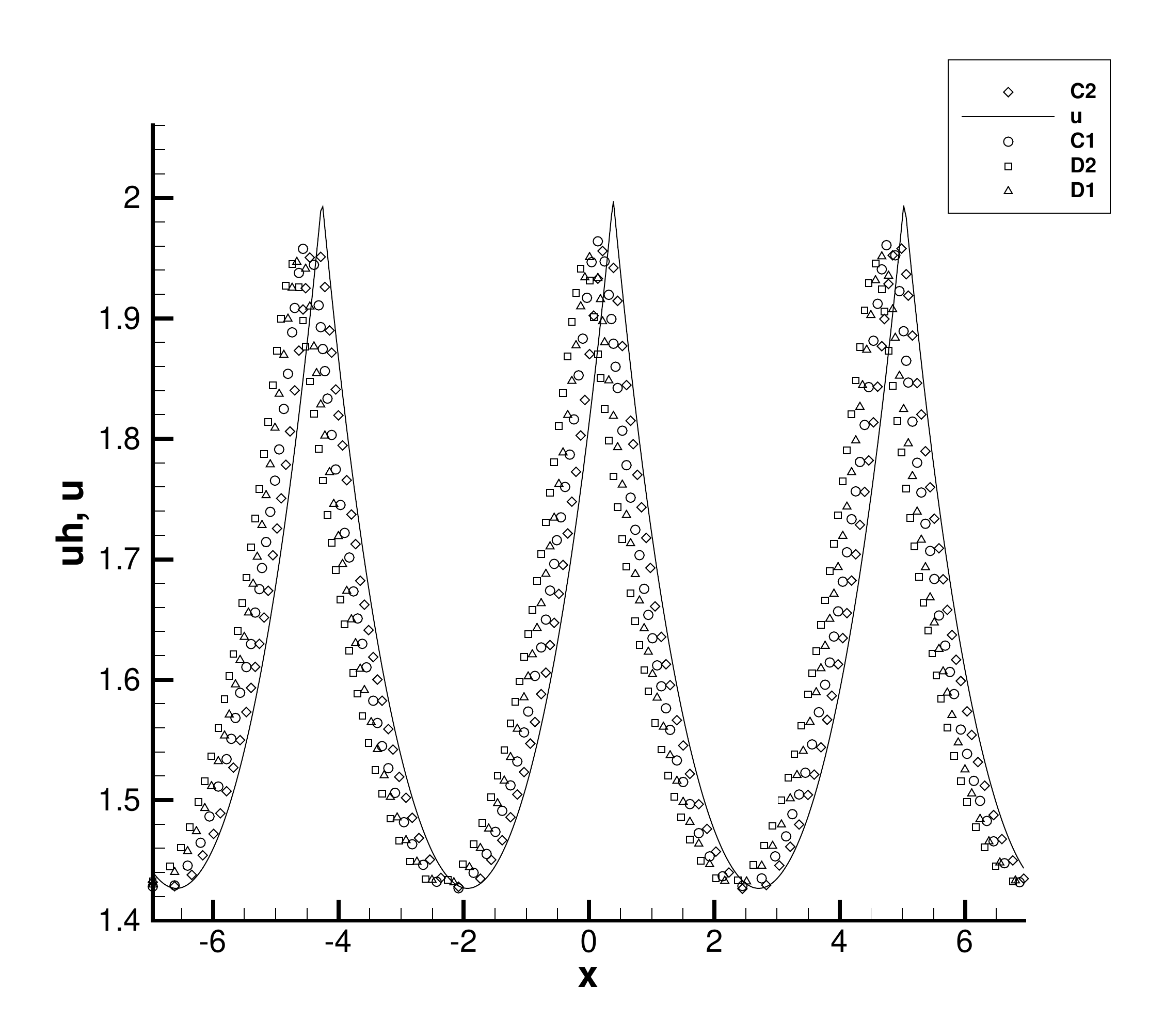}
		}
		\subfigure[$t=300.0$] {
			\includegraphics[width=0.45\columnwidth]{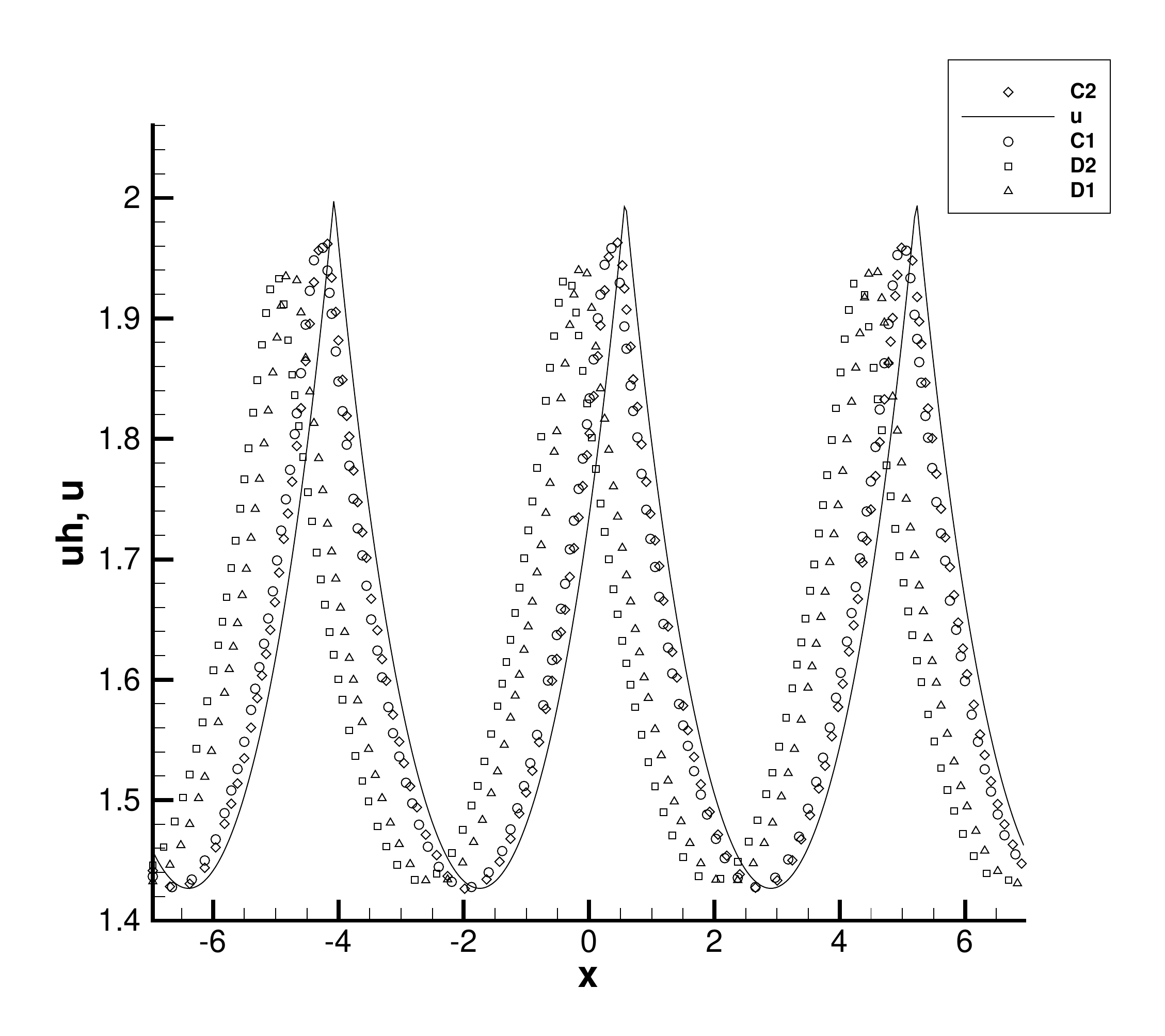}
		}
		
		\caption{\label{fig:peakonp2} Example \ref{ex:periodic_peakon}, periodic peakon solution \eqref{periodic_peakon}, $N = 80, P^2$ elements.    }
	\end{figure}


\begin{figure}[!htp] \centering
		\subfigure[$t=0.0$] {
			\includegraphics[width=0.45\columnwidth]{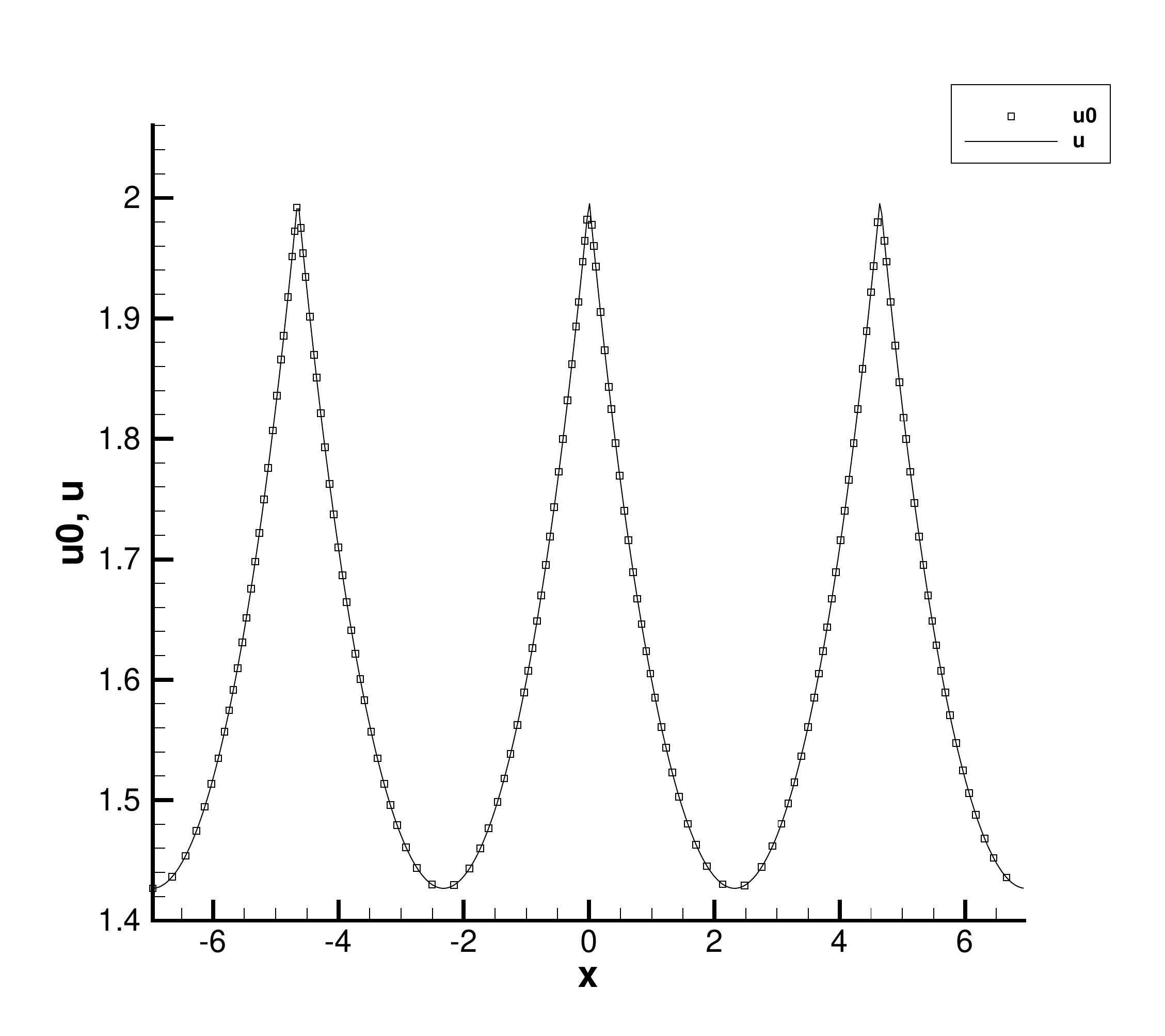}
		}
		\subfigure[$t=100.0$] {
			\includegraphics[width=0.45\columnwidth]{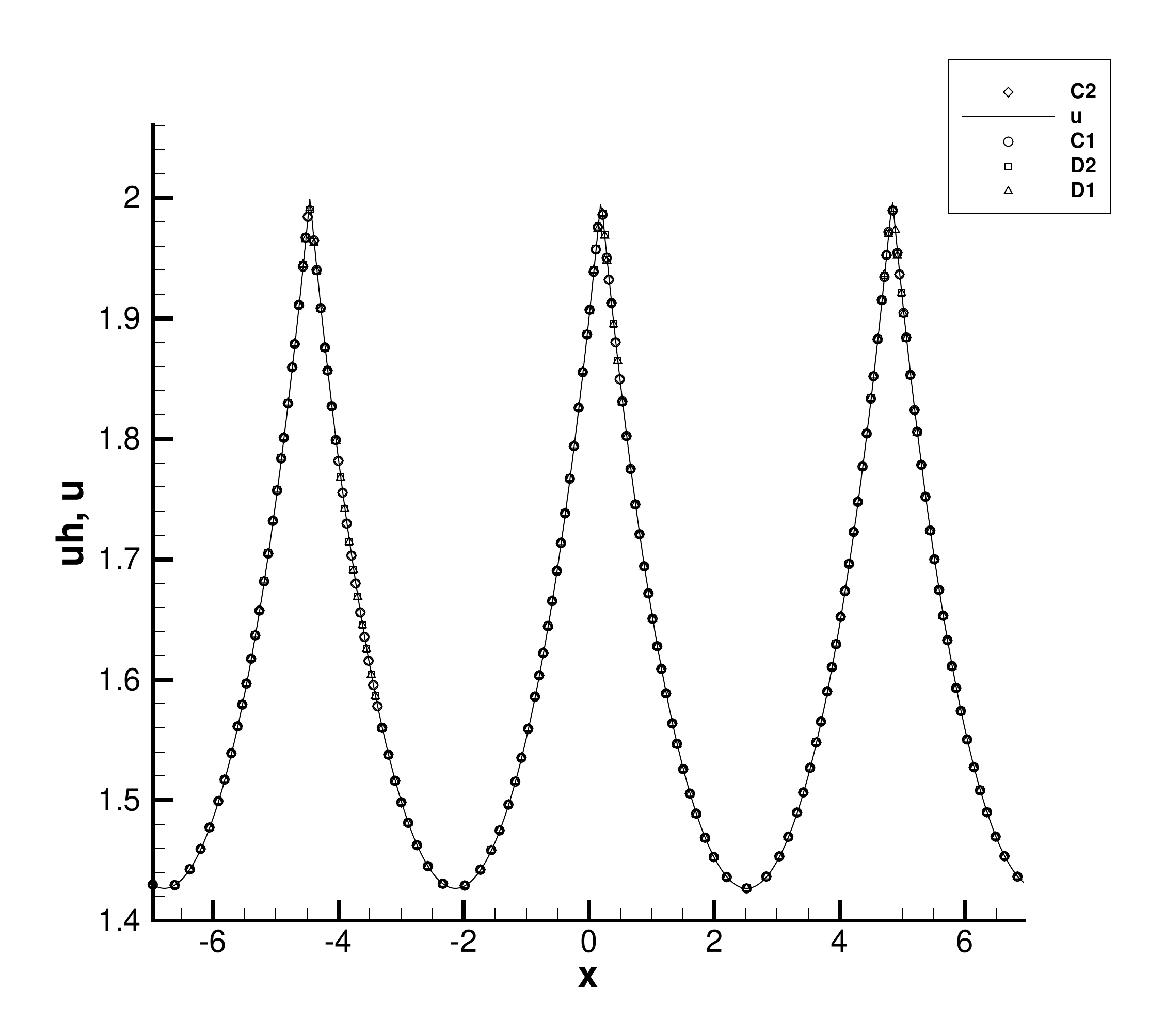}
		}
		\subfigure[$t=200.0$] {
			\includegraphics[width=0.45\columnwidth]{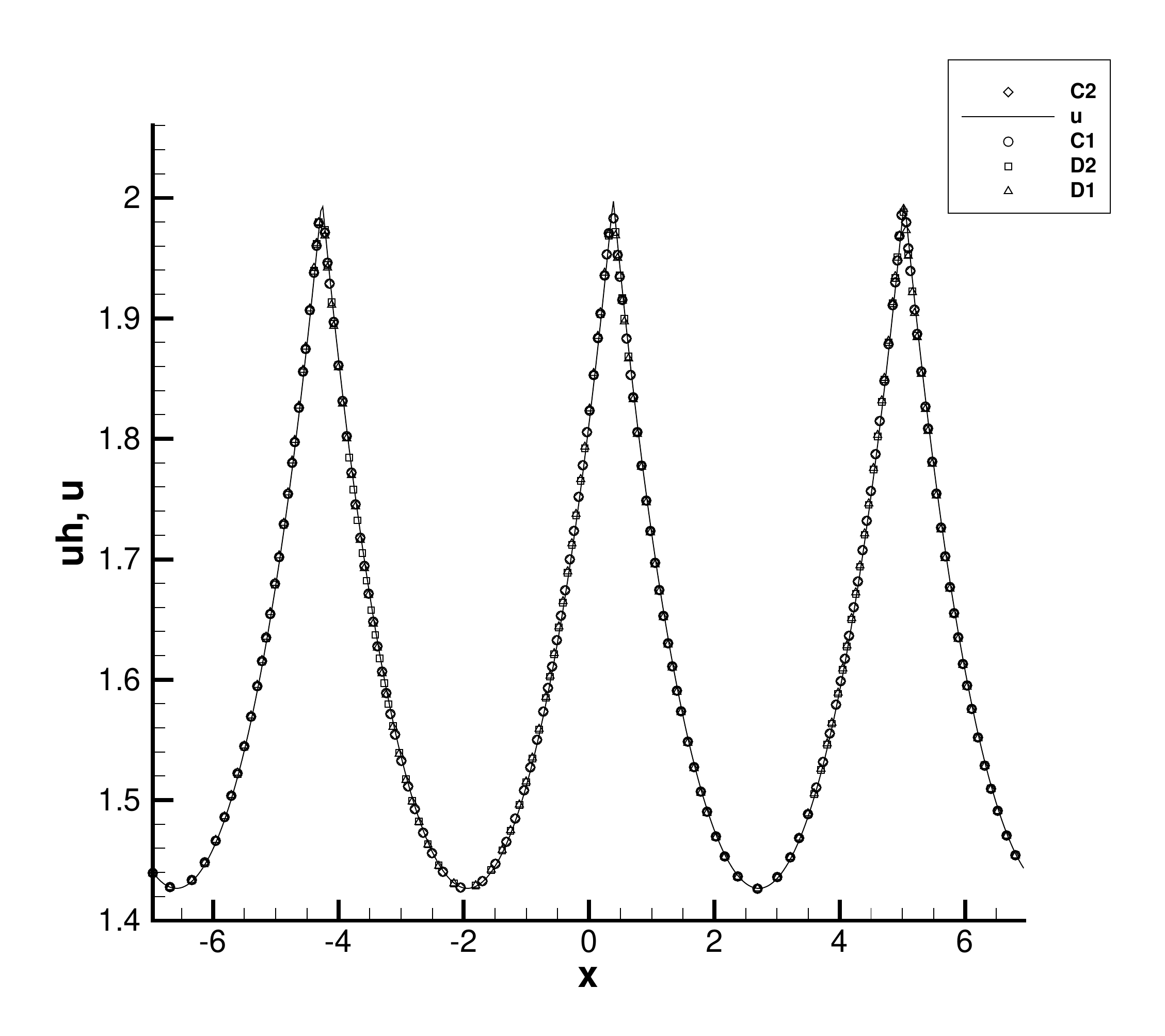}
		}
		\subfigure[$t=300.0$] {
			\includegraphics[width=0.45\columnwidth]{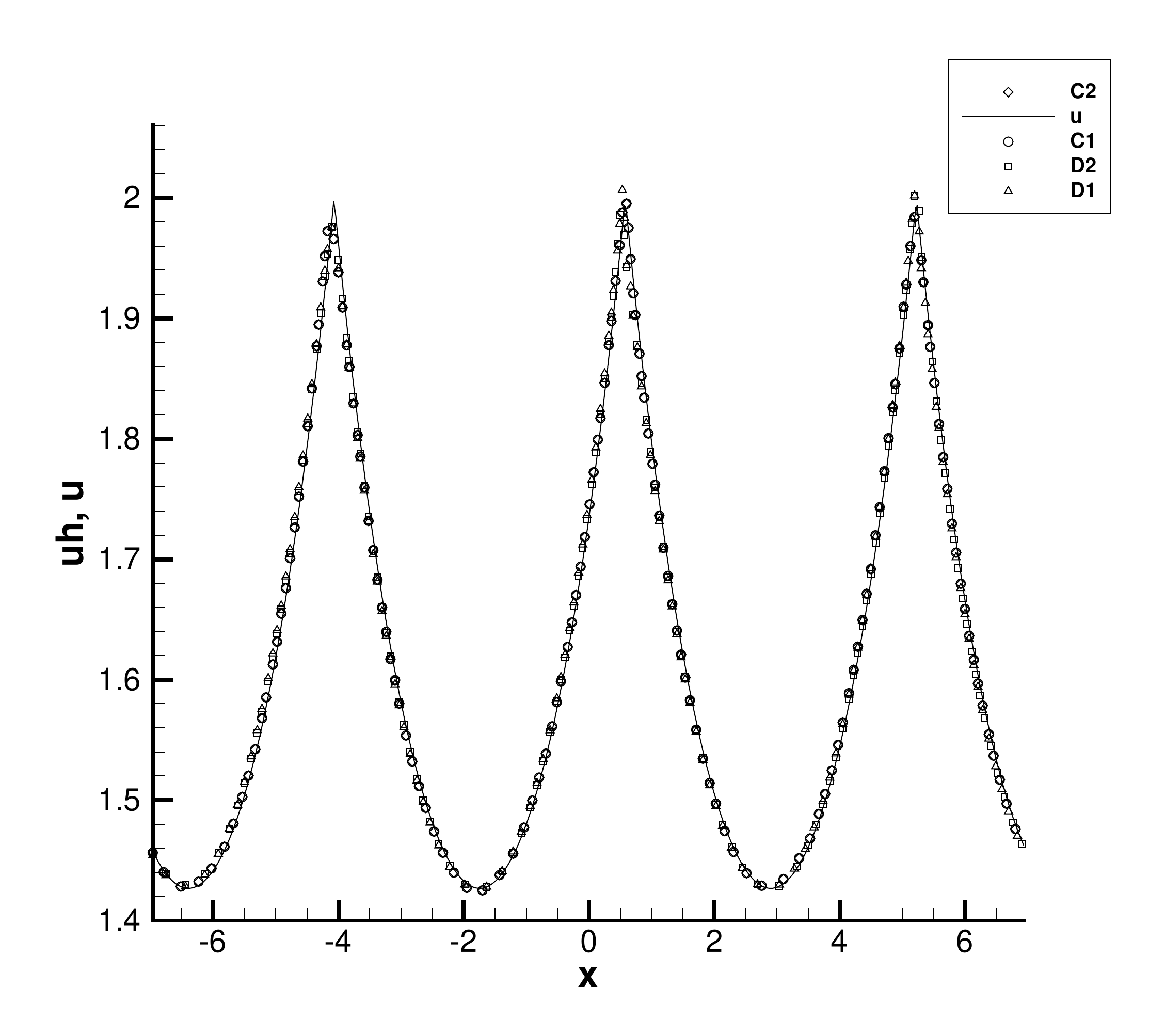}
		}
		
		\caption{\label{fig:peakonp4} Example \ref{ex:periodic_peakon}, periodic peakon solution \eqref{periodic_peakon}, $N = 80, P^4$ elements.   }
	\end{figure}


\begin{table}[!htp]
\begin{tabular}{ccccc}
\hline
     & Scheme $\mathcal{C}1$       & Scheme $\mathcal{C}2$  & Scheme $\mathcal{D}1$ & Scheme $\mathcal{D}2$     \\ \hline
$P^2$   &   41.7           &   72.3               &30.5   & 67.8 \\\hline

$P^4$   &   494.2              & 925.7              &378.8 & 898.5 \\\hline

\end{tabular}
\caption{\label{tab:CPU_time} Example \ref{ex:periodic_peakon}, CPU time of the
proposed LDG schemes for periodic peakon solution \eqref{periodic_peakon} at $T=300, N = 80$ cells. }

\end{table}

\begin{table}[!htp]
\begin{tabular}{ccccc}
\hline
     & Scheme $\mathcal{C}1$       & Scheme $\mathcal{C}2$  & Scheme $\mathcal{D}1$ & Scheme $\mathcal{D}2$     \\ \hline
$P^2$   &   1.14E-02             &  1.14E-02             &3.34E-02   & 3.36E-02 \\\hline

$P^4$   &  1.44E-03              &    1.44E-03           &3.87E-03 & 4.27E-03 \\\hline

\end{tabular}
\caption{\label{tab:Energy_fluctuation} Example \ref{ex:periodic_peakon}, energy fluctuation of the
proposed LDG schemes for periodic peakon solution \eqref{periodic_peakon} at $T = 300, N = 80$ cells. }

\end{table}

{
\begin{remark}
In summary, the numerical experiments demonstrate that the LDG scheme for \eqref{eqn:FW_intro2}  in Section \ref{DG2} is similar to the LDG scheme for  \eqref{eqn:FW_intro} in Section \ref{DG1} on the accuracy, convergence rate numerically or theoretically, and the capability for shock solutions. And the LDG scheme for \eqref{eqn:FW_intro2} is inferior in the behavior for a long time approximation and efficiency.  However, the LDG scheme for \eqref{eqn:FW_intro2} can handle the cases with different $f(u)$ on the
two sides of the equation, such as the modified Fornberg-Whitham equation $$ u_t - u_{xxt} + (\frac{1}{3}u^3)_x + u_x = (\frac{1}{2}u^2)_{xxx}.$$ The LDG scheme  for \eqref{eqn:FW_intro2} can be applied more widely. More details will be given
in our future work which is out of the scope of this paper.
\end{remark}}

\section{Conclusion}\label{conclusion}
In this paper, for the two different forms of the Fornberg-Whitham type equations, we construct dissipative LDG schemes $\mathcal{D}1$ and $\mathcal{D}2$, and conservative DG schemes $\mathcal{C}1$ and $\mathcal{C}2$.
For the dissipative schemes, the optimal order of accuracy can be achieved numerically, which are identical to the theoretical proof for odd $p$.  While for even $p$, suboptimal convergence rate can be proved. The conservative schemes have only $k$-$th$ accuracy order for odd degree polynomial space, and $(k+1)$-$th$ order for even $k$ numerically. In theory, both $k$-$th$ order can be verified. Different solutions, including shock solutions and peakon solutions, can be resolved well with our proposed LDG schemes. For a long time approximation, conservative schemes can reduce the dissipation significantly, and high-order accurate schemes can
achieve the same improvement.

\begin{appendix}
\section{Appendix: Proof of several lemmas}

\subsection{Proof of Lemma \ref{lemma:another_energy}}
\label{proofanother_energy}
In scheme \eqref{scheme:DG}, we take test functions as
\begin{align*}
\varphi = v_h, \ \psi = q_h, \ \psi = u_h.
\end{align*}
After summing up \eqref{scheme:DG1} and \eqref{scheme:DG3} (twice with different test functions) over all intervals, we obtain
\begin{align*}
(v_h , v_h)_{I} + (q_h , q_h)_{I} - \mathcal{L}^-({q_h}, v_h) -\mathcal{L}^+({v_h}, q_h)- \mathcal{L}^-({u_h}, v_h) -\mathcal{L}^+({v_h}, u_h)
                                      = -(q_h , u_h)_{I},
\end{align*}
for the dissipative scheme $\mathcal{D}1$,

then Lemma \ref{lemma:properties_DN} can be used to derive the equality \eqref{eqn:another_energy}.
Notably, we can replace all operators $\mathcal{L}^+, \mathcal{L}^-$ with $\mathcal{L}^c$ in the above equation to get
the same equality for the conservative scheme $\mathcal{C}1$.

\subsection{Proof of Lemma \ref{lemma:b}}
\label{proofb}
There are two terms for $\mathcal{B}_j$,
\begin{align*}
\mathcal{B}_j(\xi^u- \eta^u, &\xi^v- \eta^v, \xi^q- \eta^q ; \xi^u, \bm{\xi^1}, \bm{\xi^2})\\
 &=  \mathcal{B}_j(\xi^u, \xi^v, \xi^q; \xi^u, \bm{\xi^1}, \bm{\xi^2}) - \mathcal{B}_j(\eta^u, \eta^v, \eta^q; \xi^u, \bm{\xi^1},\bm{\xi^2}). \tag{A.2.1}
\end{align*}
\begin{itemize}
\item For \eqref{eqn:B_estimate} of the dissipative scheme $\mathcal{D}1$:
Via the proof of $L^2$ dissipation and Lemma \ref{lemma:another_energy}, we get
\begin{align*}
\mathcal{B}_j(\xi^u, \xi^v, \xi^q; \xi^u, \bm{\xi^1}, \bm{\xi^2}) &= (\xi_t^u,\xi^u)_{I_j}  + \norm{\xi^v}^2_{L^2(I_j)} + \norm{\xi^q}^2_{L^2(I_j)} \\
&+ (\xi^u,\xi^q)_{I_j} + \frac{1}{2}(\jump{\xi^u}+ \jump{\xi^q})^2_{j+\frac{1}{2}} + \frac{1}{2}\jump{\xi^v}^2_{j+\frac{1}{2}}.
\end{align*}
As for the second term of (A.2.1), it can be derived by
the properties of Gauss-Radau projection which causes the boundary terms to vanish, as well as the integral terms with spatial derivatives.
\item For \eqref{eqn:B_estimate2} of the conservative scheme $\mathcal{C}1$:
 The first term of $\mathcal{B}_j$ is obtained by the $L^2$ conservation and Lemma \ref{lemma:another_energy}, i.e.
\begin{align*}
&\mathcal{B}_j({\xi}^u, {\xi}^v, {\xi}^q; {\xi}^u, \bm{\xi^1}, \bm{\xi^2})= ({\xi}_t^u,{\xi}^u)_{I_j}  + \norm{{\xi}^v}^2_{L^2(I_j)} + \norm{{\xi}^q}^2_{L^2(I_j)}+ (\xi^u,\xi^q)_{I_j}
\end{align*}
The second term of $\mathcal{B}_j$ is derived by the property of $L^2$ projection, orthogonality to all polynomials of degree up to $k$. Thus the remaining is some boundary terms as we state in \eqref{eqn:B_estimate2}.
\end{itemize}
Therefore, we have the results \eqref{eqn:B_estimate} and \eqref{eqn:B_estimate2} for the bilinear term $\mathcal{B}_j$.

\subsection{Proof of Lemma \ref{lemma:H_estimate2}}
\label{proofH_estimate2}
For the part $\mathcal{T}_1$, we use Taylor expansion on $ f(u) - f(u_h), f(u) - f(u^{ref}_h)$ respectively,
\begin{align*}
f(u) - f(u_h) = -f'(u)(u_h-u) -\frac{1}{2}f''_u(u_h-u)^2,\\
f(u) - f(u_h^{ref}) = -f'(u)(u^{ref}_h-u) -\frac{1}{2} \tilde{f}''_u(u^{ref}_h-u)^2,
\end{align*}
where $f''_u,  \tilde{f}''_u$ are the mean values. Substituting the above equations into the equation \eqref{H_divided}, we obtain
\begin{align*}
\mathcal{T}_1
=& -\sum\limits_{j=1}^{N} f'(u)\eta^{ref}\jump{\xi^u} - (f'(u)\eta^u,\xi^u_x)
   +\sum\limits_{j=1}^{N} f'(u)\xi^{ref}\jump{\xi^u} + (f'(u)\xi^u,\xi^u_x)\notag \\
  & -\sum\limits_{j=1}^{N}( \frac{1}{2} \tilde{f}''_u(u_h^{ref}-u)^2\jump{\xi^u} - (\frac{1}{2} f''_u(u_h-u)^2,\xi^u_x)) \triangleq O_1 + O_2 + O_3. \tag{A.3.1}
\end{align*}
For writing convenience, we omit the subscripts $j+\frac{1}{2}$ for the boundary terms, $I_j$ for the integral terms. And we need to explain the notations we have used in the above equation. We have already defined $\eta^u = \mathcal{P}^-u - u, \xi^u = \mathcal{P}^-u-u_h$. Due to the different directions of $u_h^{ref}$, the value on the boundary point $x_{j+\frac{1}{2}}$ for $u-u_h^{ref} \triangleq   \xi^{ref}- \eta^{ref}$ can be represented by
\begin{align*}
&u - u_h^{-} = (\mathcal{P}^-u-u_h)^- - (\mathcal{P}^-u-u)^- \triangleq \xi^{ref} - \eta^{ref}, \\
&u - u_h^{+} = (\mathcal{P}^+u-u_h)^+ - (\mathcal{P}^+u-u)^+ \triangleq \xi^{ref} - \eta^{ref}, \tag{A.3.2}\\
&u - \average{u_h} = \average{\mathcal{P}u-u_h} - \average{\mathcal{P}u-u} \triangleq  \xi^{ref} - \eta^{ref}.
\end{align*}
Notably, the value of $\xi,\eta$ varies with the direction ``$ref$". This setting is for the following estimates.

\begin{itemize}
\item $O_1$ term:
For the cases $\eta^{ref} =\eta^+$ or $\eta^-$, according to the definition in (A.3.2), we have $\eta^{ref} = 0$. Now we check the case $\eta^{ref} =\average{\eta}$, i.e. the $f'(u)$ change its sign on $I_j \cup I_{j+1}$. There is an extra $h$ by $\abs{f'(u)}\leq C_*h $, thus the estimate for the $O_1$ term can be written as
\begin{align*}
 O_1  &\leq \abs{\sum\limits_{j=1}^{N} f'(u)\average{\eta}\jump{\xi^u}_{j+\frac{1}{2}} + ((f'(u)-f'(u_{j}))\eta^u,\xi^u_x)_{I_j}}\\
      &\leq C_*h\norm{\eta}_{\infty}\norm{{\xi^u}}_{L^2(\partial{I})} + C_*h\norm{\eta^u}_{L^2(I)}\norm{\xi^u_x}_{L^2(I)}
    \leq C_*h^{k+1}\norm{\xi^u}_{L^2(I)}.
\end{align*}
Owing to the property of projection, notice that $(f'(u_j)\eta^u,\xi^u_x)_{I_j} = 0$, this is the reason for the first inequality. Because of the estimate $f'(u)-f'(u_{j}) = \mathcal{O}(h)$, we obtain the second inequality. And then the inverse inequalities \eqref{eqn:inverse inequality} are applied for deriving the final result.

\item $O_2$ term:
After a simple integration by parts, the $O_2$ term becomes
\begin{align*}
O_2 = \sum\limits_{j=1}^{N} \frac{1}{2}(f'(u)_x\xi^u, \xi^u)_{I_j} + f'(u)(\xi^{ref}-\average{\xi^u} )\jump{\xi^u}_{j+\frac{1}{2}}
    \triangleq O^1_2 + O^2_2.
\end{align*}
The first term $O^1_2$ can be easily controlled by $C_*\norm{\xi^u}^2_{L^2(I)}$. We mainly focus on the second term $O^2_2$. Due to the classification of the sign of $f'(u)$, it needs to be discussed separately. For the case $f'(u) > 0$ on $I_j \cup I_{j+1}$,
we have the negative term $-\frac{1}{2} f'(u)\jump{\xi^u}^2_{j+\frac{1}{2}}$ in $O^2_2$, which can be ignored. For the case $f'(u) \leq 0$ on $I_j \cup I_{j+1}$, the second term $O^2_2$ is simplified as

\begin{align*}
 f'(u)( \xi^{+} - \average{\xi^u} )\jump{\xi^u}_{j+\frac{1}{2}}
            & = f'(u)( \xi^{+} - \xi^{u,+} + \frac{1}{2}{\jump{\xi^u}} )\jump{\xi^u}_{j+\frac{1}{2}} \\
            & = \frac{1}{2}f'(u)(\jump{\xi^u} + (\xi^{+} - \xi^{u,+}))^2 - \frac{1}{2}f'(u)(\xi^{+} - \xi^{u,+})^2\\
            & {\leq - \frac{1}{2}f'(u)(\eta^{+} - \eta^{u,+})^2 \leq C(\norm{\eta^u}^2_{L^2(\partial I_j)} + \norm{\eta}^2_{L^2(\partial I_j)})}.
\end{align*}
Because of the negative term $f'(u)(\jump{\xi^u} + (\xi^{+} - \xi^{u,+}))^2$, we can omit it directly. The fact $\xi - \xi^{u} = (\mathcal{P}^+u - \mathcal{P}^-u) = (\eta- \eta^{u})$ derives the final estimate in the above inequality.

For the last case, i.e. $f'(u)$ change its sign, the term $f'(u)$ can provide an extra $h$, and then the situation becomes
\begin{align*}
\sum\limits_{j=1}^{N} f'(u)( \average{\xi} - \average{\xi^u} )\jump{\xi^u}_{j+\frac{1}{2}}
            \leq Ch({\norm{\eta^u}_{\infty} + \norm{\eta}_{\infty}})\norm{\xi^u}_{L^2(\partial I)} \leq Ch^{2k+2} + C\norm{\xi^u}^2_{L^2(I)}
\end{align*}
Hence we conclude that for $f(u)$ with all non-negative derivative $f'(u)$, there is
\begin{align*}
O_2 \leq Ch^{2k+2} + C\norm{\xi^u}^2_{L^2(I)}
\end{align*}
However, if there exist some points satisfying $f'(u) < 0$, then we have
\begin{align*}
O_2 \leq C( \norm{\xi^u}^2_{L^2(I)} + \norm{\eta^u}^2_{L^2(\partial I)} + \norm{\eta}^2_{L^2(\partial I)} )\leq C( \norm{\xi^u}^2_{L^2(I)} + Ch^{2k+1})  .
\end{align*}
Therefore, for the specific form $f(u) = \frac{1}{p}u^p$, the estimate for term $O_2$ is finished.

\item $O_3$ term: With the definition $e = u-u_h$,
\begin{align*}
O_3 
&\leq C_*\norm{e}_{\infty}(\norm{e}_{L^2(\partial I)}\norm{\xi^u}_{L^2(\partial I)} + \norm{e}_{L^2{(I)}}\norm{\xi^u_x}_{L^2(\partial I)})\\
&\leq  C_*\norm{e}_{\infty}(\norm{\xi^u-\eta^u}_{L^2(\partial I)}\norm{\xi^u}_{L^2(\partial I)} + \norm{\xi^u-\eta^u}_{L^2{(I)}}\norm{\xi^u_x}_{L^2{(I)}})\\
& \leq C_*h^{-1}\norm{e}_{\infty}(\norm{\xi^u}_{L^2(I)} + Ch^{2k+2})\\
& { \leq C_*h^{-1}\left(\norm{\xi^u}_{\infty} + \norm{\eta^u}_{\infty}\right)(\norm{\xi^u}^2_{L^2(I)} + Ch^{2k+2})}\\
& {\leq Ch^{-\frac{3}{2}}\norm{\xi^u}^3_{L^2(I)} +C\norm{\xi^u}^2_{L^2(I)} +Ch^{2k+2}} ,
\end{align*}
where the last inequality requires small enough $h$ and $k \geq 1$.
\end{itemize}

Subsequently, we divide the term $\mathcal{T}_2$ into $O_4,  O_5$
\begin{align*}
 \mathcal{T}_2 &= \sum\limits_{j=1}^{N}(f(u_h^{ref})- f(u_h^*))\jump{\xi^u}_{j+\frac{1}{2}} + \sum\limits_{j=1}^{N} (f(u_h^*) - \widehat{f(u_h)})\jump{\xi^u}_{j+\frac{1}{2}} \triangleq O_4 + O_5
\end{align*}
where $u_h^*$ depends on the sign of $f'(\cdot)$ between $(u^+_h)_{j+\frac{1}{2}}$ and $(u^-_h)_{j+\frac{1}{2}}$,
\begin{align*}
u_h^*|_{j+\frac{1}{2}}  = \begin{cases} (u_h^+)_{j+\frac{1}{2}}, \ &\text{if} \ f'(\cdot) < 0 \ \text{between}\  (u_h^+)_{j+\frac{1}{2}} \ \text{and}\ (u_h^-)_{j+\frac{1}{2}} \\
                      (u_h^-)_{j+\frac{1}{2}},\ & \text{if}\ f'(\cdot) > 0 \ \text{between}\  (u_h^+)_{j+\frac{1}{2}} \ \text{and}\ (u_h^-)_{j+\frac{1}{2}} \\
                      \average{u_h}_{j+\frac{1}{2}}, \ & \text{otherwise}
        \end{cases}. \tag{A.3.3}
\end{align*}

\begin{itemize}
\item $O_4$ term:
Compared with the classifications in (2.25) and (A.3.3), if the signs of derivatives $f'(\cdot)$ are consistent, then the term $O_4$ vanishes. Otherwise, there must be a zero point of $f'(\cdot)$ in the interval covered by the $(u^+_h)_{j+\frac{1}{2}}$,  $(u^-_h)_{j+\frac{1}{2}}$ and the exact solution $u$ in $I_j \cup \I_{j+1}$. Taking $u_{j+\frac{1}{2}}$ as a start point, we can obtain an upper bound for the length of this interval as $Ch + \norm{e}_{\infty}$. Hence, we have $f'(u_h^{ref}) \leq C_*(Ch + \norm{e}_{\infty})$. Then the following estimate holds
\begin{align*}
O_4
&\leq \sum\limits_{j=1}^{N} \Big(\abs{f'(u_h^{ref})} \abs{u_h^{ref}- u_h^*} + C_*(u_h^{ref}- u_h^*)^2\Big)\jump{\xi^u}_{j+\frac{1}{2}}\notag \\
&\leq \Big(\max_{j=[1,N]}\abs{f'(u_h^{ref})}\norm{u_h^{ref}- u_h^*}_{L^2(\partial I)} + C_*\norm{u_h^{ref}- u_h^*}^2_{L^2(\partial I)}\Big)\norm{\xi^u}^2_{L^2(\partial I)} \notag  \\
&\leq C_*(Ch + \norm{e}_{\infty})\norm{\xi^u-\eta^u}_{L^2(\partial I)}\norm{\xi^u}^2_{L^2(\partial I)}
\leq C_*(1 + h^{-1}\norm{e}_{\infty}) (\norm{\xi^u}^2_{L^2(I)} + Ch^{2k+2})\\
&{\leq Ch^{-\frac{3}{2}}\norm{\xi^u}^3_{L^2(I)} +C\norm{\xi^u}^2_{L^2(I)} + Ch^{2k+2}},
\end{align*}
where we used $\abs{u_h^{ref}- u_h^*} \leq \abs{\jump{u_h}} = \abs{\jump{e}} $ in the third inequality.
\item  $O_5$ term:
Due to the properties of the upwind flux (2.5), the $O_5$ term will degenerate to zero if the sign of $f'(u)$ does not change between $(u^+_h)_{j+\frac{1}{2}}$ and $(u^-_h)_{j+\frac{1}{2}}$. For the case with the changed sign, we can derive the following estimate form Taylor expansion
\begin{align*}
O_5 &= \sum\limits_{j=1}^{N} (f(\average{u_h}) - f(\zeta))\jump{\xi^u}_{j+\frac{1}{2}} = \sum\limits_{j=1}^{N}(f'(\average{u_h})(\average{u_h}-\zeta) -\frac{1}{2}f''_u(\average{u_h}-\zeta)^2)\jump{\xi^u}_{j+\frac{1}{2}}\\
& \leq \sum\limits_{j=1}^{N}(Ch\abs{\jump{u_h}} + C_*\abs{\jump{u_h}}^2)\jump{\xi^u}_{j+\frac{1}{2}}\leq C_*(1 + h^{-1}\norm{e}_{\infty}) (\norm{\xi^u}^2_{L^2(I)} + Ch^{2k+2})\\
&{ \leq Ch^{-\frac{3}{2}}\norm{\xi^u}^3_{L^2(I)}+C\norm{\xi^u}^2_{L^2(I)} + Ch^{2k+2}},
\end{align*}
where $\zeta$ is between $(u^+_h)_{j+\frac{1}{2}}$ and $(u^-_h)_{j+\frac{1}{2}}$. The derivation of the
above inequalities has used $\abs{\jump{u_h}} = \abs{\jump{u_h- u}} = \abs{\jump{e}}$ and the Young's inequality.
\end{itemize}
We have now accomplished the proof for Lemma \ref{lemma:H_estimate2}.
{
\begin{remark}
 In \cite{Zhang2004_SIAM,Zhang2010_SIAM}, the authors have proved the optimal orders of accuracy for upwind fluxes in hyperbolic conservation laws. The differences between our cases and theirs are directions of the test functions, which cause suboptimal order of accuracy for even $p$ in the above term $O_2$.
\end{remark}}

\subsection{Proof of Lemma \ref{lemma:hc1}}
\label{proofhc1}

First, along the same line of Lemma \ref{lemma:H_estimate2}, we can divide the $\mathcal{T}_1$ into three analogs $O_1, O_2, O_3$ as $\mathcal{T}_1$ in (A.3.1), where the ``$ref$" is considered as an average $\average$. The estimation for the term $O_1$ shows order of $h^{2k}$ since
\begin{align*}
 O_1  \leq &\abs{\sum\limits_{j=1}^{N} f'(u)\average{{\eta}^u}\jump{{\xi}^u}_{j+\frac{1}{2}} +
  ((f'(u)-f'(u_{j})){\eta}^u,{\xi}^u_x)_{I_j}}\\
     &\leq C\norm{{\eta}^u}_{\infty}\norm{{\xi}^u}_{L^2(\partial{I})} + C_*h\norm{{\eta}^u}_{L^2(I)}\norm{{\xi}^u_x}_{L^2(I)}
       \leq Ch^{2k}+ C_*\norm{{\xi}^u}^2_{L^2(I)}
\end{align*}
After an integration by parts, the $O_2$ term can be simplified as $ \frac{1}{2}(f'(u)_x{\xi}^u, {\xi}^u) \leq C_*\norm{{\xi}^u}^2_{L^2(I)}$. As to $O_3$, it is identical with the proof of Lemma \ref{lemma:H_estimate2} which we omit here. Therefore, we have
\begin{align*}
\mathcal{T}_1 &\leq Ch^{2k}+ C_*\norm{{\xi}^u}^2_{L^2(I)} + C_*h^{-1}\norm{e}_{\infty}(\norm{\xi^u}^2_{L^2(I)} + Ch^{2k+2})\\
 & {\leq Ch^{2k}+ C_*\norm{{\xi}^u}^2_{L^2(I)} + Ch^{-\frac{3}{2}}\norm{\xi^u}^3_{L^2(I)} }
\end{align*}

Next, the remaining term $\mathcal{T}_2$  is nothing but a residual term of midpoint formula for numerical integration which can be obtained by Taylor expansion, thus
\begin{align*}
\mathcal{T}_2 &\leq \sum\limits_{j=1}^{N} C_*\jump{u_h}^2\jump{{\xi}^u}_{j+\frac{1}{2}} \leq  C_*\norm{e}_{\infty}\norm{u_h-u}_{L^2(\partial I)}\norm{{\xi}^u}_{L^2(\partial I)} \\
&\leq C_*h^{-1}\norm{e}_{\infty}(\norm{\xi^u}^2_{L^2(I)} + Ch^{2k+2}) {\leq Ch^{-\frac{3}{2}}\norm{\xi^u}^3_{L^2(I)} +C\norm{\xi^u}^2_{L^2(I)} + Ch^{2k+2}}
\end{align*}
Finally, combining the above estimates, Lemma \ref{lemma:hc1} is verified.
{
\begin{remark}
In \cite{Bona2013_MC}, the authors have given the result of the $k$-$th$ order of accuracy for the
nonlinear part $f(u) = \frac{1}{p}u^p, p \geq 2$. Here in our proof, we draw the same conclusion by another way to verify the above result.
\end{remark}}

{
\subsection{Proof of Lemma \ref{lemma}}
\label{prooflemma}
Referring to the technique in \cite{Liu2015_NM}, integration of \eqref{lemma:condition} gives
\begin{align*}
\norm{v(t,\cdot)}^2_{L^2(I)} \leq CG(t) \tag{A.7.1}
\end{align*}
where
\begin{align*}
G(t) = h^{2k+\tilde{\mu}} + \int_0^t \norm{v^u(\tau,\cdot)}^2_{L^2(I)} + h^{-\frac{3}{2}}\norm{v(\tau,\cdot)}^3_{L^2(I)}d\tau.
\end{align*}
Since the following relation
\begin{align*}
G'(t) \leq C(G(t) + h^{-\frac{3}{2}}G(t)^{\frac{3}{2}})
\end{align*}
holds, we obtain a significant estimate after integrating the above inequality
\begin{align*}
Q(G/G(0))<CT
\end{align*}
where
\begin{align*}
Q(y) = \int_1^y \frac{1}{\zeta + h^{-\frac{3}{2}}\sqrt{G(0)}\zeta^{\frac{3}{2}} }d\zeta = \int_1^y \frac{1}{\zeta + h^{k -\frac{3-\tilde{\mu}}{2}}\zeta^{\frac{3}{2}} }d\zeta.
\end{align*}
Note that the derivatives $Q'(y)$ is positive and bounded for all $y \geq 1, k \geq \frac{3-\tilde{\mu}}{2} $. There must exists $\tilde{C}$ such that
\begin{align*}
Q(\tilde{C})  = CT.
\end{align*}
By the monotonicity of $Q$, we conclude that $G\leq \tilde{C}G(0)$. With the aid of (A.7.1), the result of Lemma \ref{lemma} is obtained.
}

\subsection{Proof of Lemma \ref{lemma:another_energy2}}
\label{proofanother_energy2}
We take test functions in equations \eqref{scheme:FW2DG} and \eqref{scheme:FW1DG_t} as
\begin{align*}
\varphi = \phi = (u_h)_t \ \text{and}\ s_h,  \ \psi = \vartheta = (r_h)_t \ \text{and}\ p_h.
\end{align*}
Summed up these eight corresponding equalities over all intervals, it gives
\begin{align*}
&((u_h)_t,(u_h)_t)_{I} + ((r_h)_t,(r_h)_t)_{I} + (s_h,s_h)_{I} + (q_h,q_h)_{I}+ 2((u_h)_t,s_h)_{I} + 2((r_h)_t,p_h)_{I}  \\
&- \mathcal{L}^-((r_h)_t, (u_h)_t)- \mathcal{L}^+((u_h)_t, (r_h)_t) - \mathcal{L}^-((r_h)_t, s_h)- \mathcal{L}^+(s_h, (r_h)_t) \tag{A.6.1} \\
&- \mathcal{L}^-(p_h, (u_h)_t)- \mathcal{L}^+((u_h)_t, p_h) - \mathcal{L}^-(p_h, s_h)- \mathcal{L}^+(s_h, p_h) =-(u_h,(r_h)_t + p_h)_{I}
\end{align*}
By the properties of the operators $\mathcal{L}^+,\mathcal{L}^-$, we conclude that the
energy equality \eqref{eqn:another_energy2} can be established for the
dissipative scheme $\mathcal{D}2$. The same derivation works for the conservative scheme $\mathcal{C}2$ with the
operator $\mathcal{L}^c$ in (A.6.1) as well.

\subsection{Proof of Lemma \ref{lemma:b2}}
\label{proofb2}
The bilinear term  $\bar{\mathcal{B}}_j$ can be written as
\begin{align*}
&\bar{\mathcal{B}}_j(\xi^u- \eta^u, \xi^r- \eta^r, \xi^p- \eta^p,\xi^s- \eta^s ;  \bm{\xi^3} ,\bm{\xi^4},\bm{\xi^3},\bm{\xi^4},\xi^u)\notag \\
&=  \bar{\mathcal{B}}_j(\xi^u, \xi^r, \xi^p,\xi^s;   \bm{\xi^3} ,\bm{\xi^4},\bm{\xi^3},\bm{\xi^4},\xi^u)- \bar{\mathcal{B}}_j(\eta^u, \eta^r, \eta^p,\eta^s;  \bm{\xi^3} ,\bm{\xi^4},\bm{\xi^3},\bm{\xi^4},\xi^u).\notag
\end{align*}
\begin{itemize}
\item For \eqref{eqn:B_estimate22d} of the dissipative scheme $\mathcal{D}2$:
By the same argument as that used for the $L^2$ dissipation and energy equality \eqref{eqn:another_energy2}, the first term of the above equation is
\begin{align*}
& \bar{\mathcal{B}}_j(\xi^u, \xi^r, \xi^p,\xi^s;  \bm{\xi^3} ,\bm{\xi^4},\bm{\xi^3},\bm{\xi^4},\xi^u) = (\xi^u_t,\xi^u)_{I_j} + (\xi^u, \xi^p + \xi^r_t)_{I_j} \\
&+ \norm{\xi^s + \xi^u_t }^2_{L^2(I_j)} + \norm{\xi^p + \xi^r_t}^2_{L^2(I_j)} + \frac{1}{2}(\jump{\xi^p} + \jump{\xi^r_t})^2_{j+\frac{1}{2}} + \frac{1}{2}(\jump{\xi^s} + \jump{\xi^u_t})^2_{j+\frac{1}{2}}.
\end{align*}
Similarly, with the properties of Gauss-Radau projections, the second term of $\bar{\mathcal{B}}_j$  can be obtained straightforwardly.
\item For \eqref{eqn:B_estimate22c} of conservative scheme $\mathcal{C}2$:
The $L^2$ conservation and energy equality \eqref{eqn:another_energy2} implies
\begin{align*}
\bar{\mathcal{B}}_j({\xi}^u, {\xi}^r,& {\xi}^p,{\xi}^s;   \bm{\xi^3} ,\bm{\xi^4},\bm{\xi^3},\bm{\xi^4},{\xi}^u)\\
&= ({\xi}_t^u,{\xi}^u)_{I_j}  + \norm{{\xi}^s + {\xi}^u_t}^2_{L^2(I_j)} + \norm{{\xi}^p + {\xi}^r_t}^2_{L^2(I_j)} +({\xi}^u, {\xi}^p + {\xi}^r_t)_{I_j}.
\end{align*}
Owing to the $L^2$ projection, all integral terms in the second term of $\bar{\mathcal{B}}_j$ vanish and we
are left with the remaining boundary terms.
\end{itemize}

\end{appendix}

\end{document}